\documentclass[preprint,11pt]{elsarticle}
\usepackage[T1]{fontenc}

\usepackage[letterpaper,left=1.91cm,right=1.91cm,top=1.91cm,bottom=2cm,]{geometry}
\usepackage[pdftex,colorlinks,linkcolor={blue},citecolor={blue}]{hyperref}
\usepackage{float}
\usepackage{amsmath}
\usepackage{amsthm}
\usepackage{amssymb}
\usepackage{amsfonts}
\usepackage{graphicx}
\usepackage{epstopdf}
\usepackage{caption}
\usepackage{footnote}
\usepackage{subfigure}
\usepackage{cases}
\usepackage{url}

\usepackage{marvosym}
\usepackage{standalone}
\usepackage{pgfplots}
\pgfplotsset{compat=newest}
\usepackage{grffile}
\usepackage{nomencl}
\makenomenclature

\usepackage{etoolbox}
\renewcommand\nomgroup[1]{%
  \item[\bfseries
  \ifstrequal{#1}{A}{Decision variables of the platform}{%
  \ifstrequal{#1}{B}{Decision variables of the city planner}{%
  \ifstrequal{#1}{C}{Endogenous variables}{%
  \ifstrequal{#1}{D}{Exogenous parameters}{}}}}%
]}

\usepackage[ruled,vlined]{algorithm2e}
\SetArgSty{textnormal}

\newcommand{\e}{\text{e}}
\newcommand{\HKD}{\text{HK}\$}
\newcommand{\hr}{\text{hr}}

\newtheorem{lemma}{Lemma}

\newtheorem{proposition}{Proposition}
\newtheorem{remark}{Remark}

\begin{document}
\begin{frontmatter}

\title{On-Demand Valet Charging for Electric Vehicles: Economic Equilibrium, Infrastructure Planning and Regulatory Incentives}

\author[hkust_address]{Zhijie Lai}
\ead{zlaiaa@connect.ust.hk}
\author[hkust_address]{Sen Li}
\ead{cesli@ust.hk}
\address[hkust_address]{Department of Civil and Environmental Engineering, The Hong Kong University of Science and Technology, \\Clear Water Bay, Kowloon, Hong Kong}

\begin{abstract}
Many city residents cannot install their private electric vehicle (EV) chargers due to the lack of dedicated parking spaces or insufficient grid capacity. This presents a significant barrier towards large-scale EV adoption. To address this concern, this paper considers a novel business model, {\em on-demand valet charging}, that unlocks the potential of under-utilized public charging infrastructure to promise higher EV penetration. In the proposed model, a platform recruits a fleet of couriers that shuttle between customers and public charging stations to provide on-demand valet charging services to EV owners at an affordable price. Couriers are dispatched to pick up low-battery EVs from customers, deliver the EVs to charging stations, plug them in, and then return the fully-charged EVs to customers. To depict the proposed business model, we develop a queuing network to represent the stochastic matching dynamics, and further formulate an economic equilibrium model to capture the incentives of couriers, customers as well as the platform. These models are used to examine how charging infrastructure planning and regulatory intervention will affect the market outcome. First, we find that the optimal charging station densities for distinct stakeholders are different: couriers prefer a lower density; the platform prefers a higher density; while the density in-between leads to the highest EV penetration as it balances the time traveling to and queuing at charging stations. Second, we evaluate a regulatory policy that imposes a tax on the platform and invests the tax revenue in public charging infrastructure. Numerical results suggest that this regulation can suppress the platform's market power associated with monopoly pricing, increase social welfare, and facilitate the market expansion. 
\end{abstract}


\nomenclature[A, 01]{$p_v$}{The price of valet charging services}
\nomenclature[A, 02]{$w$}{The average hourly wage for couriers}
\nomenclature[B, 01]{$K$}{The number of charging stations}
\nomenclature[B, 02]{$p_t$}{The tax imposed on each valet charging service}
\nomenclature[C, 01]{$\lambda$}{Hourly arrival rate of valet charging customers}
\nomenclature[C, 02]{$\lambda_d$}{Hourly arrival rate of delivery requests in the delivery queue}
\nomenclature[C, 03]{$\bar{\lambda}$}{Hourly arrival rate of EVs in each charging station}
\nomenclature[C, 04]{$N$}{The number of couriers working for the platform}
\nomenclature[C, 05]{$N_i$}{The number of idle couriers}
\nomenclature[C, 06]{$M$}{The total public charger supply}
\nomenclature[C, 07]{$S$}{The number of chargers at each charging station}
\nomenclature[C, 08]{$\pi$}{The per-time payment to couriers}
\nomenclature[C, 08]{$p_c$}{The average per-delivery payment to couriers}
\nomenclature[C, 09]{$t_r$}{The average response time in the delivery queue}
\nomenclature[C, 10]{$t_p$}{The average pickup time in the delivery queue}
\nomenclature[C, 11]{$t_d$}{The average delivery time in the delivery queue}
\nomenclature[C, 12]{$t_w$}{The average waiting time in the charging queue}
\nomenclature[C, 13]{$c_v$}{The generalized cost of using valet charging}
\nomenclature[C, 14]{$c_{ev}$}{The generalized cost of charging electric vehicles}
\nomenclature[C, 15]{$\rho_d$}{The utilization rate of courier fleet.}
\nomenclature[C, 16]{$\rho_c$}{The occupancy rate of public chargers.}
\nomenclature[D, 01]{$\lambda_0$}{The total number of private vehicle owners}
\nomenclature[D, 02]{$N_0$}{The total number of potential couriers}
\nomenclature[D, 03]{$M_0$}{The nominal public charger supply}
\nomenclature[D, 04]{$A$}{The total area of the city}
\nomenclature[D, 05]{$C$}{The average charging time in the charging queue}
\nomenclature[D, 06]{$c_s$}{The generalized cost of using self-charging}
\nomenclature[D, 07]{$c_0$}{The generalized cost of refueling internal combustion engine vehicles}
\nomenclature[D, 08]{$w_0$}{The wage of outside option in the supply model}
\nomenclature[D, 09]{$t_c$}{The average charging time in the charging queue}
\nomenclature[D, 10]{$r$}{The prorated installation cost of each public charger}
\nomenclature[D, 11]{$\alpha$}{The customer's value of time assigned to the response time and pickup time}
\nomenclature[D, 12]{$\beta$}{The customer's value of time assigned to the waiting time and delivery time}
\nomenclature[D, 13]{$\tau$}{The percentage of EVs needed to be charged per unit time}
\nomenclature[D, 14]{$\theta$}{The parameter in the delivery time model}
\nomenclature[D, 15]{$\phi$}{The parameter in the pickup time model}
\nomenclature[D, 16]{$\epsilon_1$}{The sensitivity parameter of the upper nest in the demand model}
\nomenclature[D, 17]{$\epsilon_2$}{The sensitivity parameter of the lower nest in the demand model}
\nomenclature[D, 18]{$\eta$}{The sensitivity parameter in the supply model}

\end{frontmatter}

\printnomenclature

\section{Introduction} \label{section:introduction}
Electric vehicles (EVs) have significantly reshaped the landscape of urban sustainability, leading to higher energy efficiency and lower greenhouse gas emissions as opposed to their gasoline-powered counterparts. Many governments are devoted to accelerating the shift towards electrified urban mobility, announcing to ban the sale of fossil fuel vehicles in a near future \cite{reuters_france_2019}\cite{guardian_uk_2020}. With heavy subsidies from governments and awakening public awareness of environmental protection, people become more receptive to EVs. According to the International Energy Agency, the global stock of electric vehicles continues to surge and has reached 7.1 million as of 2019 \cite{iea_global_2020}. 

However, the era of e-mobility is yet to come.  The commercial success of EVs is on the premise of an ecosystem with ubiquitous charging infrastructure, which consists of residential chargers and workplace chargers for a long dwell-time recharge, and public chargers for occasional top-up \cite{levy_costs}. The major barrier towards mass adoption of EVs is the charging inconvenience arising from a few aspects. First, private charging infrastructure is not accessible to a large number of potential EV drivers. In dense cities where apartment dwelling is prevalent, installing a residential charger is not possible for plenty of people who do not have dedicated parking spaces. Second, for EV owners that solely rely on public charging stations, they have to drive miles away to plug in their EVs, wait for hours until the recharge is completed, and then fetch the vehicles with perfect timing to avoid over-stay penalty, which is inconvenient for those who live out of walking distance to a public charging station \cite{nyt_charger_2020}. Third, workplace charging may be avoid the unnecessary wait of EV owners during charging, but it may be infeasible for those who live in densely populated cities where parking is severely limited and prohibitively expensive\footnote{Take Hong Kong as an example, many commercial buildings in Hong Kong are high-rises with insufficient parking spaces and expensive parking fare \cite{EBHK}, which refrains most residents from driving to work. In fact, over 90\% of passenger trips are delivered by public transport system every day \cite{PTSS}.}. As a result, public charging infrastructure is limited to long-distance travel recharge or occasional battery top-up, leading to a lower utilization rate than expected. For instance, China locates the world’s largest charging station network, while its public charging stations are idle for 85\% of the time \cite{caixin_global_chinas_2018}. In the United States, over 90\% of the EV owners recharge their vehicles at home or workplace, other than at public charging stations \cite{INL_plugged_2015}. 

To address the aforementioned challenges, a large body of literature has investigated the deployment and operation of charging infrastructure to reduce the charging inconvenience. With an objective to maximize facility utilization, Xi et al. \cite{xi_simulation-optimization_2013} tackled charging stations siting problem in a simulation-optimization framework and asserted that the optimal locations are sensitive to the specific optimization criterion. Yang et al. \cite{yang_data-driven_2017} formulated the charger allocation problem as an integer linear programming and captured the charging dynamics using a queuing model. Interestingly, they showed that charger utilization will increase when extra waiting spots are provided. Huang and Kockelman \cite{huang_electric_2020} developed a network equilibrium model to identify the profit-maximizing charging stations placement considering the endogenously determined travel time and on-site charging congestion. Gan et al. \cite{gan_fast-charging_2020} introduced both spatial and temporal penalties to characterize the elastic demand in deploying fast charging stations. Mak et al. \cite{mak_infrastructure_2013} established a robust optimization framework to address the planning of battery-swapping stations. Sarker et al. \cite{sarker_optimal_2015} incorporated day-ahead scheduling, battery demand uncertainty, and electricity price uncertainty into the optimization model to address the operation and scheduling of battery-swapping stations. Widrick et al. \cite{widrick_optimal_2016} integrated the vehicle-to-grid (V2G) technique with battery-swapping stations and explored the optimal charging and discharging policy that maximize the total profit. Please refer to \cite{shen_optimization_2019} for a a comprehensive literature review.

Distinct from the above studies that focus on charging infrastructure planning and operation, a concept of Charging-as-a-Service (CaaS) has emerged recently. Guo et al. \cite{guo_caas_2020} proposed a business model that offers on-demand battery delivery services to electrify the Mobility-as-a-Service vehicles (e.g., ride-sourcing vehicles, taxis). They found that only 250 service vehicles can satisfy the battery delivery demand of 13000 electric taxis in NYC with a five-minute average waiting time. Zhang et al. \cite{zhang_mobile_2020} explored CaaS in a mobile manner, where mobile chargers offer on-site plug-in services.  They show that when reservation is introduced, accurate estimations on charging demand can be achieved, and charging demand across the network can be efficiently and effectively satisfied with the support of intelligent system-level decisions. Qiu and Du \cite{qiu_charging_2021} examined the dispatch and routing problem in providing EV-to-EV charging services, which allows a pair of EVs to cooperate with each other and exchange electricity on the move. They showed that the CaaS platform performs better in scenarios of low EV penetration. 

Although the above studies can facilitate transport electrification, few of them have fundamentally addressed the inconvenience of using public charging stations for long dwell-time EV recharge. In contrast, this paper aims to break down the barrier to using public charging stations as a daily charging option. Such charging style can eliminate the range anxiety for city dwellers who do not have access to private residential chargers. This will bypass the difficulty of installing private charging stations, unleash the immense but under-utilized potential in public charging infrastructure, and promote mass EV adoption among metropolitan residents.

To realize this vision, we propose a novel business model that provides on-demand valet charging services to EV owners, enabling them to conveniently utilize public charging stations for daily charging. The business model involves three entities: EV owners, couriers, and a valet charging platform. EV owners can request this service whenever needed and drop off their keys wherever convenient. Couriers are then dispatched to pick up the low-battery EVs from the customers, drive the EVs to charging stations, plug in the vehicles, and return the fully-charged EVs to customers. Couriers and EV owners are connected by a third-party platform that offers a user interface to EV owners and employs a fleet of couriers to provide on-demand valet charging services at an affordable price. With the deliver-to-door valet charging services, EV drivers are free from traveling back and forth to charging stations, making public charging as convenient as residential charging. As such, it offers a convenient charging option and enables millions of apartment dwellers without private chargers to embrace EVs. Since 2016, Luxe has partnered with Tesla and integrated charging service into valet parking \cite{luxe_demand}. However, it is offered as a luxury service exclusively for Tesla EVs. In contrast, our model is targeted at the entire EV community, which achieves economy of scale that can reduce the operational cost and improve the service quality. The key contributions of this paper are three-fold:
\begin{itemize}
    \item We propose the novel business model of {\em on-demand valet charging} that enables city residents without private chargers to adopt EVs and recharge EVs using public charging stations with minimal inconvenience. We formulate a mathematical model to characterize the interaction among stakeholders involved in the valet charging market. The model includes a queuing network that describes the matching dynamics among EVs, couriers, and charging stations, and an economic equilibrium model that captures the incentives of multiple market participants. Based on the proposed mathematical model, we find that the profit-maximizing optimum in the valet charging market always fall into the normal regime instead of the inefficient {\em wild goose chase} (WGC) regime \cite{castillo_surge_2017}.

    \item We investigate charging infrastructure planning under the proposed business model. We find that the interests of different stakeholders are not consistent: couriers receive higher surplus at a lower charging station density, whereas the platform prefers a higher density such that it enjoys a higher markup from the lower marginal cost. The optimal density that leads to the highest EV penetration is in-between where the time traveling to and waiting at charging stations are traded off. We also show that the monopoly platform has a strong market power that reaps most of the benefits associated with improved charging convenience, resulting in inefficient market outcomes.
    
    \item We explore a prospective regulatory policy that levies a tax on the platform for serving each customer and invests the tax revenue in public charging infrastructure. We show that the tax burden is primarily undertaken by the platform as opposed to EV owners. When the charger installation cost is in a certain range, imposing this tax can weaken the platform's market power, improve social welfare, and further facilitate EV adoption. Nonetheless, aside from regulation, the city planner needs to jointly redeploy charging stations to effectively bring more EVs into use.
\end{itemize}

The remainder of this paper is organized as follows. \autoref{section:preliminaries} provides a brief sketch of the proposed business model and some fundamental remarks. \autoref{section:formulation} formulates a mathematical model for the valet charging market. \autoref{section:infrastructure_planning} investigates how charging infrastructure planning will affect the market outcome. \autoref{section:taxation} introduces and verifies a potential regulatory policy for the proposed business model. \autoref{section:sensitivity_analysis} presents a sensitivity analysis to examine the robustness of economic insights. Finally, \autoref{section:conlusion} concludes key findings of this study and discusses possible directions for future work.

\section{Preliminaries} \label{section:preliminaries}

This section elucidates the business model of valet charging and elaborates the interaction among the platform, EV owners, couriers, and charging infrastructure. The platform dispatches a fleet of couriers to provide app-based {\em on-demand} chauffeur charging services for EV owners. Customers order valet charging services via the user interface, specify the pickup location, then wait for an idle courier to pick up the vehicle. Drivers can register as couriers and offer EV delivery services contingent on their schedules. Couriers are paid on {\em per-delivery} basis. Resembling designated drivers,\footnote{The term ``designated driver'' refers to a driver who is dedicated to driving alcohol consumers home safe in their own cars \cite{ditter_effectiveness_2005}. This concept is intended to discourage drunk driving.  Designated drivers typically come and forth between customers' home and restaurants.} couriers typically shuttle between customers and charging stations by public transit, bicycle, or scooter. Upon receiving each request, the platform dispatches a vacant courier to pick up and drive the low-battery EV to a nearby charging station. The charging process will not be interrupted until the battery is fully recharged unless an early recall is issued. Upon charging completion, another vacant courier will be sent to drive the EV back to its owner. The platform charges a service fare from customers, pays wages to couriers and keeps the difference between service fare and courier wage as a commission to make a profit. \autoref{fig:business_model} demonstrates the interaction among different entities in the valet charging market. We have fundamental remarks as below.

\begin{figure}[ht]
    \centering
    \includegraphics[width=0.5\textwidth]{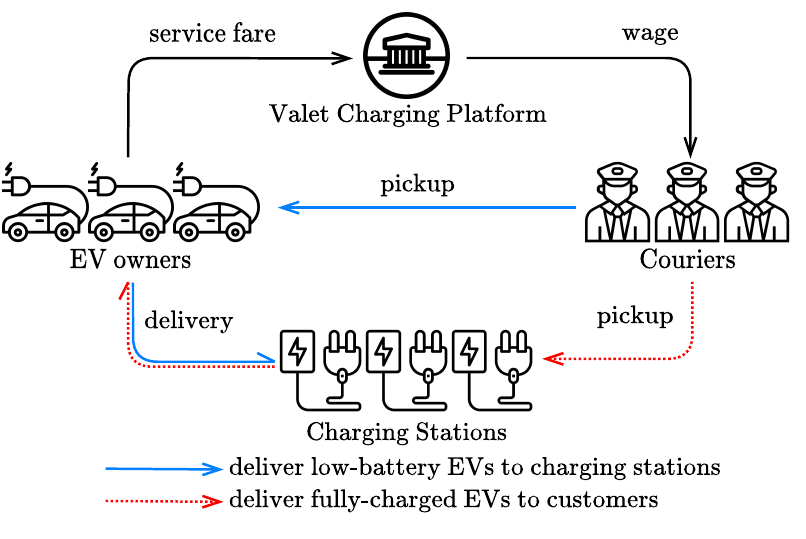}
    \caption{A schematic diagram for the business model of valet charging.}
    \label{fig:business_model}
\end{figure}

\begin{itemize}
    \item The remaining power of an EV that orders the valet charging service should be adequate for a trip to the nearest charging station. This is not a strong assumption. Due to  range anxiety, many EV drivers are inclined to recharge their cars proactively although there remains adequate battery power for another trip \cite{franke_understanding_2013}. If the remaining battery of an EV does not suffice to travel to the closest charging station, towing service is needed for the stranded vehicle, but this is beyond the service scope of valet charging.

    \item As highlighted in \autoref{fig:business_model}, there are two delivery trips in the process of valet charging services: first deliver low-battery EVs to charging stations and then return fully-charged EVs to customers. We argue that two delivery trips can be accomplished by different couriers. For one thing, it usually takes several hours to fully recharge an EV battery.\footnote{Prevailing Alternating Current (AC) Level 2 chargers supply power at 1.4 - 19.2kW, whereas the battery capacity of a standard-range Tesla Model 3 is 50kWh.} It is clearly a waste of time for couriers to wait hours in charging stations. For another, customers' arrival and departure are random. It is inefficient to designate the same courier to drive the fully-charged EV back to its owner. Instead, couriers should be flexibly dispatched to serve another delivery trip, which benefits both the platform and couriers: the platform will enjoy the higher labor efficiency as there are more dispatchable couriers in the system; couriers will idle a shorter time and receive a higher income since they are paid on per-delivery basis.
    
    \item EV charging may suffer from delay due to congestion at charging stations. We argue that when such delay occurs, couriers should be liberated from the queue and a dedicated coordinator should be employed to manage the waiting EVs at each charging station. This is well-reasoned and necessary: a coordinator is in need to settle the rotation between fully-charged EVs and queuing EVs and take care of the keys. Couriers are not appropriate candidates for these jobs because of the uncertainty and stochasticity in the customer arrival and departure. The advantages of this strategy are two-fold. First, couriers can be free from being stuck in line and therefore receive higher wages. Second, dedicated coordinators can promptly unplug the fully-charged EVs without delay and hand over chargers to the waiting vehicles, which can mitigate the ``overstay'' problem in charging stations \cite{zeng_solving_2020}.
    
    \item Public charging stations are typically operated by government agencies or third-party companies. For the valet charging customers, charging station operators charge an electricity fare, and parking lot operators charge a parking fee, while the platform independently charges a service fare. We consider the electricity fare and parking fee as exogenous and only focus on the pricing strategy of the valet charging platform. The case that the platform operates its own charging stations dedicated to valet charging services is beyond the scope of this paper and will be discussed in future research.
\end{itemize}

\section{Formulation} \label{section:formulation}

This section presents a queuing network to capture the matching dynamics among customers, couriers, and charging stations, followed by an economic equilibrium model to predict the market outcome.

\subsection{Queuing Network}

\begin{figure}[ht]
    \centering
    \includegraphics[width=0.8\textwidth]{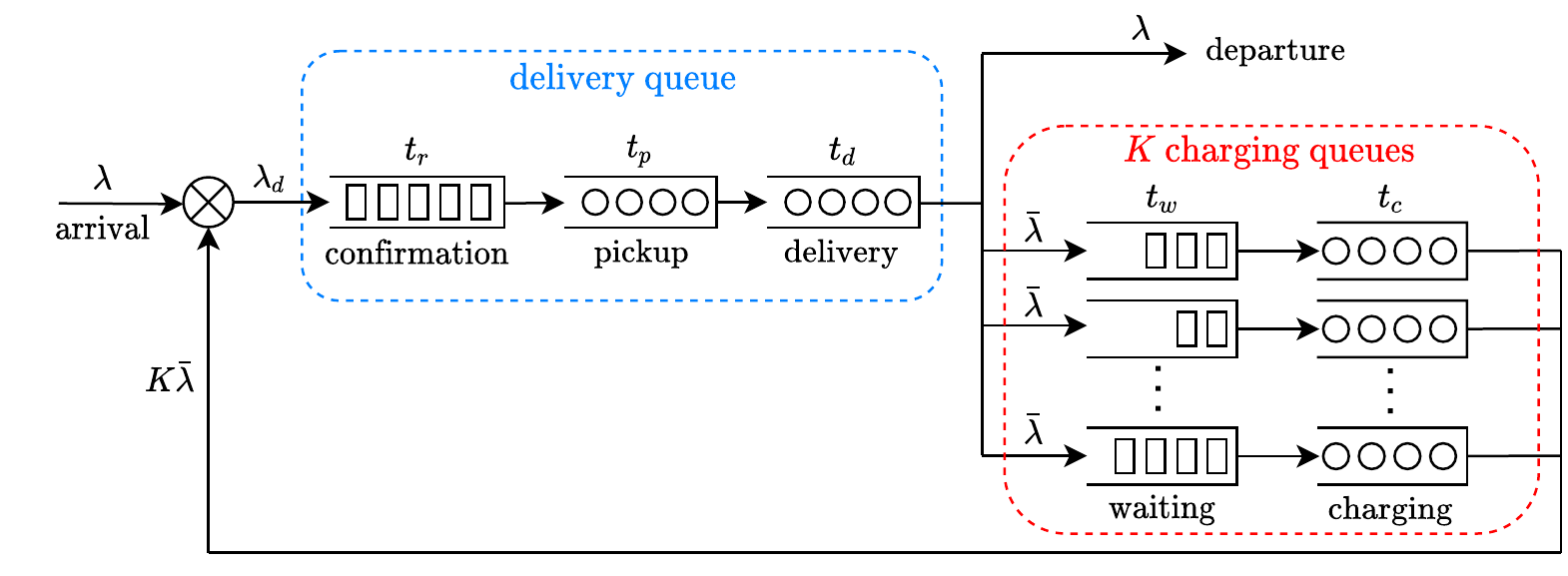}
    \caption{Queuing network of the valet charging market.}
    \label{fig:queuing_network}
\end{figure}
As shown in \autoref{fig:queuing_network}, the service process is characterized by a queuing network that consists of a delivery queue and multiple charging queues (one at each charging station). The delivery queue captures the matching process between EVs and couriers. From entering to leaving the queue, each EV sequentially experiences a response period (time for the platform to confirm the order and dispatch a courier, denoted by $t_r$), a pickup period (time for couriers to pick up the EVs, denoted by $t_p$), and a delivery period (time for couriers to deliver the vehicle, denoted by $t_d$). On the other hand, the charging queue captures the matching process between EVs and charging stations, where each EV experiences a waiting period (time for waiting in line for an available charger, denoted by $t_w$) and a charging period (time for an EV to be fully recharged, denoted as $t_c$) before it exits from this queue. Through the lifetime of valet charging service, the low-battery EV first enters the delivery queue and gets delivered to the nearest charging station. It then enters a charging queue and waits for an available outlet. After the charging is completed, it enters the delivery queue again and is returned to its owner. During the entire process, each EV enters the delivery queue twice, once from the EV owner's residence to the charging station, and once from the charging station back to the EV owner's residence. The service follows the first-come-first-served (FCFS) principle, and both queues are assumed to be infinite, i.e., there is no limitation on the number of permissible customers in each queue.

In the proposed queuing network model, we assume that the charging infrastructure is uniformly distributed across the city such that the waiting time at distinct charging stations have the same average value. We argue that an aggregate model that predicts the average market outcome suffices our purpose as we primarily focus on infrastructure planning and regulatory policy at the city level\footnote{The spatial heterogeneity of valet charging demand and charging infrastructure distribution is left for future research.}.

\subsubsection{Delivery Queue}
In the delivery queue, couriers act as ``servers'' and each delivery trip is defined as a ``job''. Denote $N$ as the number of couriers. Let $\lambda$ represent the arrival rate (per unit time) of new-coming valet charging customers, and denote $\lambda_d$ as the arrival rate (per unit time) of delivery request in the delivery queue. Since each EV will enter the delivery queue twice (each delivery request is initiated from either a low-battery EV or a fully-charged EV), we always have $\lambda_d = 2 \lambda$ in steady state. For each EV that enters the delivery queue, its dwell time consists of the response time $t_r$, the pickup time $t_p$, and the delivery time $t_d$ (see \autoref{fig:queuing_network}). These are endogenous variables that depend on the valet charging demand, courier supply and charging station density, which will be modeled below.

The response time $t_r$ can be derived based on the property of delivery queue. We assume that the jobs' arrival follows a Poisson process with rate $\lambda_d$, and the service time is exponentially distributed with mean $t_p + t_d$. This leads to a standard M/M/N queue, whose average waiting time is given in closed-form by Erlang C Formula \cite{hillier_introduction_2004}:
\begin{equation} \label{eqn:original_response_time}
    t_r = \frac{\rho_d Q_0 \big(\lambda_d (t_p + t_d )\big)^N}{N!(1-\rho_d)^2 \lambda_d},
\end{equation}
where $\rho_d = \lambda_d (t_p + t_d )/N < 1$ is the utilization rate of courier fleet, i.e., the expected fraction of time that each courier is occupied, and $Q_0$ represents the probability that there is no customer waiting in the queue:
\[Q_0 = \left[\sum_{n=0}^{N-1}\dfrac{\big(\lambda_d (t_p + t_d)\big)^n}{n!} + \dfrac{\big(\lambda_d (t_p + t_d)\big)^N}{N!}\cdot\dfrac{1}{1-\rho_d}\right]^{-1}.\]
Equation \eqref{eqn:original_response_time} is expressed explicitly. However, since summation and factorial terms are involved, the computation  is tedious and the formula is intractable for further analysis. In this regard, we adopt the formula in \cite{sakasegawa_approximation_1977} to approximate the mean waiting time in the delivery queue:
\begin{equation} \label{eqn:response_time}
    t_r = (t_p + t_d) \dfrac{\rho_d^{\sqrt{2N + 2} - 1}}{N(1 - \rho_d)} = \dfrac{1}{\lambda_d} \cdot \dfrac{\rho_d^{\sqrt{2N + 2}}}{1 - \rho_d}.
\end{equation}
We emphasize that this formula is a sufficiently exact approximation of the mean waiting time for an M/M/N queue. The performance of this approximation is evaluated in \ref{appendix:approximation_evalution}.

The pickup time $t_p$ depends on the availability of idle couriers. It captures the time elapsed from a delivery request being issued to the EV being picked up. This duration is primarily determined by the travel distance between the new-coming customers and her nearest idle courier. It is well-established that the pickup distance is inversely proportional to the square root of the density of idle couriers \cite{arnott_taxi_1996} \cite{li_regulating_2019}, thereby we have:
\begin{equation} \label{eqn:pickup_time}
    t_p = \dfrac{\phi}{\sqrt{N_i/A}},
\end{equation}
where $\phi$ is a model parameter, $A$ is the total area of the city, and $N_i$ denotes the number of idle couriers. This equation implies that the EV will be picked up more promptly with more dispatchable couriers per unit area. It follows the intuition that the distance between a customer and the closest idle courier depends on the average distance between any two nearby idle couriers (refer to \cite{arnott_taxi_1996} and \cite{li_regulating_2019} for more justification). Based on the Little's Law, the number of idle couriers is given by
\begin{equation} \label{eqn:idle_courier}
    N_i = N - \lambda_d t_p - \lambda_d t_d = N - 2 \lambda t_p -2 \lambda t_d,
\end{equation}
where the second term on the right-hand side of (\ref{eqn:idle_courier}) denotes the number of couriers on the way to pick up EVs, and the third term represents the number of couriers that are delivering EVs. Note that $N_i > 0$ is required to stabilize the queue. 

The delivery time $t_d$ depends on the number of charging stations per unit area. In valet charging services, the average delivery time $t_d$ equals the average travel time between the customer and the nearest charging station, which is inversely proportional to the square root of charging station density \cite{ahn_analytical_2015}:
\begin{equation} \label{eqn:delivery_time}
    t_d = \theta \sqrt{\dfrac{A}{K}},
\end{equation}
where $\theta$ is a model parameter and $K$ is the number of charging stations in the city. The rationale behind this is straightforward. Similar to the derivation of (\ref{eqn:pickup_time}), if the charging stations are located across the city according to a given distribution, the expected distance to the closet charging station is proportional to the average distance between any two nearby charging stations, which is further proportional to $\sqrt{A/K}$.

\subsubsection{Charging Queue} 
The charging queue at each charging station models the matching dynamics between EVs and chargers. After arriving at the charging station, each EV joins the charging queue and experiences a waiting time $t_w$ and a charging time $t_c$. We assume that the EVs' arrival is a Poisson process and the charging time has an exponential distribution with mean $t_c$. Note that $t_c$ is regarded as an exogenous parameters. Consider there are $K$ charging stations in total and each of them has $S$ charging outlets. This constitutes $K$ homogeneous M/M/S queues (one at each charging station), and the arrival rate of EVs to each queue is $\bar{\lambda} = \lambda/K$. Similar to \eqref{eqn:response_time}, the waiting time $t_w$ is also approximated by
\begin{equation} \label{eqn:waiting_time}
    t_w = t_c \dfrac{\rho_c^{\sqrt{2S + 2} - 1}}{S(1 - \rho_c)} = \dfrac{1}{\bar{\lambda}} \cdot \dfrac{\rho_c^{\sqrt{2S+2}}}{1 - \rho_c},
\end{equation}
where $\rho_c = \bar{\lambda} t_c/S < 1$ is the occupancy of chargers. The assessment of this approximation is also presented in \ref{appendix:approximation_evalution}.


\subsection{Market Equilibrium Model}

\subsubsection{Customer Incentives}

Customers decide whether to use valet charging services according to the generalized cost, which is the weighted sum of service fare and the total waiting time. The generalized cost $c_v$ is given by
\begin{equation} \label{eqn:valet_cost}
    c_v = p_v + \alpha (t_r + t_p) + \beta (2 t_d + t_w),
\end{equation}
where $p_v$ is the average price for each valet charging service, and $\alpha > \beta > 0$ represents the customer's value of time. The non-monetary component in \eqref{eqn:valet_cost} reflects the time cost spent within the service. In particular, from placing a request to receiving the returned vehicle, the elapsed time can be decomposed into the following segments: 
(a) waiting to be confirmed, i.e., $t_r$, (b) waiting to be picked up, i.e., $t_p$, (c) delivering to a charging station, i.e., $t_d$, (d) waiting at the charging station for an idle charger, i.e., $t_w$, (e) completing the charging process, i.e., $t_c$, (f) waiting to be confirmed again, i.e., $t_r$, (g) waiting to be picked up again, i.e., $t_p$, and (h) delivering back to the customer, i.e., $t_d$. Together, the customer waiting time before the first EV pickup is $t_r + t_p$, whereas the waiting time after the EV being taken away  is $2t_d + t_w + t_c + t_r + t_p$. Since $t_c$ is exogenous, without loss of generality, it can be normalized to zero in \eqref{eqn:valet_cost}.\footnote{Ev owners have to wait $t_c$ for the charging to be completed regardless of using valet charging or self-charging. In this sense, this term can be normalized to zero in the discrete choice model (which will be presented later), since it appears in the cost of valet charging as well as self-charging. Note that we only treat $t_c$ as $0$ in \eqref{eqn:valet_cost}, not in other equations.} In this case, the post-pickup waiting time can be simplified as $2t_d + t_w + t_r + t_p$. Crucially, customers perceive the pre-pickup waiting time and the post-pickup waiting time differently: they tend to be impatient when waiting for the courier's arrival. However, after the EV is picked up, customers become less sensitive to the time when the vehicle is returned. For this reason, we let $\alpha > \beta$. Note that since $t_r + t_p$ appears twice (once before pickup and once after pickup), we can combine them as a single  term in \eqref{eqn:valet_cost} for notation brevity.

The generalized cost $c_v$ determines how many EV owners will adopt valet charging. This relation can be captured in the following two steps:
\vspace{-0.3cm}
\begin{itemize}
    \item In the first step,  private vehicle owners decide whether to purchase EVs or fossil fuel vehicles based on the price differential and the charging cost of EVs. We treat price differentials as exogenous, and model the EV charging cost, $c_{ev}$, as the composite of valet charging cost $c_v$ and self-charging cost $c_s$.\footnote{We assume that EV owners charge their vehicles by either valet charging or self-charging.} We assume that $c_s$ is exogenous in this study. Obviously, $c_{ev}$ increases with $c_v$, and $c_{ev}$ degrades to $c_s$ in the absence of valet charging services, i.e.,
    \[\dfrac{\mathrm{d } c_{ev}}{\mathrm{d } c_v} > 0 \quad \text{and} \quad \lim_{c_v \rightarrow \infty} c_{ev} = c_s.\]
    The cost $c_{ev}$ determines the EV penetration rate $P_{ev}$ (the percentage of private car owners who choose EVs instead of fossil fuel vehicles), which is decreasing in $c_{ev}$ and therefore $c_v$.

    \item In the second step,  EV owners decide whether to charge their low-battery vehicles using valet charging services or by themselves. Denote $P_{vc}$ as the percentage of EV owners that use valet services. It should be a decreasing function of $c_v$, which indicates that valet charging customers will be deterred by a high price or long waiting time. 
\end{itemize}
Based on the above discussion, we can summarize the valet charging demand as a generic function $F_v(c_v)$:
\begin{equation} \label{eqn:F_v}
    \lambda = \tau \lambda_0 P_{ev} P_{vc} = \tau \lambda_0 F_v(c_v),
\end{equation}
where $\tau$ is the percentage of EVs that run out of power per unit time, $\lambda_0$ is the total number of potential customers owning or intending to purchase a private vehicle (either EV or conventional fossil fuel vehicle), and $F_v(c_v) = P_{ev} P_{vc}$ characterizes the portion of private vehicle owners that adopt EVs {\em and} use valet charging services. As such, $\lambda_0 F_v(c)$ refers to the total number of EV owners that use valet charging as a daily charging option. By assuming that a certain percentage $\tau$ of EVs will need to be recharged per unit time, we can express the arrival rate of valet charging customers as $\tau\lambda_0 F_v(c_v)$. Interestingly, by the definition of \eqref{eqn:F_v}, we notice that the EV penetration is maximized when the valet charging demand is maximized, because both $P_{ev}$ and $\lambda$ decrease with $c_v$. This property can be used to examine how public policies will affect the valet charging market, and further influence the EV penetration.

We emphasize that the generalized cost $c_v$ varies among customers. Each element in \eqref{eqn:valet_cost} is random and $c_v$ only represents the average value. Here we implicitly assume that the arrival rate of valet charging demand is uniquely determined by the average cost and does not depend on the higher-order moments of the cost distribution among customers. The nested-logit model \cite{train_discrete_2009} is a special case of \eqref{eqn:F_v} that satisfies this assumption.

\subsubsection{Courier Incentives}
Couriers will receive a per-time payment $\pi$ from the platform for providing valet charging services. In this case, the courier's average per-delivery payment $p_c$  can be derived as  $p_c=\pi x(N_i, K)$, where $x(N_i, K)$ is the average travel time for each delivery. This  travel time  is the sum of pickup time (\ref{eqn:pickup_time}) and delivery time (\ref{eqn:delivery_time}), thus depending on $N_i$ and $K$. Couriers  typically decide weather to join the platform based on the long-term expected earning instead of the per-trip income \cite{banerjee_pricing_2015}. Couriers' average wage $w$ is collectively decided by the delivery demand $\lambda_d$, the average per-delivery payment $p_c$, and the number of couriers $N$. It equals the total payment to all couriers divided by the number of couriers, i.e., $w = p_c \lambda_d / N= \pi x(N_i, K)\lambda_d / N$.
Therefore, the courier supply can be written as a function of $w$:
\begin{equation} \label{eqn:F_c}
    N = N_0 F_c(w),
\end{equation}
where $N_0$ is the number of overall potential couriers (i.e., drivers looking for a job), and $F_c(w)$ is the portion of drivers that enter the valet charging industry. $F_c(\cdot)$ is assumed to be strictly increasing, reflecting that more drivers would engage in the business if the expected income is higher. We emphasize that we do not assume each courier has the same wage. In general, different couriers may receive different earnings. To derive equation \eqref{eqn:F_c}, we only need to assume that the total number of couriers depends on the average wage instead of any higher-order moments of the wage distribution. The well-established logit model is a special case of \eqref{eqn:F_c}.

\subsubsection{Platform Decisions}

We consider a monopolistic platform that determines the service fare $p_v$ and the courier wage $w$ to maximize its profit, i.e., the difference between the gross revenue from customers ($\lambda p_v$) and the total salary expenditure. The latter consists of two components: (a) total payment to couriers, i.e., $N w$, (b) payment to dedicated coordinators, i.e., $KC$ (recall the ``dedicated coordinator'' strategy in \autoref{section:preliminaries}). Here $C$ is regarded as a constant that represents the per-time wage for the coordinator in each charging station. The platform's decisions are subject to the demand model \eqref{eqn:F_v}, the supply model \eqref{eqn:F_c}, and the queuing network model \eqref{eqn:response_time},\eqref{eqn:pickup_time},\eqref{eqn:delivery_time},\eqref{eqn:waiting_time}. For each exogenous $K$, the profit-maximizing problem is cast as follows:
\begin{equation} \hspace{-3cm} \label{eqn:valet_profit_maximization}
    \max_{p_v, \pi} \quad \lambda p_v - N w - K C 
\end{equation}\
\begin{subnumcases} {\label{eqn:valet_constraint}}
    \lambda = \tau\lambda_0 F_v \big(p_v + \alpha (t_r + t_p) + \beta (2 t_d + t_w) \big) \label{eqn:valet_constraint_demand}\\
    N = N_0 F_c\left(w\right) \label{eqn:valet_constraint_supply}\\
    t_r = \dfrac{\rho_d^{\sqrt{2N + 2}}}{2\lambda(1 - \rho_d)} \label{eqn:valet_constraint_response_time}\\
    t_p = \dfrac{\phi}{\sqrt{(N - 2 \lambda(t_p + t_d))/A}} \label{eqn:valet_constraint_pickup_time}\\
    t_w = \dfrac{K \rho_c^{\sqrt{2S+2}}}{\lambda(1 - \rho_c)}, \label{eqn:valet_constraint_waiting_time}
\end{subnumcases}
where \eqref{eqn:valet_constraint_demand} specifies the valet charging demand, \eqref{eqn:valet_constraint_supply} gives the courier supply, and \eqref{eqn:valet_constraint_response_time}-\eqref{eqn:valet_constraint_waiting_time} defines the endogenous response time, pickup time, and waiting time, respectively. The delivery time is exogenous and only depends on $K$. It is important to note that the queuing network implicitly imposes an upper bound on the number of customers:
\begin{equation} \label{eqn:demand_upper_bound}
    \lambda < \min \left\{KS/t_c \, , \, N / (2(t_p + t_d))\right\},
\end{equation}
which requires the arrival rate of valet charging customers should not exceed the service capacity of both the delivery queue and charging queue.

\subsubsection{Solution Properties}
To solve \eqref{eqn:valet_profit_maximization}, we insert \eqref{eqn:valet_constraint_response_time}-\eqref{eqn:valet_constraint_waiting_time} into \eqref{eqn:valet_constraint_demand}, plug (\ref{eqn:valet_constraint_demand}) and (\ref{eqn:valet_constraint_supply}) into the objective function, and change the decision variables to $\lambda$ and $N$, which transforms \eqref{eqn:valet_profit_maximization} into an unconstrained optimization problem as follows:
\begin{equation} \label{eqn:valet_profit_maximization_unconstrained}
    \max_{\lambda, N} \quad \lambda \left( c_v(\lambda) - \alpha(t_r+t_p) - \beta(2 t_d + t_w)\right ) - N w(N) - KC,
\end{equation}
in which $c_v(\lambda)$ denotes the inverse demand model, $w(N)$ denotes the inverse supply model, $t_r$, $t_p$ and $t_w$ are functions of $\lambda$ and $N$. Note that $t_p$ is a function of $N_i$ and it may have multiple solutions even when $\lambda$ and $N$ are given. To see this, we consider the market outcome under each pair of $\lambda$ and $N$. Given the demand and supply, we can combine \eqref{eqn:pickup_time} and \eqref{eqn:idle_courier} to derive the number of idle couriers as:
\begin{equation} \label{eqn:fixed_point_equation}
    N_i = N - 2 \lambda \left( \dfrac{\phi}{\sqrt{N_i/A}} + t_d \right),
\end{equation}
which is a cubic equation of $\sqrt{N_i}$. To guarantee the existence of positive roots, $\lambda$ and $N$ must satisfy
\begin{equation} \label{eqn:root_existence_condition}
    (N - 2 \lambda t_d)^{\frac{3}{2}} - \sqrt{27A} \phi \lambda \geq 0.
\end{equation}
Equation \eqref{eqn:fixed_point_equation} has two positive roots if strict inequality holds in \eqref{eqn:root_existence_condition}. These two roots correspond to two different market outcomes under the same demand and supply. However, we show that the larger positive root will always lead to a higher profit, which indicates that the smaller root can be excluded. Thus, the mapping from $\lambda$ and $N$ to $t_p$ is well-defined. Formally, this can be summarized as the  following proposition. 

\begin{proposition} \label{prop:optimal_condition}
    At the optimum of \eqref{eqn:valet_profit_maximization_unconstrained}, the market outcome must satisfy $N_i \geq \dfrac{N - 2 \lambda t_d}{3}$.
\end{proposition}

Please see \ref{appendix:proof_prop_optimal_condition} for the proof. \autoref{prop:optimal_condition} gives a sufficient condition for the solution to \eqref{eqn:valet_profit_maximization_unconstrained}. Given $\lambda$ and $N$, choosing the larger $N_i$ will lead to a shorter pickup time, which enables the platform to raise the price and consequently increase its profit. Hence, the optimal solution is achieved only when the larger root of \eqref{eqn:fixed_point_equation} is selected. From \autoref{prop:optimal_condition}, we further show that the profit-maximizing optimum demonstrates the following property: 

\begin{lemma} \label{lemma:WGC}
    The optimum of \eqref{eqn:valet_profit_maximization} always falls in a regime where $\dfrac{\partial t_p}{\partial \lambda} > 0$, i.e., under a fixed supply $N$, the pickup time $t_p$ increases with the demand $\lambda$.
\end{lemma}

\begin{proof}
    The partial derivative of $t_p$ with respect to $\lambda$ is
    \[\dfrac{\partial t_p}{\partial \lambda} = \dfrac{(t_p + t_d) \phi \sqrt{A} N_i^{-\frac{2}{3}}}{1 - \lambda \phi \sqrt{A} N_i^{-\frac{2}{3}}} = \dfrac{(t_p + t_d) t_p}{N_i - \lambda t_p}.\]
    We cannot pin down the sign of this partial derivative as the sign of its denominator remains undetermined. However, based on \autoref{prop:optimal_condition}, at the profit-maximizing optimum, we have
    \[N_i \geq \dfrac{N - 2 \lambda t_d}{3} = \dfrac{N_i + 2 \lambda t_p}{3}.\]
    Moving $N_i$ to the left-hand side, we obtain $N_i \geq \lambda t_p$, showing that $\dfrac{\partial t_p}{\partial \lambda} > 0$. The proof is completed.
\end{proof}

\begin{figure}[t]
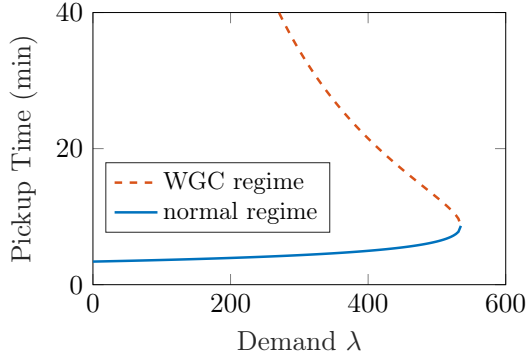

    \centering
    \includestandalone[width=0.4\linewidth]{figure/fig_WGC_analysis}
    \caption{The pickup time under varying demand when the supply is fixed at $N = 500$.}
    \label{fig:WGC_analysis}
\end{figure}

\autoref{lemma:WGC} states that the profit-maximizing solution always lies in a normal regime as opposed to the {\em wild goose chase} (WGC) regime, which refers to an inefficient outcome in the ride-hailing market where the pickup efficiency is extremely low \cite{castillo_surge_2017}. Intuitively, when demand increases, the customer waiting time should also increase. This is because a higher demand will occupy more couriers and reduce the number of idle couriers. Conversely, the WGC regime demonstrates the opposite properties. In the WGC regime, the idle couriers are spatially imbalanced so that the platform has to send distant couriers on a wild goose chase to serve new-coming customers. These long-distance pickup trips take up most of the courier labor time, reduce the fleet occupancy rate, and result in accumulated waiting customers, which in turn exacerbates the insufficiency of idle couriers. As such, increasing the demand in this regime will shorten the pickup time. Nonetheless, the pickup time in the WGC regime is much longer than that in the normal regime. \autoref{fig:WGC_analysis} illustrates how the pickup time changes with respect to the demand in two distinct regimes. Our observation in \autoref{lemma:WGC} is in line with the result in \cite[Lemma 1]{ke_pricing_2020}. Combining \autoref{prop:optimal_condition} and \autoref{lemma:WGC}, we can conclude that the condition of WGC regime is $N_i < (N - 2 \lambda t_d) / 3$. In other words, when \eqref{eqn:fixed_point_equation} has two positive roots, the larger one corresponds to the market outcome in the normal regime, whereas the smaller one pushes the market into the inefficient WGC regime. In \eqref{eqn:valet_profit_maximization_unconstrained}, the platform automatically excludes the market outcome in the WGC regime as it does not maximize its profit. Note that once the objective function \eqref{eqn:valet_profit_maximization_unconstrained} can be uniquely determined as a function of $\lambda$ and $N$, we can efficiently solve it through grid search, which guarantees global optimality.

\begin{remark}
    We emphasize that the change of variable in \eqref{eqn:valet_profit_maximization_unconstrained} does not affect the optimality of \eqref{eqn:valet_profit_maximization}. This is because given $\lambda$ and $N$, the time-based payment $\pi$ and the customer fare $p_v$ are both uniquely determined (after excluding the solution in the WGC regime). Hence, optimizing over $\pi$ and $p_v$ is equivalent to optimizing over $\lambda$ and $N$.
\end{remark}

\section{Charging Infrastructure Planning} \label{section:infrastructure_planning}

This section uncovers the strategic effect of charging infrastructure planning on the valet charging market.

\subsection{Problem Statement}

Appropriately deploying public charging infrastructure can significantly alleviate the inconvenience of EV charging. The siting and sizing problems of charging stations have been extensively studied. Particularly, some studies highlight the elasticity between driving distance and queuing time in planning charging infrastructure \cite{yang_data-driven_2017}\cite{huang_electric_2020}\cite{gan_fast-charging_2020}. The elastic demand refers to the phenomenon that customers will neither travel too far nor stay in line too long to obtain the service \cite{marianov_allocating_2005}. In this sense, given a limited charger supply, the placement of charging facilities is actually a trade-off between accessibility and charging congestion: (1) large and sparsely distributed charging stations may cause inconvenience for EV recharge due to the long driving distance, (2) small but densely distributed charging stations may result in congested charging queues and therefore a long waiting time.

Note that our model also demonstrates such characteristics. To see this, we assume that the total number of public chargers is fixed as $M_0$. Since there are $K$ homogeneous charging stations, each charging station would have $S = M_0/K$ charging outlets on average. By substituting $S = M_0/K$ and $\bar{\lambda} = \lambda/K$ into \eqref{eqn:waiting_time}, we can rewrite the waiting time $t_w$ as:
\begin{equation} \label{eqn:waiting_time_rewritten}
    t_w = \dfrac{K}{\lambda} \dfrac{M_0}{M_0 - \lambda t_c} \left(\dfrac{\lambda t_c}{M_0}\right)^{\sqrt{2 M_0/K + 2}}.
\end{equation}
From \eqref{eqn:waiting_time_rewritten}, we have $\dfrac{\partial t_w}{\partial K} > 0$. On the contrary, \eqref{eqn:delivery_time} gives $\dfrac{\mathrm{d } t_d}{\mathrm{d } K} < 0$. This reveals that when the number of chargers is fixed, building more charging stations will intensify the congestion, but at the same time reduce the travel distance to charging stations.

To capture the aforementioned trade-off, we consider the following charging infrastructure planning problem. Given a limited budget $B$ for public charging facilities, the city planner needs to decide the charger supply $M$ and charging station density $K$ to maximize the EV ownership, which is equivalent to maximizing the valet charging demand under the proposed business model. Let $\gamma_1$ and $\gamma_2$ denote the installation cost for each charger and the deployment cost for each charging station, respectively. The charger deployment is subject to
\begin{equation} \label{eqn:fixed_budget}
    \gamma_1 M + \gamma_2 K \leq B.
\end{equation}
Typically, the deployment cost $\gamma_2$ for each charging station includes the costs for parking spots, power-grid upgrade, construction, etc. These costs are positively related to the scale of each charging station, i.e., the number of outlets. In this light, we assume that the deployment cost $\gamma_2 K$ can be translated into a per-charger basis for simplification. Thereby \eqref{eqn:fixed_budget} can be rewritten as $\gamma M \leq B$, where $\gamma$ represents the combined installation cost for each charger. Obviously, EV penetration will always be maximized when the equality holds, i.e., the budget $B$ is used up. Let $M_0 = B/\gamma$ denote the nominal charger supply. The infrastructure planning problem is reduced to allocating theses $M_0$ public chargers into $K$ homogeneous charging stations.

To facilitate EV adoption, the city planner should subtly balance the delivery time and waiting time by determining an appropriate density of charging stations. We will demonstrate the decision-making problem with a numerical example, followed by a formal analysis that reveals in-depth economic insights.

\subsection{Numerical Example} \label{section:numerical_valet}

As a numerical example, we consider that the valet charging demand follows the nested logit model \cite{train_discrete_2009}. Specifically, potential private vehicle owners decide to purchase either an EV or a fossil fuel vehicle in the first place, after which EV owners choose to charge their cars using either valet charging or self-charging. In the lower nest, the EV owner's charging options are modeled as the following binary logit model:
\begin{equation} \label{eqn:P_vc}
    P_{vc} = \dfrac{\exp(-\epsilon_2 c_v)}{\exp(-\epsilon_2 c_s) + \exp(-\epsilon_2 c_v)},
\end{equation}
where $\epsilon_2 > 0$ is a sensitivity parameter, $c_v$ is the generalized cost of valet charging, and $c_s$ is the generalized cost of self-charging. The composite cost of EV charging is
\begin{equation} \label{eqn:c_ev}
    c_{ev} = -\dfrac{1}{\epsilon_2} \log \big(\exp(-\epsilon_2 c_v) + \exp(-\epsilon_2 c_s) \big).
\end{equation}
In the higher-level nest, the EV penetration is given by
\begin{equation} \label{eqn:P_ev}
    P_{ev} = \dfrac{\exp(-\epsilon_1 c_{ev})}{\exp(-\epsilon_1 c_{0}) + \exp(-\epsilon_1 c_{ev})},
\end{equation}
where $\epsilon_1 > 0$ is a sensitivity parameter, and $c_{0}$ is the generalized cost of refueling the conventional fossil fuel vehicles. As such, the demand model \eqref{eqn:F_v} can be rewritten as
\begin{equation} \label{eqn:demand_logit}
    \lambda = \tau \lambda_0 \dfrac{\exp(-\epsilon_1 c_{ev})}{\exp(-\epsilon_1 c_{0}) + \exp(-\epsilon_1 c_{ev})} \dfrac{\exp(-\epsilon_2 c_v)}{\exp(-\epsilon_2 c_s) + \exp(-\epsilon_2 c_v)},
\end{equation}
in which $c_{ev}$ is given by \eqref{eqn:c_ev}. Note that both $c_s$ and $c_{0}$ are considered as exogenous.

Similarly, the courier supply can also be described by a binary logit model:
\begin{equation} \label{eqn:supply_logit}
    N = N_0 \dfrac{\exp(\eta w)}{\exp(\eta w_0) + \exp(\eta w)},
\end{equation}
where $\eta > 0$ is a sensitivity parameter, and $w_0$ is the wage of outside option, e.g., selecting other jobs.

Moreover, to assess the economic impact of charging infrastructure planning, we define the social welfare as the sum of the platform's profit, customer surplus, and courier surplus:
\begin{equation} \label{eqn:social_welfare}
    \Pi_{SW} = \underbrace{\vphantom{\int_{c_v}^{\infty}}\lambda p_v - N w - K C}_\text{platform's profit} \, + \, \underbrace{\tau \lambda_0 \int_{c_v}^{\infty} F_v(x) \, \mathrm{d}x}_\text{customer surplus} \, + \, \underbrace{N_0 \int_{0}^{w} F_v(y) \, \mathrm{d}y}_\text{courier surplus}.
\end{equation}

Under the above assumptions, the model parameters in this study include
$$\Theta = \{\lambda_0, N_0, M_0, A, C, c_s, c_{0}, w_0, t_c, \alpha, \beta,  \tau, \theta, \phi, \epsilon_1, \epsilon_2, \eta\}.$$
These parameters are calibrated with partial reference to existing studies and official data in Hong Kong, which are specified as below:
\[\lambda_0 = 6\e5, \, N_0 = 5\e 4, \, M_0 = 3\e 3, \, A = 1\e 3\text{km}^2, \, C = \HKD 60, \, c_s = \HKD 80, \, c_{0} = \HKD 75, \, w_0 = \HKD 110/ \hr,\]
\[t_c = 5 \hr, \, \alpha = \HKD 60/ \hr, \, \beta = \HKD 10/ \hr, \, \tau = 1\%, \, \theta = 0.06, \, \phi = 0.04, \, \epsilon_1 = 0.11, \, \epsilon_2 = 0.1, \, \eta = 0.1.\]
The calibration procedure of these parameters is given in \ref{appendix:parameters_calibration}. We emphasize that they are only for preliminary illustration of the proposed business model. The economic insights we derive later are robust to perturbations in model parameters. One may fine-tune the parameters with real-world data in future research.

To evaluate how the market outcome is affected by charging station density under an exogenous charger supply $M_0$, we first fix $K$ and solve the profit-maximization problem \eqref{eqn:valet_profit_maximization_unconstrained}. By varying $K$, we can trace the market outcomes as a function of $K$. The simulation results are reported in \autoref{fig:valet_supply}-\ref{fig:valet_waiting_time}.

\begin{figure*}[htbp]
    \begin{minipage}[b]{0.32\linewidth}
        \centering
%
%
\definecolor{mycolor1}{rgb}{0.85000,0.32500,0.09800}%
\begin{tikzpicture}

\begin{axis}[%
width=1.694in,
height=1.03in,
at={(0.378in,0.457in)},
scale only axis,
xmin=20,
xmax=120,
xlabel style={font=\color{white!15!black}},
xlabel={Charging Station Density $K$},
ymin=100,
ymax=550,
scaled y ticks=base 10:-2,
ytick={100, 300, 500},
yticklabels={1.0, 3.0, 5.0},
ylabel style={font=\color{white!15!black}},
ylabel={Courier Supply},
axis background/.style={fill=white}
]
\addplot [color=black, line width=1.0pt, forget plot]
  table[row sep=crcr]{%
20	129.077846026837\\
21	152.221560432185\\
22	176.971634104008\\
23	203.23028460329\\
24	230.891464984092\\
25	259.842159325176\\
26	289.964011709658\\
27	321.133864763218\\
28	353.218681866322\\
29	386.037142665019\\
30	419.151397977774\\
31	451.04117802483\\
32	477.75522115611\\
33	495.556090416381\\
34	505.565706706516\\
35	510.747211030301\\
36	513.128895549504\\
37	513.833133179007\\
38	513.484653270981\\
39	512.449454242412\\
40	510.953593049175\\
41	509.143280114717\\
42	507.116669739422\\
43	504.941693668475\\
44	502.666524795998\\
45	500.325991703384\\
46	497.94573932499\\
47	495.544754384192\\
48	493.137304871395\\
49	490.734189402927\\
50	488.343568510832\\
51	485.971613236382\\
52	483.623008891412\\
53	481.301259403925\\
54	479.008967656174\\
55	476.748039601849\\
56	474.519779748202\\
57	472.325171350752\\
58	470.164726667268\\
59	468.038718993841\\
60	465.947224333846\\
61	463.890127581163\\
62	461.867207444403\\
63	459.878145064264\\
64	457.92244447403\\
65	455.999645404845\\
66	454.109190985539\\
67	452.250495692884\\
68	450.422991404198\\
69	448.626003624724\\
70	446.858920264318\\
71	445.121090463901\\
72	443.411873312032\\
73	441.730609434433\\
74	440.076692193876\\
75	438.449461531274\\
76	436.84829646171\\
77	435.272597035536\\
78	433.721751308358\\
79	432.195170867728\\
80	430.692278402352\\
81	429.212511035393\\
82	427.755318538641\\
83	426.320162004791\\
84	424.906520136007\\
85	423.513879969793\\
86	422.141748758389\\
87	420.789640521872\\
88	419.457092273634\\
89	418.143636441048\\
90	416.848830008381\\
91	415.57225731071\\
92	414.313482084473\\
93	413.072101560361\\
94	411.8476920081\\
95	410.639951689907\\
96	409.448432410108\\
97	408.272776653407\\
98	407.112670378993\\
99	405.967741426246\\
100	404.837629539982\\
101	403.722097986837\\
102	402.620769501305\\
103	401.533328323118\\
104	400.459523845143\\
105	399.399025308496\\
106	398.351574237328\\
107	397.316900650246\\
108	396.294731609882\\
109	395.284821104713\\
110	394.286886873226\\
111	393.300749840419\\
112	392.326067268958\\
113	391.362701059066\\
114	390.410398843952\\
115	389.468943135013\\
116	388.538118029437\\
117	387.617714649454\\
118	386.7075352331\\
119	385.807386026613\\
120	384.917073613405\\
};
\addplot [color=mycolor1, dashed, line width=1.0pt, forget plot]
  table[row sep=crcr]{%
20	513.833133179007\\
37	513.833133179007\\
};
\addplot [color=mycolor1, dashed, line width=1.0pt, forget plot]
  table[row sep=crcr]{%
37	0\\
37	513.833133179007\\
};
\addplot [color=mycolor1, only marks, mark size=2.0pt, mark=*, mark options={solid, mycolor1}, forget plot]
  table[row sep=crcr]{%
37	513.833133179007\\
};
\node[right, align=left]
at (axis cs:35,390.833) {$K^*_{N} = 37$};
\end{axis}

\end{tikzpicture}%
        \vspace*{-0.3in}
        \caption{Courier supply $N$ as a function of $K$.}
        \label{fig:valet_supply}
    \end{minipage}  
    \begin{minipage}[b]{0.008\linewidth}
        \hfill
    \end{minipage}
    \begin{minipage}[b]{0.32\linewidth}
        \centering
%
%
\definecolor{mycolor1}{rgb}{0.85000,0.32500,0.09800}%
\begin{tikzpicture}

\begin{axis}[%
width=1.694in,
height=1.03in,
at={(0.378in,0.457in)},
scale only axis,
xmin=20,
xmax=120,
xlabel style={font=\color{white!15!black}},
xlabel={Charging Station Density $K$},
ymin=50,
ymax=520,
scaled y ticks=base 10:-2,
ytick={ 100, 300, 500},
yticklabels={ 1.0, 3.0, 5.0},
ylabel style={font=\color{white!15!black}},
ylabel={Valet Demand},
axis background/.style={fill=white}
]
\addplot [color=black, line width=1.0pt, forget plot]
  table[row sep=crcr]{%
20	76.2746744647424\\
21	94.358117978335\\
22	114.595925476098\\
23	136.989063059804\\
24	161.519575348134\\
25	188.151473425522\\
26	216.832071876532\\
27	247.492561238841\\
28	280.043411681914\\
29	314.337348520488\\
30	349.961990775797\\
31	385.392026628372\\
32	416.516390604229\\
33	439.169153399371\\
34	454.073114583995\\
35	464.024410437136\\
36	471.006439174236\\
37	476.134576900209\\
38	480.037612783629\\
39	483.088912484619\\
40	485.52277201883\\
41	487.493661044956\\
42	489.107704281976\\
43	490.44039545968\\
44	491.547002689592\\
45	492.469017575616\\
46	493.238309125072\\
47	493.879678881996\\
48	494.412837426899\\
49	494.853653317144\\
50	495.215037656677\\
51	495.507607527478\\
52	495.740203607528\\
53	495.920228734084\\
54	496.053943079208\\
55	496.14667996307\\
56	496.202972459226\\
57	496.226834350121\\
58	496.221628461462\\
59	496.190322340336\\
60	496.135528976181\\
61	496.059504220459\\
62	495.964311788985\\
63	495.851770161639\\
64	495.72338193612\\
65	495.580617392721\\
66	495.424725311594\\
67	495.256829394991\\
68	495.078020227296\\
69	494.889146041694\\
70	494.691081261514\\
71	494.484563994516\\
72	494.270296830488\\
73	494.04887627909\\
74	493.820916794545\\
75	493.586902477509\\
76	493.34731541676\\
77	493.102599841583\\
78	492.853148359804\\
79	492.59933184005\\
80	492.341488784444\\
81	492.079929975424\\
82	491.81494409497\\
83	491.546796849191\\
84	491.275735018251\\
85	491.001984361012\\
86	490.725758007574\\
87	490.447250778379\\
88	490.16664733706\\
89	489.884112022684\\
90	489.599801707749\\
91	489.313872141953\\
92	489.026450853835\\
93	488.737667808904\\
94	488.447601202739\\
95	488.15647864139\\
96	487.864300807931\\
97	487.571163493633\\
98	487.277228049545\\
99	486.982527362454\\
100	486.687100167236\\
101	486.391152538069\\
102	486.094660224689\\
103	485.797676451421\\
104	485.500315970106\\
105	485.20259008703\\
106	484.904584131057\\
107	484.606345367483\\
108	484.307919570147\\
109	484.009366949224\\
110	483.710666256708\\
111	483.411984069799\\
112	483.113164500353\\
113	482.814411614577\\
114	482.515705556284\\
115	482.217077298246\\
116	481.918560037557\\
117	481.620165374568\\
118	481.321936824755\\
119	481.023896811971\\
120	480.726063631451\\
};
\addplot [color=mycolor1, dashed, line width=1.0pt, forget plot]
  table[row sep=crcr]{%
20	496.226834350121\\
57	496.226834350121\\
};
\addplot [color=mycolor1, dashed, line width=1.0pt, forget plot]
  table[row sep=crcr]{%
57	0\\
57	496.226834350121\\
};
\addplot [color=mycolor1, only marks, mark size=2.0pt, mark=*, mark options={solid, mycolor1}, forget plot]
  table[row sep=crcr]{%
57	496.226834350121\\
};
\node[right, align=left]
at (axis cs:59,430.227) {$K^*_{\lambda} = 57$};
\end{axis}

\end{tikzpicture}%
        \vspace*{-0.3in}
        \caption{Valet charging demand $\lambda$ as a function of $K$.}
        \label{fig:valet_demand}
    \end{minipage}  
    \begin{minipage}[b]{0.008\linewidth}
        \hfill
    \end{minipage}
    \begin{minipage}[b]{0.32\linewidth}
        \centering
%
%
\definecolor{mycolor1}{rgb}{0.85000,0.32500,0.09800}%
\begin{tikzpicture}

\begin{axis}[%
width=1.694in,
height=1.03in,
at={(0.378in,0.457in)},
scale only axis,
xmin=20,
xmax=120,
xlabel style={font=\color{white!15!black}},
xlabel={Charging Station Density $K$},
ymin=-903.558772567462,
ymax=9000,
scaled y ticks=base 10:-3,
ytick={0, 3000, 6000, 9000},
yticklabels={0.0, 3.0, 6.0, 9.0},
ylabel style={font=\color{white!15!black}},
ylabel={Profit (HK\$/hr)},
axis background/.style={fill=white}
]
\addplot [color=black, line width=1.0pt, forget plot]
  table[row sep=crcr]{%
20	-803.558772567462\\
21	-739.436561648233\\
22	-654.741382451431\\
23	-548.731861645983\\
24	-420.839821907206\\
25	-270.663096284881\\
26	-97.957363250538\\
27	97.3724149297923\\
28	315.277891177771\\
29	555.55253056706\\
30	817.703084455297\\
31	1100.33004921157\\
32	1399.29390834947\\
33	1706.30349928446\\
34	2012.38275554247\\
35	2311.67711980602\\
36	2601.22600369804\\
37	2879.70723992879\\
38	3146.64643407748\\
39	3402.01334212082\\
40	3646.01538558859\\
41	3878.98736948526\\
42	4101.32995674364\\
43	4313.4741399252\\
44	4515.86024578919\\
45	4708.92539522956\\
46	4893.09604395107\\
47	5068.78365336985\\
48	5236.3823253557\\
49	5396.2676832112\\
50	5548.7965469102\\
51	5694.30711246291\\
52	5833.1194463916\\
53	5965.53617083373\\
54	6091.84325673071\\
55	6212.31087024081\\
56	6327.19423601158\\
57	6436.73449343018\\
58	6541.15953046112\\
59	6640.68478547898\\
60	6735.51401147332\\
61	6825.83999971991\\
62	6911.84526184719\\
63	6993.70267048364\\
64	7071.57605946813\\
65	7145.62078514032\\
66	7215.984250545\\
67	7282.80639454549\\
68	7346.22014792126\\
69	7406.35185853124\\
70	7463.3216875678\\
71	7517.24397887641\\
72	7568.22760320474\\
73	7616.37627916004\\
74	7661.78887253545\\
75	7704.55967557198\\
76	7744.7786676027\\
77	7782.53175843744\\
78	7817.90101574124\\
79	7850.96487756831\\
80	7881.79835112922\\
81	7910.47319878515\\
82	7937.05811218867\\
83	7961.6188754202\\
84	7984.21851790399\\
85	8004.91745782742\\
86	8023.773636732\\
87	8040.84264589346\\
88	8056.1778450603\\
89	8069.83047407866\\
90	8081.84975788845\\
91	8092.28300534238\\
92	8101.1757022636\\
93	8108.57159912604\\
94	8114.51279371572\\
95	8119.0398091001\\
96	8122.19166721784\\
97	8124.00595836378\\
98	8124.5189068381\\
99	8123.76543300536\\
100	8121.77921198411\\
101	8118.59272918214\\
102	8114.23733287317\\
103	8108.74328399686\\
104	8102.13980334807\\
105	8094.45511632089\\
106	8085.7164953461\\
107	8075.95030016661\\
108	8065.18201607354\\
109	8053.43629022433\\
110	8040.73696615905\\
111	8027.10711659935\\
112	8012.56907467593\\
113	7997.1444635885\\
114	7980.85422489644\\
115	7963.71864541959\\
116	7945.75738288076\\
117	7926.98949034243\\
118	7907.43343950317\\
119	7887.10714291998\\
120	7866.02797521014\\
};
\addplot [color=mycolor1, dashed, line width=1.0pt, forget plot]
  table[row sep=crcr]{%
20	8124.5189068381\\
98	8124.5189068381\\
};
\addplot [color=mycolor1, dashed, line width=1.0pt, forget plot]
  table[row sep=crcr]{%
98	-803.558772567462\\
98	8124.5189068381\\
};
\addplot [color=mycolor1, only marks, mark size=2.0pt, mark=*, mark options={solid, mycolor1}, forget plot]
  table[row sep=crcr]{%
98	8124.5189068381\\
};
\node[right, align=left]
at (axis cs:58,5624.519) {$K^*_{\Pi} = 98$};
\end{axis}

\end{tikzpicture}%
        \vspace*{-0.3in}
        \caption{The platform's profit $\Pi$ as a function of $K$.}
        \label{fig:valet_profit}
    \end{minipage}    

    \begin{minipage}[b]{0.32\linewidth}
        \centering
%
%
\definecolor{mycolor1}{rgb}{0.00000,0.44700,0.74100}%
\begin{tikzpicture}

\begin{axis}[%
width=1.694in,
height=1.03in,
at={(0.378in,0.457in)},
scale only axis,
xmin=20,
xmax=120,
xlabel style={font=\color{white!15!black}},
xlabel={Charging Station Density $K$},
ymin=500,
ymax=5500,
scaled y ticks=base 10:-3,
ytick={1000, 3000, 5000},
yticklabels={1.0, 3.0, 5.0},
ylabel style={font=\color{white!15!black}},
ylabel={Surplus},
axis background/.style={fill=white},
legend style={nodes={scale=0.9, transform shape}, at={(1, 0.1)}, anchor=south east, legend cell align=left, align=left, draw=white!15!black}
]
\addplot [color=black, line width=1.0pt]
  table[row sep=crcr]{%
20	766.804261059831\\
21	949.808784749687\\
22	1155.17461735269\\
23	1383.10679987406\\
24	1633.64124718962\\
25	1906.65123533858\\
26	2201.8578299162\\
27	2518.83185397568\\
28	2856.94133511814\\
29	3214.95716678525\\
30	3588.85215723897\\
31	3962.74505293288\\
32	4292.90496591143\\
33	4534.21607537596\\
34	4693.45379890245\\
35	4799.98584699871\\
36	4874.83165741137\\
37	4929.85727954214\\
38	4971.76760990919\\
39	5004.55032392083\\
40	5030.71083340803\\
41	5051.90254644041\\
42	5069.26231052942\\
43	5083.59939276611\\
44	5095.50659474807\\
45	5105.42917984132\\
46	5113.70930317095\\
47	5120.61334629774\\
48	5126.35308791948\\
49	5131.09908179466\\
50	5134.99013561306\\
51	5138.14042404318\\
52	5140.6450416445\\
53	5142.58363426781\\
54	5144.02356757562\\
55	5145.0222439478\\
56	5145.6284607684\\
57	5145.88543236791\\
58	5145.82936943701\\
59	5145.49223050214\\
60	5144.90215904032\\
61	5144.08345465881\\
62	5143.058349243\\
63	5141.84643454782\\
64	5140.46390108318\\
65	5138.92659141786\\
66	5137.24796275738\\
67	5135.44012514081\\
68	5133.51483151177\\
69	5131.48122477079\\
70	5129.34872924072\\
71	5127.12530092119\\
72	5124.8185127712\\
73	5122.43479504589\\
74	5119.98077053911\\
75	5117.46165810999\\
76	5114.88265437735\\
77	5112.24854782113\\
78	5109.56357069527\\
79	5106.83172072226\\
80	5104.05664673535\\
81	5101.24169873385\\
82	5098.38998839558\\
83	5095.50437968865\\
84	5092.58753248855\\
85	5089.64188006284\\
86	5086.66971941011\\
87	5083.67315015876\\
88	5080.65416231654\\
89	5077.61452690155\\
90	5074.55593430802\\
91	5071.48006256091\\
92	5068.38828582505\\
93	5065.2820042366\\
94	5062.16206086035\\
95	5059.030905553\\
96	5055.88854723205\\
97	5052.73601763101\\
98	5049.57505330379\\
99	5046.40600908153\\
100	5043.22930261394\\
101	5040.04715073337\\
102	5036.8592934858\\
103	5033.66630390792\\
104	5030.4694163881\\
105	5027.2687530288\\
106	5024.06523162307\\
107	5020.85936064058\\
108	5017.65163242309\\
109	5014.44269438109\\
110	5011.23231832586\\
111	5008.02229467575\\
112	5004.81094812722\\
113	5001.60047176794\\
114	4998.3906521103\\
115	4995.18182186348\\
116	4991.97433757226\\
117	4988.76832367902\\
118	4985.56424748057\\
119	4982.362349599\\
120	4979.16282625934\\
};
\addlegendentry{customer surplus}

\addplot [color=mycolor1, dashed, line width=1.0pt]
  table[row sep=crcr]{%
20	1284.09666164287\\
21	1516.18667768945\\
22	1764.50486607175\\
23	2028.09354636596\\
24	2305.91142512648\\
25	2596.84609093546\\
26	2899.7298978529\\
27	3213.34493343018\\
28	3536.37145318322\\
29	3867.00026588421\\
30	4200.83079649052\\
31	4522.52799308386\\
32	4792.17288289828\\
33	4971.93118530918\\
34	5073.0395665954\\
35	5125.38661974377\\
36	5149.44984252897\\
37	5156.56529178617\\
38	5153.04432114827\\
39	5142.58502537951\\
40	5127.47174366035\\
41	5109.18204163527\\
42	5088.70786704304\\
43	5066.73573553553\\
44	5043.75246548889\\
45	5020.11000242313\\
46	4996.0674688182\\
47	4971.81669022555\\
48	4947.50179811683\\
49	4923.23185840896\\
50	4899.08927481724\\
51	4875.13634503688\\
52	4851.42035259548\\
53	4827.97664367964\\
54	4804.83146063822\\
55	4782.00400508501\\
56	4759.50740361455\\
57	4737.35153794095\\
58	4715.54153401896\\
59	4694.08010541882\\
60	4672.96797448045\\
61	4652.20393515728\\
62	4631.78571057512\\
63	4611.71004021251\\
64	4591.97187733235\\
65	4572.56653694566\\
66	4553.48836058141\\
67	4534.73140214294\\
68	4516.28989263707\\
69	4498.15699117198\\
70	4480.32648809918\\
71	4462.79178403806\\
72	4445.54638143077\\
73	4428.58359784943\\
74	4411.89728431117\\
75	4395.48075417289\\
76	4379.32771629749\\
77	4363.43209298234\\
78	4347.78768670289\\
79	4332.38853846492\\
80	4317.22880234973\\
81	4302.30277909017\\
82	4287.6048980091\\
83	4273.12970358325\\
84	4258.87191882632\\
85	4244.8263517065\\
86	4230.98801491761\\
87	4217.35199021917\\
88	4203.91359203429\\
89	4190.66808928008\\
90	4177.61099923445\\
91	4164.73811565962\\
92	4152.04502524937\\
93	4139.52765307478\\
94	4127.1817134519\\
95	4115.00415132557\\
96	4102.99043543852\\
97	4091.13694940162\\
98	4079.44051563822\\
99	4067.89737106279\\
100	4056.5038731063\\
101	4045.25762494086\\
102	4034.15481335577\\
103	4023.19224563592\\
104	4012.36738625727\\
105	4001.67689205861\\
106	3991.11814890976\\
107	3980.68842702955\\
108	3970.38496716957\\
109	3960.20528180969\\
110	3950.146518344\\
111	3940.20686453104\\
112	3930.38285879251\\
113	3920.6731024232\\
114	3911.07504524261\\
115	3901.58648886187\\
116	3892.20525111373\\
117	3882.92922086381\\
118	3873.75639922595\\
119	3864.68482791889\\
120	3855.71255226352\\
};
\addlegendentry{courier surplus}

\end{axis}

\end{tikzpicture}%
        \vspace*{-0.3in}
        \caption{Customer and courier surplus as a function of $K$.}
        \label{fig:valet_surplus}
    \end{minipage}
    \begin{minipage}[b]{0.008\linewidth}
        \hfill
    \end{minipage}
    \begin{minipage}[b]{0.32\linewidth}
        \centering
%
%
\definecolor{mycolor1}{rgb}{0.85000,0.32500,0.09800}%
\begin{tikzpicture}

\begin{axis}[%
width=1.694in,
height=1.03in,
at={(0.497in,0.457in)},
scale only axis,
xmin=20,
xmax=120,
xlabel style={font=\color{white!15!black}},
xlabel={Charging Station Density $K$},
ymin=0,
ymax=18000,
scaled y ticks=base 10:-4,
ytick={0, 5000, 10000, 15000},
yticklabels={0.0, 0.5, 1.0, 1.5},
ylabel style={font=\color{white!15!black}},
ylabel={Social Welfare},
axis background/.style={fill=white}
]
\addplot [color=black, line width=1.0pt, forget plot]
  table[row sep=crcr]{%
20	1247.34215013523\\
21	1726.55890079091\\
22	2264.93810097301\\
23	2862.46848459404\\
24	3518.71285040889\\
25	4232.83422998915\\
26	5003.63036451856\\
27	5829.54920233565\\
28	6708.59067947914\\
29	7637.50996323652\\
30	8607.38603818478\\
31	9585.60309522831\\
32	10484.3717571592\\
33	11212.4507599696\\
34	11778.8761210403\\
35	12237.0495865485\\
36	12625.5075036384\\
37	12966.1298112571\\
38	13271.4583651349\\
39	13549.1486914212\\
40	13804.197962657\\
41	14040.071957561\\
42	14259.3001343161\\
43	14463.8092682268\\
44	14655.1193060262\\
45	14834.464577494\\
46	15002.8728159402\\
47	15161.2136898931\\
48	15310.237211392\\
49	15450.5986234148\\
50	15582.8759573405\\
51	15707.583881543\\
52	15825.1848406316\\
53	15936.0964487812\\
54	16040.6982849445\\
55	16139.3371192736\\
56	16232.3301003945\\
57	16319.971463739\\
58	16402.5304339171\\
59	16480.2571213999\\
60	16553.3841449941\\
61	16622.127389536\\
62	16686.6893216653\\
63	16747.259145244\\
64	16804.0118378837\\
65	16857.1139135038\\
66	16906.7205738838\\
67	16952.9779218292\\
68	16996.0248720701\\
69	17035.990074474\\
70	17072.9969049077\\
71	17107.1610638357\\
72	17138.5924974067\\
73	17167.3946720554\\
74	17193.6669273857\\
75	17217.5020878549\\
76	17238.9890382775\\
77	17258.2123992409\\
78	17275.2522731394\\
79	17290.1851367555\\
80	17303.0838002143\\
81	17314.0176766092\\
82	17323.0529985933\\
83	17330.2529586921\\
84	17335.6779692189\\
85	17339.3856895968\\
86	17341.4313710597\\
87	17341.8677862714\\
88	17340.7455994111\\
89	17338.1130902603\\
90	17334.0166914309\\
91	17328.5011835629\\
92	17321.609013338\\
93	17313.3812564374\\
94	17303.856568028\\
95	17293.0748659787\\
96	17281.0706498884\\
97	17267.8789253964\\
98	17253.5344757801\\
99	17238.0688131497\\
100	17221.5123877043\\
101	17203.8975048564\\
102	17185.2514397147\\
103	17165.6018335407\\
104	17144.9766059934\\
105	17123.4007614083\\
106	17100.8998758789\\
107	17077.4980878367\\
108	17053.2186156662\\
109	17028.0842664151\\
110	17002.1158028289\\
111	16975.3362758061\\
112	16947.7628815957\\
113	16919.4180377796\\
114	16890.3199222494\\
115	16860.4869561449\\
116	16829.9369715668\\
117	16798.6870348853\\
118	16766.7540862097\\
119	16734.1543204379\\
120	16700.903353733\\
};
\addplot [color=mycolor1, dashed, line width=1.0pt, forget plot]
  table[row sep=crcr]{%
20	17341.8677862714\\
87	17341.8677862714\\
};
\addplot [color=mycolor1, dashed, line width=1.0pt, forget plot]
  table[row sep=crcr]{%
87	0\\
87	17341.8677862714\\
};
\addplot [color=mycolor1, only marks, mark size=2.0pt, mark=*, mark options={solid, mycolor1}, forget plot]
  table[row sep=crcr]{%
87	17341.8677862714\\
};
\node[right, align=left]
at (axis cs:40,12391.868) {$K^*_{SW} = 87$};
\end{axis}
\end{tikzpicture}%
        \vspace*{-0.3in}
        \caption{The social welfare $\Pi_{SW}$ as a function of $K$.}
        \label{fig:valet_social_welfare}
    \end{minipage}
    \begin{minipage}[b]{0.008\linewidth}
        \hfill
    \end{minipage}
    \begin{minipage}[b]{0.32\linewidth}
        \centering
%
%
\definecolor{mycolor1}{rgb}{0.00000,0.44700,0.74100}%
\begin{tikzpicture}

\begin{axis}[%
width=1.694in,
height=1.03in,
at={(0.573in,0.457in)},
scale only axis,
xmin=20,
xmax=120,
xlabel style={font=\color{white!15!black}},
xlabel={Charging Station Density $K$},
ymin=0.36,
ymax=0.43,
ytick={0.36, 0.39, 0.42},
yticklabels={36\%, 39\%, 42\%},
ylabel style={font=\color{white!15!black}},
ylabel={EV Penetration},
axis background/.style={fill=white},
legend style={nodes={scale=0.9, transform shape}, at={(1, 0.2)}, anchor=south east, legend cell align=left, align=left, draw=white!15!black}
]
\addplot [color=black, line width=1.0pt]
  table[row sep=crcr]{%
20	0.374716767303427\\
21	0.376811129063251\\
22	0.379153048777144\\
23	0.381742004047127\\
24	0.384575232667685\\
25	0.387647850400585\\
26	0.390953022205594\\
27	0.394482046334253\\
28	0.398223845544258\\
29	0.402160743410203\\
30	0.406244763525152\\
31	0.410300861821251\\
32	0.413859489520229\\
33	0.416446869277872\\
34	0.418147985640686\\
35	0.419283284535007\\
36	0.420079581989523\\
37	0.420664312500401\\
38	0.421109277617059\\
39	0.421457095915926\\
40	0.421734503967598\\
41	0.421959124998703\\
42	0.4221430643691\\
43	0.422294932086742\\
44	0.422421030589715\\
45	0.422526090768784\\
46	0.422613745960602\\
47	0.422686823494952\\
48	0.422747570153999\\
49	0.422797794638056\\
50	0.422838968458279\\
51	0.422872301597221\\
52	0.422898801532318\\
53	0.422919311843486\\
54	0.422934545866982\\
55	0.422945111300221\\
56	0.422951524638215\\
57	0.422954243191422\\
58	0.422953650091717\\
59	0.422950083426178\\
60	0.422943840879802\\
61	0.422935179444026\\
62	0.422924334217114\\
63	0.422911512355882\\
64	0.422896885027826\\
65	0.422880619718347\\
66	0.422862858677626\\
67	0.422843729906339\\
68	0.422823357633155\\
69	0.422801838476315\\
70	0.422779272039369\\
71	0.422755742396396\\
72	0.422731329576271\\
73	0.422706101523054\\
74	0.422680128226918\\
75	0.422653464827365\\
76	0.422626166236072\\
77	0.42259828305246\\
78	0.422569859995531\\
79	0.42254093930858\\
80	0.422511559548225\\
81	0.422481756113785\\
82	0.422451561887473\\
83	0.422421007134565\\
84	0.422390119964802\\
85	0.422358926093541\\
86	0.42232744979801\\
87	0.422295713269905\\
88	0.422263737544351\\
89	0.422231541341195\\
90	0.422199142530147\\
91	0.422166558853672\\
92	0.422133804838897\\
93	0.422100895291249\\
94	0.422067839116998\\
95	0.422034662249954\\
96	0.422001364764045\\
97	0.421967957572187\\
98	0.421934459059348\\
99	0.421900872970548\\
100	0.421867203717459\\
101	0.421833474783582\\
102	0.421799683403098\\
103	0.421765835640027\\
104	0.421731944571166\\
105	0.421698011483689\\
106	0.421664046101633\\
107	0.421630053811073\\
108	0.421596039828327\\
109	0.421562011015393\\
110	0.421527964949877\\
111	0.421493920618045\\
112	0.421459860251218\\
113	0.421425807109303\\
114	0.421391758929176\\
115	0.421357719241474\\
116	0.42132369183054\\
117	0.421289678019068\\
118	0.421255682767859\\
119	0.421221708633642\\
120	0.421187757702625\\
};
\addlegendentry{with valet}

\addplot [color=mycolor1, dashed, line width=1.0pt]
  table[row sep=crcr]{%
20	0.365864408989199\\
120	0.365864408989199\\
};
\addlegendentry{without valet}

\end{axis}

\end{tikzpicture}%
        \vspace*{-0.3in}
        \caption{The EV penetration $P_{ev}$ as a function of $K$.}
        \label{fig:valet_EV_penetration}
    \end{minipage}    

    \begin{minipage}[b]{0.32\linewidth}
        \centering
%
%
\begin{tikzpicture}

\begin{axis}[%
width=1.694in,
height=1.03in,
at={(0.454in,0.457in)},
scale only axis,
xmin=20,
xmax=120,
xlabel style={font=\color{white!15!black}},
xlabel={Charging Station Density $K$},
ymin=79,
ymax=91,
scaled y ticks=base 10:0,
ytick={80, 85, 90},
yticklabels={80, 85, 90},
ylabel style={font=\color{white!15!black}},
ylabel={Price (HK\$)},
axis background/.style={fill=white}
]
\addplot [color=black, line width=1.0pt, forget plot]
  table[row sep=crcr]{%
20	90.5428175038607\\
21	89.5438338310882\\
22	88.5765172091453\\
23	87.6431891961541\\
24	86.7446342791146\\
25	85.8807252592219\\
26	85.050793876136\\
27	84.2538783422108\\
28	83.4890005088631\\
29	82.7560321756353\\
30	82.0595550111371\\
31	81.422863914766\\
32	80.9076739048323\\
33	80.5643816212644\\
34	80.3610549808393\\
35	80.2417954674914\\
36	80.1706145257163\\
37	80.1281172729628\\
38	80.1036623236248\\
39	80.0910690237559\\
40	80.0865391815632\\
41	80.0876229952913\\
42	80.0926758085392\\
43	80.1005571090739\\
44	80.110458434516\\
45	80.121790496068\\
46	80.1341180005926\\
47	80.1471159607171\\
48	80.1605363392804\\
49	80.17419168792\\
50	80.1879366931192\\
51	80.2016590625973\\
52	80.2152727960514\\
53	80.2287108881447\\
54	80.2419215418544\\
55	80.2548658358356\\
56	80.2675136061423\\
57	80.2798416631886\\
58	80.291834826306\\
59	80.3034808438187\\
60	80.3147716876749\\
61	80.3257064970776\\
62	80.336278934023\\
63	80.3464884308876\\
64	80.3563411864064\\
65	80.365836879436\\
66	80.3749808665048\\
67	80.38377962997\\
68	80.3922335956745\\
69	80.4003537045586\\
70	80.4081435584826\\
71	80.4156116257176\\
72	80.4227627618515\\
73	80.4296061468141\\
74	80.436146217044\\
75	80.442391707481\\
76	80.4483489339948\\
77	80.4540254532708\\
78	80.4594278948561\\
79	80.4645627556109\\
80	80.4694366971549\\
81	80.4740566519876\\
82	80.4784289896874\\
83	80.4825597893989\\
84	80.4864554098375\\
85	80.4901217498042\\
86	80.4935648938194\\
87	80.4967903095991\\
88	80.4998041413197\\
89	80.5026110292059\\
90	80.5052165089977\\
91	80.5076265757776\\
92	80.5098452437695\\
93	80.5118776663792\\
94	80.5137311852842\\
95	80.5154034472052\\
96	80.5169049126847\\
97	80.5182409627728\\
98	80.519409890154\\
99	80.520420044701\\
100	80.5212781393536\\
101	80.5219791353411\\
102	80.5225330692762\\
103	80.5229441323183\\
104	80.5232139152544\\
105	80.5233473310287\\
106	80.5233458456949\\
107	80.5232137384339\\
108	80.5229533778579\\
109	80.522567071221\\
110	80.5220637221351\\
111	80.5214347557887\\
112	80.5206995599834\\
113	80.5198465244558\\
114	80.5188830435578\\
115	80.5178119283705\\
116	80.5166346119741\\
117	80.5153559852446\\
118	80.5139761098455\\
119	80.5124980619344\\
120	80.5109245968167\\
};
\end{axis}

\end{tikzpicture}%
        \vspace*{-0.3in}
        \caption{Valet charging price $p_v$ as a function of $K$.}
        \label{fig:valet_price}
    \end{minipage}
    \begin{minipage}[b]{0.008\linewidth}
        \hfill
    \end{minipage}
    \begin{minipage}[b]{0.32\linewidth}
        \centering
%
%
\begin{tikzpicture}

\begin{axis}[%
width=1.694in,
height=1.03in,
at={(0.454in,0.457in)},
scale only axis,
xmin=20,
xmax=120,
xlabel style={font=\color{white!15!black}},
xlabel={Charging Station Density $K$},
ymin=45,
ymax=90,
scaled y ticks=base 10:0,
ytick={50, 70, 90},
yticklabels={50, 70, 90},
ylabel style={font=\color{white!15!black}},
ylabel={Marginal Cost},
axis background/.style={fill=white}
]
\addplot [color=black, line width=1.0pt, forget plot]
  table[row sep=crcr]{%
20	85.032102357798\\
21	83.7134434086659\\
22	82.4642430223287\\
23	81.2786454299783\\
24	80.1513265354815\\
25	79.0774781907427\\
26	78.0527835765457\\
27	77.0733821423287\\
28	76.1357813780985\\
29	75.2367349580775\\
30	74.3727366962388\\
31	73.5383177111955\\
32	72.7238951709522\\
33	71.9225267042345\\
34	71.1380481177703\\
35	70.3766073606\\
36	69.6412098119819\\
37	68.9325605836648\\
38	68.2501955790978\\
39	67.5931812413817\\
40	66.9604066895323\\
41	66.3506505581371\\
42	65.7627259279135\\
43	65.1955076569586\\
44	64.6478770697688\\
45	64.118813226792\\
46	63.6073513686872\\
47	63.11256180897\\
48	62.6336103430921\\
49	62.1696944759859\\
50	61.7200690385866\\
51	61.2840308528604\\
52	60.8609364632858\\
53	60.4501759242211\\
54	60.0511751248594\\
55	59.6633904536556\\
56	59.2863240075934\\
57	58.919503074876\\
58	58.562474662124\\
59	58.2148204813655\\
60	57.8761614715526\\
61	57.5461087018865\\
62	57.2243114911749\\
63	56.9104418488918\\
64	56.6041896681487\\
65	56.3052489045483\\
66	56.0133528694201\\
67	55.7282216618918\\
68	55.449604946598\\
69	55.1772646934548\\
70	54.9109727750372\\
71	54.6505145349664\\
72	54.3956891046545\\
73	54.1462921686722\\
74	53.902137074085\\
75	53.6630490248175\\
76	53.4288567472619\\
77	53.199406201527\\
78	52.9745294514525\\
79	52.7540915533434\\
80	52.537945697172\\
81	52.3259586166212\\
82	52.1179979782957\\
83	51.9139451452136\\
84	51.713673738531\\
85	51.5170836652711\\
86	51.3240490340372\\
87	51.1344846269038\\
88	50.9482793005454\\
89	50.7653390224331\\
90	50.5855675120267\\
91	50.4088815169999\\
92	50.2351943507042\\
93	50.0644170818952\\
94	49.8964873376413\\
95	49.7313119699494\\
96	49.5688226837401\\
97	49.4089550373403\\
98	49.2516349709852\\
99	49.0967938079207\\
100	48.9443811879826\\
101	48.7943217352816\\
102	48.6465666677984\\
103	48.5010479471112\\
104	48.3577262836692\\
105	48.2165419970531\\
106	48.0774392055622\\
107	47.940365800817\\
108	47.8052833783405\\
109	47.6721358729041\\
110	47.5408885085105\\
111	47.4114891299336\\
112	47.2839048619731\\
113	47.1580906634962\\
114	47.0339960595212\\
115	46.9116041095847\\
116	46.7908566863889\\
117	46.6717313847201\\
118	46.5541868190544\\
119	46.4381895740342\\
120	46.3237176826284\\
};
\end{axis}

\end{tikzpicture}%
        \vspace*{-0.3in}
        \caption{The marginal cost $C_m$ as a function of $K$.}
        \label{fig:valet_marginal_cost}
    \end{minipage}
    \begin{minipage}[b]{0.008\linewidth}
        \hfill
    \end{minipage}
    \begin{minipage}[b]{0.32\linewidth}
        \centering
%
%
\begin{tikzpicture}

\begin{axis}[%
width=1.694in,
height=1.03in,
at={(0.378in,0.457in)},
scale only axis,
xmin=20,
xmax=120,
xlabel style={font=\color{white!15!black}},
xlabel={Charging Station Density $K$},
ymin=0,
ymax=0.04,
scaled y ticks=base 10:2,
ytick={0, 0.02, 0.04},
yticklabels={0.0, 2.0, 4.0},
ylabel style={font=\color{white!15!black}},
ylabel={Response Time (min)},
axis background/.style={fill=white}
]
\addplot [color=black, line width=1.0pt, forget plot]
  table[row sep=crcr]{%
20	0.0373812394220379\\
21	0.0257777252368972\\
22	0.018144524185862\\
23	0.0130071317651439\\
24	0.00947893802162808\\
25	0.00701175857309065\\
26	0.00525817867408236\\
27	0.00399320811876256\\
28	0.00306837249923208\\
29	0.00238436490575884\\
30	0.00187531487479539\\
31	0.0015014312325617\\
32	0.00124293089537208\\
33	0.00107733455414025\\
34	0.000969443353275437\\
35	0.000892686317040335\\
36	0.000833437941771556\\
37	0.000785023543020693\\
38	0.000743930694277319\\
39	0.000708138838078152\\
40	0.000676391578518883\\
41	0.000647853768834542\\
42	0.000621942030705328\\
43	0.000598230780056088\\
44	0.00057639536801352\\
45	0.000556185178478657\\
46	0.000537397607590991\\
47	0.000519869474145536\\
48	0.000503466110982519\\
49	0.000488071533022066\\
50	0.000473588773171961\\
51	0.000459934897303691\\
52	0.000447036487466191\\
53	0.00043483025909282\\
54	0.000423260433978033\\
55	0.000412276726808264\\
56	0.000401836314906661\\
57	0.000391897403886149\\
58	0.000382425226026752\\
59	0.000373388399779471\\
60	0.000364757670878085\\
61	0.000356504928909821\\
62	0.000348607704143083\\
63	0.000341043223170533\\
64	0.000333791058857124\\
65	0.000326833113662176\\
66	0.000320151992278043\\
67	0.000313731284089332\\
68	0.000307557427739942\\
69	0.00030161616641851\\
70	0.000295895362253269\\
71	0.000290382779225934\\
72	0.000285068238044713\\
73	0.000279941178265919\\
74	0.000274992209475167\\
75	0.000270212669816258\\
76	0.000265594500532855\\
77	0.000261129357796185\\
78	0.000256810393279876\\
79	0.00025263087274029\\
80	0.000248584419082451\\
81	0.000244664866112535\\
82	0.000240866577566902\\
83	0.000237184357521751\\
84	0.000233613060914997\\
85	0.000230148078328627\\
86	0.000226784770315794\\
87	0.000223519075031573\\
88	0.000220346636224983\\
89	0.000217264205359632\\
90	0.00021426808814269\\
91	0.000211354210569606\\
92	0.000208520091116759\\
93	0.000205762455599497\\
94	0.000203078408683759\\
95	0.000200464916301134\\
96	0.000197919659694553\\
97	0.000195439730139498\\
98	0.000193023308617558\\
99	0.000190667766161338\\
100	0.000188370951397049\\
101	0.000186130846073164\\
102	0.00018394548946109\\
103	0.000181813132110937\\
104	0.000179731302872346\\
105	0.000177698985061352\\
106	0.000175714435717668\\
107	0.000173775776215689\\
108	0.000171881874698943\\
109	0.000170031171068326\\
110	0.000168221727010178\\
111	0.000166453205217826\\
112	0.00016472323076612\\
113	0.000163031505353525\\
114	0.000161376492235484\\
115	0.000159756999176564\\
116	0.000158172247638342\\
117	0.000156620741014253\\
118	0.000155101871025014\\
119	0.000153614465787302\\
120	0.000152157557250082\\
};
\end{axis}

\end{tikzpicture}%
        \vspace*{-0.3in}
        \caption{The response time $t_r$ as a function of $K$.}
        \label{fig:valet_response_time}
    \end{minipage}
    
    \begin{minipage}[b]{0.32\linewidth}
        \centering
%
%
\begin{tikzpicture}

\begin{axis}[%
width=1.694in,
height=1.03in,
at={(0.454in,0.457in)},
scale only axis,
xmin=20,
xmax=120,
xlabel style={font=\color{white!15!black}},
xlabel={Charging Station Density $K$},
ymin=7.5,
ymax=15,
ylabel style={font=\color{white!15!black}},
ylabel={Pickup Time (min)},
axis background/.style={fill=white}
]
\addplot [color=black, line width=1.0pt, forget plot]
  table[row sep=crcr]{%
20	14.4252664456148\\
21	13.48735100679\\
22	12.6809447598229\\
23	11.9805521317035\\
24	11.3669481621546\\
25	10.8253473933299\\
26	10.3441730750407\\
27	9.91423427233016\\
28	9.52820022938606\\
29	9.18049006858029\\
30	8.86840962519982\\
31	8.59684551656855\\
32	8.38390121562625\\
33	8.24138966688934\\
34	8.15249261698412\\
35	8.09496785349397\\
36	8.05533232802139\\
37	8.0265106433029\\
38	8.00467602784315\\
39	7.98761917992894\\
40	7.97398275120577\\
41	7.96288560761416\\
42	7.9537312128866\\
43	7.94610056930395\\
44	7.93968772306294\\
45	7.9342668542719\\
46	7.92966368261929\\
47	7.92574484848287\\
48	7.92240566875565\\
49	7.91955964118144\\
50	7.91713824534141\\
51	7.91508584871671\\
52	7.91335482786461\\
53	7.91190634519863\\
54	7.91070764339097\\
55	7.90972990819105\\
56	7.90895091450882\\
57	7.90834714624597\\
58	7.90790185123968\\
59	7.90760045859273\\
60	7.90742920872197\\
61	7.90737408929789\\
62	7.90742661537716\\
63	7.90757641716909\\
64	7.90781516418978\\
65	7.9081359707375\\
66	7.90853186362286\\
67	7.90899599208094\\
68	7.90952431851808\\
69	7.91011110778178\\
70	7.91075246129123\\
71	7.91144351755744\\
72	7.91218164987128\\
73	7.91296289510942\\
74	7.91378413880725\\
75	7.91464296106096\\
76	7.91553711792211\\
77	7.91646315641187\\
78	7.91741961119114\\
79	7.91840459209199\\
80	7.91941627814396\\
81	7.92045267424372\\
82	7.92151222583206\\
83	7.92259371697484\\
84	7.92369555406118\\
85	7.92481673658187\\
86	7.92595576676294\\
87	7.92711189572591\\
88	7.92828338020301\\
89	7.92947039394547\\
90	7.93067185718205\\
91	7.93188549983595\\
92	7.93311215432428\\
93	7.93435065479211\\
94	7.93560049196579\\
95	7.93686008316101\\
96	7.93812984833841\\
97	7.93940832174992\\
98	7.94069605233818\\
99	7.94199190066165\\
100	7.94329556445277\\
101	7.94460631573747\\
102	7.94592406742583\\
103	7.94724880257668\\
104	7.94857835015828\\
105	7.94991422579996\\
106	7.95125577867654\\
107	7.95260194888038\\
108	7.95395340111616\\
109	7.95530952492358\\
110	7.95666881107685\\
111	7.9580331912164\\
112	7.9594001310427\\
113	7.96077131588843\\
114	7.96214573075152\\
115	7.96352298629781\\
116	7.96490370446593\\
117	7.96628648862477\\
118	7.96767218646651\\
119	7.96906004577618\\
120	7.97044978872756\\
};
\end{axis}
\end{tikzpicture}%
        \vspace*{-0.3in}
        \caption{The pickup time $t_p$ as a function of $K$.}
        \label{fig:valet_pickup_time}
    \end{minipage}
    \begin{minipage}[b]{0.008\linewidth}
        \hfill
    \end{minipage}
    \begin{minipage}[b]{0.32\linewidth}
        \centering
%
%
\begin{tikzpicture}

\begin{axis}[%
width=1.694in,
height=1.03in,
at={(0.454in,0.457in)},
scale only axis,
xmin=20,
xmax=120,
xlabel style={font=\color{white!15!black}},
xlabel={Charging Station Density $K$},
ymin=10,
ymax=26,
ylabel style={font=\color{white!15!black}},
ylabel={Delivery Time (min)},
axis background/.style={fill=white}
]
\addplot [color=black, line width=1.0pt, forget plot]
  table[row sep=crcr]{%
20	25.4558441227157\\
21	24.8423601363247\\
22	24.2711950486767\\
23	23.7376970422483\\
24	23.2379000772445\\
25	22.7683991532123\\
26	22.3262522260575\\
27	21.9089023002066\\
28	21.5141149680191\\
29	21.1399279025293\\
30	20.7846096908265\\
31	20.4466260328943\\
32	20.1246117974981\\
33	19.8173477722745\\
34	19.5237412036791\\
35	19.2428094176946\\
36	18.9736659610103\\
37	18.7155088167613\\
38	18.4676103375328\\
39	18.229308607506\\
40	18\\
41	17.7791327396926\\
42	17.5662013130736\\
43	17.3607415981348\\
44	17.1623266064207\\
45	16.9705627484771\\
46	16.7850865483256\\
47	16.6055617445199\\
48	16.431676725155\\
49	16.2631422522945\\
50	16.0996894379985\\
51	15.9410679397217\\
52	15.7870443475265\\
53	15.6374007394705\\
54	15.4919333848297\\
55	15.350451577604\\
56	15.2127765851133\\
57	15.078740698501\\
58	14.9481863736732\\
59	14.8209654526671\\
60	14.6969384566991\\
61	14.5759739432205\\
62	14.457947920245\\
63	14.3427433120127\\
64	14.2302494707577\\
65	14.1203617299493\\
66	14.0129809949074\\
67	13.9080133671526\\
68	13.8053697992527\\
69	13.7049657772839\\
70	13.6067210283322\\
71	13.5105592507344\\
72	13.4164078649987\\
73	13.3241977835569\\
74	13.2338631976885\\
75	13.145341380124\\
76	13.0585725019802\\
77	12.973499462816\\
78	12.8900677327098\\
79	12.8082252053668\\
80	12.7279220613579\\
81	12.6491106406735\\
82	12.5717453238524\\
83	12.4957824210112\\
84	12.4211800681624\\
85	12.3478981302606\\
86	12.2758981104685\\
87	12.2051430651746\\
88	12.1355975243384\\
89	12.0672274167718\\
90	12\\
91	11.9338837943723\\
92	11.8688485211242\\
93	11.8048650441112\\
94	11.741905314962\\
95	11.6799423214149\\
96	11.6189500386223\\
97	11.5589033832228\\
98	11.4997781699989\\
99	11.4415510709471\\
100	11.3841995766062\\
101	11.3277019594959\\
102	11.2720372395327\\
103	11.2171851512956\\
104	11.1631261130288\\
105	11.1098411972706\\
106	11.0573121030111\\
107	11.0055211292834\\
108	10.9544511501033\\
109	10.9040855906769\\
110	10.8544084047995\\
111	10.805404053378\\
112	10.7570574840095\\
113	10.7093541115571\\
114	10.6622797996637\\
115	10.6158208431526\\
116	10.5699639512647\\
117	10.5246962316844\\
118	10.4800051753125\\
119	10.4358786417441\\
120	10.3923048454133\\
};
\end{axis}

\end{tikzpicture}%
        \vspace*{-0.3in}
        \caption{The delivery time $t_d$ as a function of $K$.}
        \label{fig:valet_delivery_time}
    \end{minipage}
    \begin{minipage}[b]{0.008\linewidth}
        \hfill
    \end{minipage}
    \begin{minipage}[b]{0.32\linewidth}
        \centering
%
%
\begin{tikzpicture}

\begin{axis}[%
width=1.694in,
height=1.03in,
at={(0.454in,0.457in)},
scale only axis,
xmin=20,
xmax=120,
xlabel style={font=\color{white!15!black}},
xlabel={Charging Station Density $K$},
ymin=0,
ymax=15.5,
ylabel style={font=\color{white!15!black}},
ylabel={Waiting Time (min)},
axis background/.style={fill=white}
]
\addplot [color=black, line width=1.0pt, forget plot]
  table[row sep=crcr]{%
20	4.88663857183048e-15\\
21	3.74101369073425e-13\\
22	1.72344495443336e-11\\
23	5.19413520265122e-10\\
24	1.09406274378035e-08\\
25	1.6987960654159e-07\\
26	2.03163113591319e-06\\
27	1.94123800216694e-05\\
28	0.000152873097963362\\
29	0.0010177096440093\\
30	0.00580469006207588\\
31	0.0276527823650869\\
32	0.0990978656305413\\
33	0.244595716977431\\
34	0.446381572007505\\
35	0.676932927986483\\
36	0.919901363329815\\
37	1.16725490695241\\
38	1.41505675681435\\
39	1.66132091906304\\
40	1.90502657726166\\
41	2.14565918754298\\
42	2.38297852921324\\
43	2.61689908574961\\
44	2.84742175922647\\
45	3.07459838184116\\
46	3.29851527843574\\
47	3.51926558710907\\
48	3.73695401967925\\
49	3.95168954991525\\
50	4.1635794005342\\
51	4.37272702073129\\
52	4.57923453500874\\
53	4.78319858439978\\
54	4.98471199951476\\
55	5.1838638436536\\
56	5.38073312722369\\
57	5.57541162005713\\
58	5.76796776827621\\
59	5.95847268271721\\
60	6.14699651245087\\
61	6.33359865129963\\
62	6.51834847352181\\
63	6.70131026193465\\
64	6.88252380302495\\
65	7.0620515795635\\
66	7.23994022948681\\
67	7.41623476561643\\
68	7.59099419379777\\
69	7.76424709291701\\
70	7.93604306674527\\
71	8.10641806309781\\
72	8.2754133913968\\
73	8.44305922310329\\
74	8.60940009984154\\
75	8.77446139355296\\
76	8.93827631760465\\
77	9.10087731121419\\
78	9.2622919806908\\
79	9.4225492027433\\
80	9.58167604390758\\
81	9.73969823623952\\
82	9.89664099580691\\
83	10.0525283411594\\
84	10.207383653755\\
85	10.3612287609652\\
86	10.5140856905695\\
87	10.66597492577\\
88	10.8169172011453\\
89	10.9669305998104\\
90	11.1160336404581\\
91	11.2642468114774\\
92	11.4115851433937\\
93	11.5580661851186\\
94	11.7036950434036\\
95	11.848519592085\\
96	11.9925267769094\\
97	12.1357308700123\\
98	12.2781660749732\\
99	12.4198316464359\\
100	12.5607292136167\\
101	12.7009096996317\\
102	12.8403587406739\\
103	12.9790849144719\\
104	13.1171157725648\\
105	13.2544492138457\\
106	13.3911057198792\\
107	13.5270951890135\\
108	13.6624276654579\\
109	13.7971181756295\\
110	13.9311573468799\\
111	14.0645939607168\\
112	14.1973778449938\\
113	14.3295705296933\\
114	14.4611646869732\\
115	14.5921690807024\\
116	14.7225935368117\\
117	14.8524413338206\\
118	14.9817262100643\\
119	15.110455495302\\
120	15.2386354338329\\
};
\end{axis}

\end{tikzpicture}%
        \vspace*{-0.3in}
        \caption{The waiting time $t_w$ as a function of $K$.}
        \label{fig:valet_waiting_time}
    \end{minipage}
\end{figure*}

\subsection{Analysis} \label{section:analysis_valet}

According to numerical results, the market outcomes under different charging station densities $K$ can be segmented into four regimes:
\begin{itemize}
    \item When $K \leq 37$, the market expands sharply, which benefits all participants. The service price drops and the service quality improves. The platform's profit rises. Couriers receive higher wages. The number of couriers is maximized at $K^*_N = 37$.
    \item When $37 < K \leq 57$, the courier fleet size reduces, while more customers use valet charging services. The platform's profit keeps increasing. The number of customers peaks at $K^*_{\lambda} = 57$.
    \item When $57 < K \leq 98$, both the demand and supply reduce. Conversely, the platform's profit remains increasing and reaches the maximum at $K^*_{\Pi} = 98$.
    \item When $K > 98$, customers, couriers, and the platform suffer varying degrees of loss: customers pay a higher cost, couriers earn a lower wage, and the platform receives a lower profit. The decrease of courier supply outpaces the decrease of customer demand.
\end{itemize}

First and foremost, \autoref{fig:valet_EV_penetration} shows that the proposed business model can effectively increase EV penetration. Without valet charging services, the EV penetration is 36.6\%, which is derived by setting $c_v$ to be infinity in \eqref{eqn:demand_logit}. After introducing valet charging, even with the same public charger supply, the EV penetration can be raised to 42.3\% at $K = K^*_{\lambda}$. This confirms the effectiveness of valet charging in promoting EV adoption. However, it is important to note that EV penetration is sensitive to the density of charging stations. Therefore, to fully unlock the potential of valet charging in boosting EV penetration, the city planner should carefully decide $K$ in response to the new business model. For instance, the EV penetration merely increases by less than 1\% at a sub-optimal charging station density, e.g., $K = 20$.

The results also indicate that the charging station density should fall in a certain range, i.e., $K \in [K^*_N, K^*_{\Pi}]$, where the delivery time and queuing time are traded off. The city planner should determine an optimal deployment policy within this range in accordance with her objective: (a) select $K^*_N$ to maximize courier surplus, (b) select $K^*_{\lambda}$ to maximize EV penetration, or (c) select $K^*_{\Pi}$ to maximize the platform's profit. As shown in \autoref{fig:valet_social_welfare}, the social welfare is maximized between $[K^*_N, K^*_{\Pi}]$, where the benefits of multiple stakeholders are balanced. Any charging station density outside this range will lead to inefficient market outcomes and inferior surplus for the stakeholders. For instance, at a lower charging station density, i.e., $K < K^*_N$, the delivery time is so long that the platform needs to employ a larger courier fleet to fulfill the demand. The high marginal cost (see \autoref{fig:valet_marginal_cost}) deters the platform from serving more customers, resulting in a sluggish market: charging stations are under-utilized in this regime.\footnote{The definition of marginal cost is given in \ref{appendix:definition_of_Cm_Rm}.} At a higher charging station density, i.e., $K > K^*_{\Pi}$, charging stations are congested and all stakeholders are worse off: fewer EV owners adopt valet charging, fewer couriers are hired, and the platform suffers a profit decline.

The above discussion discloses that the interests of customers, couriers, and the platform are not aligned. As the delivery time is counted twice in \eqref{eqn:valet_cost}, compared to couriers, customers prefer a slightly higher charging station density, which shortens the time wasted on the road while does not result in long charging queues. Hence, it can decrease the generalized cost of using valet charging services. On the other hand, the platform prefers a much higher charging station density as it can reduce the marginal cost. After employing dedicated coordinators to handle on-site operations, the platform becomes less sensitive to charging congestion and tends to expand the market to compensate for the fixed cost. This is an important observation since it indicates the charging infrastructure should be deployed to cohere with the goals of city planners. In future work, multi-objective optimization can be employed to explore the Pareto frontier of the optimal charging infrastructure planning strategy.

Another important finding is obtained from the market outcomes in $[K_N^*, K_{\lambda}^*]$, where increasing $K$ leads to a downsized courier fleet and an increased valet charging demand. This is surprising as it reveals that the platform can serve more customers with fewer couriers. We show that this regime always exists regardless of model parameters. Denote $\lambda(K)$ and $N(K)$ as the demand and supply at equilibrium under density $K$. The conclusion is formally summarized as below:

\begin{proposition} \label{prop:different_optimal_K}
    Assume that $\lambda(K)$ and $N(K)$ first increase and then decrease with respect to K, and have a unique maxima at $K^*_N$ and $K^*_\lambda$, respectively. For any pickup time model $t_p(N_i)$ following the square root law \eqref{eqn:pickup_time} and any supply model such that $N w(N)$ is convex, we always have $K^*_N < K^*_\lambda$.
\end{proposition}

The proof can be found in \ref{appendix:proof_prop_different_optimal_K}. The convexity of $N w(N)$ follows the law of diminishing marginal utility \cite{pindyck_microeconomics_2018}.\footnote{As more couriers are employed, hiring an additional courier will yield greater additions to labor costs, which is equivalent to smaller additions to the platform's utility.} \autoref{prop:different_optimal_K} points out that compared with couriers, customers always prefer charging stations to be denser, i.e., $K^*_N < K^*_{\lambda}$. To understand the reason behind this result, we rewrite the first-order condition \eqref{eqn:FOC_2} as
\begin{equation} \label{eqn:marginal_analysis}
    w + N w^\prime(N) = - \alpha \lambda \left( \dfrac{\partial t_r}{\partial N} + \dfrac{\partial t_p}{\partial N} \right).
\end{equation}
The left-hand (right-hand) side of \eqref{eqn:marginal_analysis} corresponds to the marginal cost (revenue) for hiring one extra unit of courier when fixing the demand. The right-hand side decreases with $K$ but increases with $\lambda$. We will show these in the detailed proof. At $K = K^*_N$, $N$ attains its maximum. If we further increase $K$ marginally, $N$ remains constant,\footnote{This is because $\dfrac{\partial N}{\partial K}$ = 0 at $K = K^*_N$.} and therefore the left-hand side of \eqref{eqn:marginal_analysis} keeps unchanged. To maintain the equality, $\lambda$ must increase. This indicates that the customer arrival rate still increases after $N$ reaches the maximum, which dictates that $K^*_N < K^*_{\lambda}$. This further proves that the second regime always exists regardless of model parameters: no matter what parameter value we select, there is always a regime in which the platform can hire fewer couriers to serve more customers and meanwhile earn a higher profit when the charging station density increases.

\begin{remark}
    The intuition behind the assumption in \autoref{prop:different_optimal_K} that $\lambda(K)$ and $N(K)$ first increase and then decrease roots in the aforementioned trade-off between the delivery time $t_d$ and the waiting time $t_w$. We first note that $t_d(K)$ is convex and decreasing, while $t_w(K)$ is convex and increasing. When $K$ is small, increasing $K$ will lead to dramatic decrease of the delivery time, while the waiting time increases slowly. In this case, the platform always tends to expand the market under such favorable market conditions. On the contrary, if $K$ is large and is further increased, the adverse effect from the high density outweighs the beneficial effect, resulting in a shrinking market. In light of this trade-off, we assume that both the demand and supply will first increase and then decrease with the charging station density. Although we do not rigorously prove this assumption due to the complex interaction among endogenous variables, it has been validated by the substantial sensitivity analysis in \autoref{section:sensitivity_analysis}. The results demonstrate that this assumption is insensitive to model parameters.
\end{remark}

Note that although our simulation results in \autoref{fig:valet_demand}-\ref{fig:valet_profit} show that $K^*_{\Pi} > K^*_{\lambda}$, we will not prove this because the platform's profit is sensitive to the value of fixed cost $C$. If it is very large, the payment to dedicated coordinators will rapidly increase with $K$. In other words, $K^*_{\Pi}$ would reduce as $C$ increases. In view of this, the general relation between $K^*_{\lambda}$ and $K^*_{\Pi}$ cannot be determined. Nonetheless, we argue that $K^*_{\Pi} > K^*_{\lambda}$ holds for any $C$ in the regime of practical interest, and it will be violated only when $C$ is prohibitively high. We will include this as a sensitivity analysis in \autoref{section:sensitivity_analysis}.

\begin{figure}[htbp]
    \centering
%
%
\definecolor{mycolor1}{rgb}{0.85000,0.32500,0.09800}%

\begin{tikzpicture}

\begin{axis}[%
width=1.694in,
height=1.03in,
at={(0.497in,0.457in)},
scale only axis,
xmin=20,
xmax=120,
xlabel style={font=\color{white!15!black}},
xlabel={Charging Station Density $K$},
ymin=0.05,
ymax=0.45,
ytick={0.1, 0.2, 0.3, 0.4},
yticklabels={0.1, 0.2, 0.3, 0.4},
ylabel style={font=\color{white!15!black}},
ylabel={Lerner Index},
axis background/.style={fill=white}
]
\addplot [color=black, line width=1.0pt, forget plot]
  table[row sep=crcr]{%
20	0.0608630844277376\\
21	0.0651121375193803\\
22	0.069005582736838\\
23	0.0726188061451233\\
24	0.0760082487917135\\
25	0.0792173918879276\\
26	0.0822803642465929\\
27	0.0852245183386968\\
28	0.0880741066002332\\
29	0.090861016651927\\
30	0.0936736534076388\\
31	0.0968345477484568\\
32	0.101149598535069\\
33	0.107266446326808\\
34	0.114769609050902\\
35	0.122943262291384\\
36	0.131337457945478\\
37	0.139720700676885\\
38	0.147976589342921\\
39	0.156045960363786\\
40	0.163899359694802\\
41	0.171524287066952\\
42	0.178917107412934\\
43	0.186079223292029\\
44	0.193015764319795\\
45	0.199733145879476\\
46	0.206238828657017\\
47	0.212541074592083\\
48	0.218647813457802\\
49	0.22456724330963\\
50	0.230307305763578\\
51	0.235875771534497\\
52	0.241279941563925\\
53	0.246526894735962\\
54	0.25162341615241\\
55	0.256576036452427\\
56	0.261390800037729\\
57	0.266073501713285\\
58	0.270629762181781\\
59	0.275064793335836\\
60	0.279383353082056\\
61	0.283590382065545\\
62	0.2876902907319\\
63	0.291687254037929\\
64	0.295585278866278\\
65	0.299388258856598\\
66	0.303099642879506\\
67	0.306723048873478\\
68	0.310261669983237\\
69	0.313718632430266\\
70	0.317096871722959\\
71	0.3203991932645\\
72	0.323628196338748\\
73	0.326786555813353\\
74	0.329876683442307\\
75	0.33290087619527\\
76	0.335861363778907\\
77	0.338760168906325\\
78	0.341599476438246\\
79	0.344381057365967\\
80	0.347106828957963\\
81	0.349778539897569\\
82	0.352397920379706\\
83	0.354966525902527\\
84	0.357486008357497\\
85	0.359957687411544\\
86	0.362383202908956\\
87	0.364763682747658\\
88	0.36710058062869\\
89	0.369395124289624\\
90	0.371648575016589\\
91	0.373862034430331\\
92	0.376036630071751\\
93	0.378173525037517\\
94	0.380273568208933\\
95	0.382337914973515\\
96	0.384367509686392\\
97	0.386363208553149\\
98	0.388325932366182\\
99	0.390256610923484\\
100	0.392155932953804\\
101	0.394024808390909\\
102	0.395863930088412\\
103	0.397674185044545\\
104	0.399456083129485\\
105	0.401210411697907\\
106	0.40293788465159\\
107	0.404639189432464\\
108	0.406314828592887\\
109	0.40796552312175\\
110	0.409591777570875\\
111	0.411194183589417\\
112	0.412773297793455\\
113	0.41432960072554\\
114	0.41586377900849\\
115	0.417376068896188\\
116	0.418867207852346\\
117	0.420337514333623\\
118	0.421787507357228\\
119	0.423217628419485\\
120	0.424628174193643\\
};
\addplot [color=mycolor1, dashed, line width=1.0pt, forget plot]
  table[row sep=crcr]{%
20	0.266073501713285\\
57	0.266073501713285\\
};
\addplot [color=mycolor1, dashed, line width=1.0pt, forget plot]
  table[row sep=crcr]{%
57	0\\
57	0.266073501713285\\
};
\addplot [color=mycolor1, only marks, mark size=2.0pt, mark=*, mark options={solid, mycolor1}, forget plot]
  table[row sep=crcr]{%
57	0.266073501713285\\
};
\node[right, align=left]
at (axis cs:44, 0.32) {$K^*_{\lambda}$};
\end{axis}

\end{tikzpicture}%
    \caption{Lerner Index $L$ as a function of $K$.}
    \label{fig:valet_lerner_index}
\end{figure}
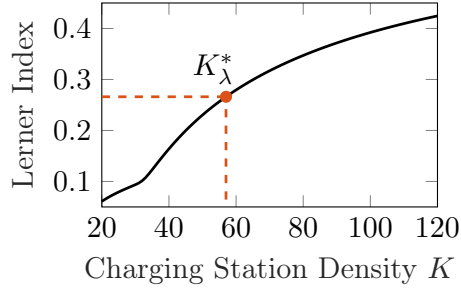

In addition, we observe that the platform has a strong market power that compromises the efficiency of the market outcomes. As $K$ keeps going up from $K^*_N$, the marginal cost continues to drop, while the price barely changes. This underscores a growing market power, which can be measured by the Lerner Index:
\begin{equation} \label{eqn:lerner_index}
    L = \dfrac{p_v - C_m}{p_v},
\end{equation}
where $C_m$ refers to the marginal cost. The Lerner Index measures the relative distance between the market price and the competitive price, thus the higher $L$ is, the higher market power the platform has \cite{pindyck_microeconomics_2018}. As shown in \autoref{fig:valet_lerner_index}, the platform charges a 26.6\% markup over the marginal cost at $K^*_{\lambda}$. For $K \in[K^*_{\lambda}, K^*_{\Pi}]$, it enjoys a profit growth even though both demand and supply are shrinking. Compared with a competitive market, most benefits of increased charging convenience are reaped by the monopolistic platform. In this light, regulatory interventions are in need to weaken the platform's market power.

\section{Taxation Regulation} \label{section:taxation}
This section examines a prospective taxation scheme for valet charging market. According to the numerical results in the precedent section, the platform's market power is notably high, e.g., $L(K^*_{\lambda}) = 26.6\%$, which motivates the government to regulate the market.

\subsection{Taxation Scheme}
One way to suppress the platform's market power is to nurture a competitive market by introducing competition among platforms. However, like Uber, Airbnb and other two-sided platforms, the success of their business models crucially depends on the economy of scale. In this case, a competitive but fragmented market may lead to inefficient matching between supply and demand, incurring inferior market outcomes.

Likewise, we argue that the valet charging platform should enjoy a concentrated market, but meanwhile be regulated by the government to ensure fairness and efficiency. Inspired by the congestion tax for vehicles that create traffic congestion on the road \cite{li2021impact}, a prospective policy is to impose a ``valet-charging'' tax on the platform to penalize the charging congestion caused by this service. Yet, purely taxing may discourage the platform's willingness to expand the market, which contradicts the objective to promote EV adoption. To avoid excessive disincentives, the tax revenue is invested in providing more public chargers to accommodate a larger market. To put it another way, the platform is obligated to undertake part of the cost for EV charging infrastructure. This will not only curb the platform's market power but also bring more EVs into use by providing more public chargers. The city planner can regulate the platform by deciding the tax rate.

Suppose a per-service tax $p_t$ is imposed on the platform for serving each customer. The tax revenue $\lambda p_t$ is in turn put into investing public charging facilities. We denote $r$ as the prorated installation cost of each public charger.\footnote{The cost includes (a) equipment price and installation fee, and (b) maintenance cost. As the former is sunk, we split it into a daily expense over a long period and combine it with the latter to derive a per-time cost of each charger.} Under this taxation scheme, the profit-maximizing problem becomes:
\begin{equation} \label{eqn:tax_profit_maximization}
    \max_{\lambda, N} \quad \lambda (p_v - p_t) - N w - K C
\end{equation}
\begin{subnumcases} {\label{eqn:tax_profit_maximization_constraint}}
    \lambda = \tau\lambda_0 F_v \big(p_v + \alpha (t_p + t_r) + \beta (t_w + 2 t_d) \big) \label{eqn:tax_constraint_demand}\\
    N = N_0 F_c\left(w\right) \label{eqn:tax_constraint_supply}\\
    M = M_0 + \dfrac{\lambda p_t}{r} \label{eqn:tax_constraint_M}\\
    t_r = \dfrac{\rho_d^{\sqrt{2N + 2}}}{2\lambda(1 - \rho_d)} \label{eqn:tax_constraint_response_time}\\
    t_p = \dfrac{\phi}{\sqrt{(N - 2 \lambda(t_p + t_d))/A}} \label{eqn:tax_constraint_pickup_time}\\
    t_w = \dfrac{K \rho_c^{\sqrt{2M/K+2}}}{\lambda(1 - \rho_c)} \label{eqn:tax_constraint_waiting_time}\\
    \lambda (p_v - p_t) - N w - K C \geq 0
    \label{eqn:tax_constraint_profitability}
\end{subnumcases}
where \eqref{eqn:tax_constraint_M} specifies the updated charger supply after the tax revenue is used to expand the public charging infrastructure, and \eqref{eqn:tax_constraint_profitability} guarantees that the platform is economically viable under the taxation scheme. Clearly, for the platform, the profit-optimizing decisions depend on the exogenous tax rate $p_t$ and charging station density $K$.

On top of \eqref{eqn:tax_profit_maximization}, we introduce a regulatory agency that determines the charging station density and the tax rate to maximize the EV penetration, i.e., $\lambda$, subject to the the platform's profit-maximizing decisions in \eqref{eqn:tax_profit_maximization}. The decision-making process follows the Stackelberg game, where the regulatory agency plays as the leader and the platform plays as the follower \cite{osborne_introduction_2004}. To evaluate the impact of this taxation scheme, under each exogenous tax rate $p_t$, we first find the optimal charging station density at which the valet charging demand is maximized. We then vary $p_t$ and trace the market outcomes as a function of $p_t$. 

For notation convenience, we denote by $\lambda (p_t, K)$ and $\Pi(p_t, K)$ the valet charging demand and the platform's profit at tax rate $p_t$ and charging station density $K$. For each $p_t$, we further denote by $K^*(p_t)$ the optimal charging station density whereby the demand is maximized, and denote by $\lambda^*(p_t) = \lambda(p_t, K^*(p_t))$ the maximal valet charging demand. Regarding the charger cost $r$, we have the following proposition:

\begin{proposition} \label{prop:existence_of_r_hat}
    When $\Pi(0, K^*(0)) \geq 0$, there exists a threshold $\hat{r} > 0$ such that
    \begin{enumerate}
        \item [(i)] if $r \in (0, \hat{r}), \exists p_t > 0, \lambda^*(p_t) > \lambda^*(0)$;
        \item [(ii)] if $r \in (\hat{r}, +\infty), \forall p_t > 0, \lambda^*(p_t) < \lambda^*(0)$.
    \end{enumerate}
\end{proposition}

It says that when the installation cost of each charger is lower than the threshold, i.e., $r < \hat{r}$, the regulatory agency can always encourage the platform to serve more customers by deciding an appropriate tax rate.\footnote{This does not necessarily mean that the taxation scheme will benefit the platform in terms of its profit. Instead, it means that the increase of charging facilities can partly relieve the tax penalty imposed on the platform. In other words, the platform can mitigate the tax burden by serving more customers.} On the other hand, if the cost is too high, i.e., $r > \hat{r}$, the tax revenue cannot assume adequate chargers to accommodate the market expansion. In this case, the tax burden will always reduce the EV adoption. The condition $\Pi(0, K^*(0)) \geq 0$ guarantees that the profit is positive without the tax, thereby the platform will provide valet charging services, i.e., $\lambda^*(0) > 0$. The proof can be found in \ref{appendix:proof_prop_effect_of_cost}.

\begin{remark}
    To capture the potential regulatory risk arising from the taxation scheme, we impose a profitability constraint \eqref{eqn:tax_constraint_profitability} in the profit-maximizing problem, which rules out the case of over-taxation. In practice, the regulator decides the tax rate to maximize the EV penetration, social welfare, etc. The desired tax rate may lie on the boundary where the platform receives a zero net profit. In this case, the regulator compromises the platform's profit to achieve its target. 
\end{remark}

\subsection{Numerical Example} \label{section:analysis_tax}

This subsection presents a numerical example that solves the platform's decision-making problem \eqref{eqn:tax_profit_maximization} under the proposed taxation scheme, and traces the market outcomes as we perturb the tax rate $p_t$. The model parameters are consistent with those in \autoref{section:numerical_valet}. Besides, we set $r = \HKD 25$. Numerical results are presented in \autoref{fig:tax_demand}-\ref{fig:tax_lerner_index}. Key findings are summarized as below.
\begin{figure*}[t]
    \begin{minipage}[b]{0.32\linewidth}
        \centering
%
%
\definecolor{mycolor1}{rgb}{0.85000,0.32500,0.09800}%
\begin{tikzpicture}

\begin{axis}[%
width=1.694in,
height=1.03in,
at={(0.573in,0.457in)},
scale only axis,
xmin=0,
xmax=18,
xtick={0, 6, 12, 18},
xticklabels={0, 6, 12, 18},
xlabel style={font=\color{white!15!black}},
xlabel={Tax Rate $p_t$},
ymin=490,
ymax=515,
scaled y ticks=base 10:-2,
ytick={ 490, 500, 510},
yticklabels={ 4.9, 5.0, 5.1},
ylabel style={font=\color{white!15!black}},
ylabel={Valet Demand},
axis background/.style={fill=white}
]
\addplot [color=black, line width=1.0pt, forget plot]
  table[row sep=crcr]{%
0	496.228191193006\\
0.2	496.542591869114\\
0.4	496.854388880499\\
0.6	497.163534436287\\
0.8	497.469977465129\\
1	497.773670156189\\
1.2	498.07456091363\\
1.4	498.372597058388\\
1.6	498.667722523766\\
1.8	498.959888760477\\
2	499.249042807393\\
2.2	499.535124406072\\
2.4	499.818079625141\\
2.6	500.097850525907\\
2.8	500.374382342701\\
3	500.647587537861\\
3.2	500.917488615321\\
3.4	501.183944835591\\
3.6	501.446917198292\\
3.8	501.706338294366\\
4	501.96217155798\\
4.2	502.214323919227\\
4.4	502.46274150822\\
4.6	502.707356756402\\
4.8	502.948103146899\\
5	503.184912074933\\
5.2	503.417715891725\\
5.4	503.646437397008\\
5.6	503.8710144389\\
5.8	504.091366092423\\
6	504.307424997293\\
6.2	504.519111505164\\
6.4	504.726340109389\\
6.6	504.929078545541\\
6.8	505.127194880863\\
7	505.320633099553\\
7.2	505.509309455476\\
7.4	505.693135893339\\
7.6	505.872040103225\\
7.8	506.045931410205\\
8	506.214725978474\\
8.2	506.378326540046\\
8.4	506.536666615408\\
8.6	506.689637221082\\
8.8	506.837155611867\\
9	506.97912591149\\
9.2	507.115450034583\\
9.4	507.246040637557\\
9.6	507.37079797725\\
9.8	507.48961910867\\
10	507.60241052069\\
10.2	507.709066937118\\
10.4	507.809486392061\\
10.6	507.90356069061\\
10.8	507.991188694241\\
11	508.072259356305\\
11.2	508.146667659931\\
11.4	508.214293943532\\
11.6	508.275031672956\\
11.8	508.328768749602\\
12	508.375382763363\\
12.2	508.414756196331\\
12.4	508.446771766549\\
12.6	508.471306825311\\
12.8	508.488236042594\\
13	508.497438681219\\
13.2	508.498780624368\\
13.4	508.492138490064\\
13.6	508.477378311686\\
13.8	508.454369937636\\
14	508.422974144435\\
14.2	508.383053298052\\
14.4	508.3344766781\\
14.6	508.277092047263\\
14.8	508.210763202159\\
15	508.135341269802\\
15.2	508.05068407647\\
15.4	507.956636316251\\
15.6	507.85304355318\\
15.8	507.739757070346\\
16	507.616621899297\\
16.2	507.483471297698\\
16.4	507.340154630903\\
16.6	507.186502088265\\
16.8	507.022347817936\\
17	506.847527540292\\
17.2	506.661864154093\\
17.4	506.465194026242\\
17.6	506.257334405327\\
17.8	506.038108429874\\
18	505.807338408748\\
18.2	505.564840288497\\
18.4	505.310428456448\\
18.6	505.043916915273\\
18.8	504.765111837645\\
19	504.473822006988\\
19.2	504.169849618631\\
19.4	503.853000317125\\
19.6	503.523069126591\\
19.8	503.179850352561\\
20	502.823146656562\\
20.2	502.45273845436\\
20.4	502.068415460782\\
20.6	501.669969599004\\
20.8	501.257176942313\\
21	500.829819112661\\
21.2	500.387669534519\\
21.4	499.930510751489\\
21.6	499.4581051606\\
21.8	498.970226151361\\
22	498.466635092297\\
22.2	497.947097278331\\
22.4	497.411373948244\\
22.6	496.859221105407\\
22.8	496.290386426999\\
23	495.704631380565\\
23.2	495.101694391717\\
23.4	494.481330511478\\
23.6	493.843275035708\\
23.8	493.18727437481\\
24	492.513059552472\\
24.2	491.820365686939\\
24.4	491.108926167444\\
24.6	490.378467023872\\
24.8	489.628721862625\\
25	488.85940888819\\
};
\addplot [color=mycolor1, dashed, line width=1.0pt, forget plot]
  table[row sep=crcr]{%
0	508.498780624368\\
13.2	508.498780624368\\
};
\addplot [color=mycolor1, dashed, line width=1.0pt, forget plot]
  table[row sep=crcr]{%
13.2	0\\
13.2	508.498780624368\\
};
\addplot [color=mycolor1, only marks, mark size=2.0pt, mark=*, mark options={solid, mycolor1}, forget plot]
  table[row sep=crcr]{%
13.2	508.498780624368\\
};
\node[right, align=left]
at (axis cs:9.5,511) {$p^*_{t, \lambda} = 13.2$};
\end{axis}
\end{tikzpicture}%
        \vspace*{-0.3in}
        \caption{Valet charging demand $\lambda$ as a function of $p_t$.}
        \label{fig:tax_demand}
    \end{minipage}
    \begin{minipage}[b]{0.005\linewidth}
        \hfill
    \end{minipage}
    \begin{minipage}[b]{0.32\linewidth}
        \centering
%
%
\begin{tikzpicture}

\begin{axis}[%
width=1.694in,
height=1.03in,
at={(0.497in,0.457in)},
scale only axis,
xmin=0,
xmax=18,
xtick={0, 6, 12, 18},
xticklabels={0, 6, 12, 18},
xlabel style={font=\color{white!15!black}},
xlabel={Tax Rate $p_t$},
ymin=400,
ymax=480,
scaled y ticks=base 10:-2,
ytick={400, 440, 480},
yticklabels={4.0, 4.4, 4.8},
ylabel style={font=\color{white!15!black}},
ylabel={Courier Supply},
axis background/.style={fill=white}
]
\addplot [color=black, line width=1.0pt, forget plot]
  table[row sep=crcr]{%
0	471.651770218696\\
0.2	471.104716828228\\
0.4	470.540929723336\\
0.6	469.981944107257\\
0.8	469.406725654076\\
1	468.836413404949\\
1.2	468.270861403185\\
1.4	467.689614296307\\
1.6	467.133335666656\\
1.8	466.541516059281\\
2	465.954635304775\\
2.2	465.372535493985\\
2.4	464.775693367441\\
2.6	464.164515905918\\
2.8	463.596411654066\\
3	462.957183500973\\
3.2	462.360746618064\\
3.4	461.768888471121\\
3.6	461.14490696213\\
3.8	460.544013116874\\
4	459.911764640675\\
4.2	459.302272200912\\
4.4	458.679717944767\\
4.6	458.061851887996\\
4.8	457.396846195005\\
5	456.788373444848\\
5.2	456.133458365161\\
5.4	455.483581758508\\
5.6	454.838575539344\\
5.8	454.181889249278\\
6	453.546248750002\\
6.2	452.8668525846\\
6.4	452.224322116287\\
6.6	451.539030582668\\
6.8	450.889902443371\\
7	450.198930159778\\
7.2	449.528225080147\\
7.4	448.84700792098\\
7.6	448.140698256789\\
7.8	447.48368970002\\
8	446.786939563541\\
8.2	446.0516659702\\
8.4	445.35037478623\\
8.6	444.65371416402\\
8.8	443.947537540633\\
9	443.245901252228\\
9.2	442.507663434814\\
9.4	441.774459608889\\
9.6	441.086106871103\\
9.8	440.335493158789\\
10	439.589873729802\\
10.2	438.86192225581\\
10.4	438.112728461879\\
10.6	437.368151636261\\
10.8	436.590591615535\\
11	435.84270682628\\
11.2	435.062534101093\\
11.4	434.299249621091\\
11.6	433.564115850601\\
11.8	432.77374894836\\
12	431.964863676737\\
12.2	431.172488970959\\
12.4	430.384623203873\\
12.6	429.612240441006\\
12.8	428.777157737924\\
13	427.969084162135\\
13.2	427.176149251485\\
13.4	426.344328539256\\
13.6	425.506582949761\\
13.8	424.694286886728\\
14	423.854973014049\\
14.2	423.009847999852\\
14.4	422.16907188843\\
14.6	421.312557613076\\
14.8	420.450571533082\\
15	419.602528938002\\
15.2	418.739185497182\\
15.4	417.870411761482\\
15.6	416.977707419509\\
15.8	416.107738341719\\
16	415.205186762875\\
16.2	414.324795767189\\
16.4	413.40384712054\\
16.6	412.496007885486\\
16.8	411.583456071858\\
17	410.674718580416\\
17.2	409.752772061098\\
17.4	408.801510241369\\
17.6	407.870764529617\\
17.8	406.92746597503\\
18	405.979882793462\\
18.2	405.035833623434\\
18.4	404.064207788141\\
18.6	403.073574332207\\
18.8	402.109456010531\\
19	401.126310170999\\
19.2	400.14657007056\\
19.4	399.126881925\\
19.6	398.132369975368\\
19.8	397.113011793299\\
20	396.097139407361\\
20.2	395.070872654442\\
20.4	394.0477892805\\
20.6	393.001272134799\\
20.8	391.958014696593\\
21	390.904909877513\\
21.2	389.854726513939\\
21.4	388.769992123632\\
21.6	387.700761509576\\
21.8	386.610108373514\\
22	385.510576453872\\
22.2	384.425594127879\\
22.4	383.314159713598\\
22.6	382.194141540405\\
22.8	381.065675645099\\
23	379.950742914637\\
23.2	378.805426128347\\
23.4	377.662618693407\\
23.6	376.501093354464\\
23.8	375.341940556638\\
24	374.174671752831\\
24.2	372.994415701581\\
24.4	371.811244026867\\
24.6	370.591424716099\\
24.8	369.393014384064\\
25	368.177556436604\\
};
\end{axis}
\end{tikzpicture}%
        \vspace*{-0.3in}
        \caption{Courier supply $N$ as a function of $p_t$.}
        \label{fig:tax_supply}
    \end{minipage}
    \begin{minipage}[b]{0.005\linewidth}
        \hfill
    \end{minipage}
    \begin{minipage}[b]{0.32\linewidth}
        \centering
%
%
\begin{tikzpicture}

\begin{axis}[%
width=1.694in,
height=1.03in,
at={(0.429in,0.457in)},
scale only axis,
xmin=0,
xmax=18,
xtick={0, 6, 12, 18},
xticklabels={0, 6, 12, 18},
xlabel style={font=\color{white!15!black}},
xlabel={Tax Rate $p_t$},
ymin=0,
ymax=6700,
scaled y ticks=base 10:-3,
ytick={0, 3000, 6000},
yticklabels={0.0, 3.0, 6.0},
ylabel style={font=\color{white!15!black}},
ylabel={Profit (HK\$/hr)},
axis background/.style={fill=white}
]
\addplot [color=black, line width=1.0pt, forget plot]
  table[row sep=crcr]{%
0	6469.64302791507\\
0.2	6418.96206078567\\
0.4	6368.63795288284\\
0.6	6317.6264008602\\
0.8	6266.92745669371\\
1	6215.52891118401\\
1.2	6163.44056577528\\
1.4	6111.61600240899\\
1.6	6058.16172210041\\
1.8	6005.87371837829\\
2	5952.87403389319\\
2.2	5899.1729239514\\
2.4	5845.64873593509\\
2.6	5792.26668110516\\
2.8	5736.4779310407\\
3	5683.322869637\\
3.2	5627.78259669279\\
3.4	5571.54399107374\\
3.6	5516.17729404848\\
3.8	5459.31049799567\\
4	5403.25092742214\\
4.2	5345.71332906103\\
4.4	5288.19687035908\\
4.6	5229.96153811522\\
4.8	5173.10334253731\\
5	5113.42271587769\\
5.2	5055.05872637741\\
5.4	4995.938627932\\
5.6	4936.07449136941\\
5.8	4876.10436562944\\
6	4814.77415657507\\
6.2	4754.53390917771\\
6.4	4692.34488957262\\
6.6	4631.16385906284\\
6.8	4568.08518706401\\
7	4505.93442275484\\
7.2	4442.4712175026\\
7.4	4378.78935802437\\
7.6	4315.37145867918\\
7.8	4249.67591004716\\
8	4184.7281289449\\
8.2	4120.42785923518\\
8.4	4054.38578499127\\
8.6	3987.58161905203\\
8.8	3920.46213500364\\
9	3852.58277389939\\
9.2	3785.18552561166\\
9.4	3716.98350335225\\
9.6	3646.83457344017\\
9.8	3577.84936594339\\
10	3508.05468327483\\
10.2	3437.11505727824\\
10.4	3366.0961818967\\
10.6	3294.28868501367\\
10.8	3222.6551984662\\
11	3149.58556865701\\
11.2	3076.62914354644\\
11.4	3002.57194561167\\
11.6	2927.19062939632\\
11.8	2852.38922175984\\
12	2777.2699832844\\
12.2	2701.06832634337\\
12.4	2624.06083010729\\
12.6	2546.04590005482\\
12.8	2468.51798176751\\
13	2389.73459772564\\
13.2	2309.96104990623\\
13.4	2230.12177007608\\
13.6	2149.61434824963\\
13.8	2067.96098530122\\
14	1985.9634887432\\
14.2	1903.27891802507\\
14.4	1819.76998982741\\
14.6	1735.68447885065\\
14.8	1650.87595399628\\
15	1565.14197978081\\
15.2	1478.78757623822\\
15.4	1391.69685542372\\
15.6	1303.99432477761\\
15.8	1215.32433291472\\
16	1126.05360558923\\
16.2	1035.85303746345\\
16.4	945.009630816481\\
16.6	853.278120973288\\
16.8	760.730053180656\\
17	667.354619653638\\
17.2	573.1609658163\\
17.4	478.08866664097\\
17.6	382.189075036257\\
17.8	285.409810782284\\
18	187.759715802168\\
18.2	89.2817734133178\\
18.4	-10.2277133868683\\
18.6	-110.783161479609\\
18.8	-212.000877945672\\
19	-314.270507154082\\
19.2	-417.377068647828\\
19.4	-521.878103293163\\
19.6	-626.953731678253\\
19.8	-733.296762121007\\
20	-840.49786465943\\
20.2	-948.789452544391\\
20.4	-1057.92404469549\\
20.6	-1168.44888708002\\
20.8	-1279.83025106706\\
21	-1392.35911228854\\
21.2	-1505.72673267401\\
21.4	-1620.93891891505\\
21.6	-1736.67034691643\\
21.8	-1853.96858507783\\
22	-1972.5076006892\\
22.2	-2091.50530416175\\
22.4	-2212.3405892246\\
22.6	-2334.45056543825\\
22.8	-2457.84976392554\\
23	-2581.64657410342\\
23.2	-2707.63877565552\\
23.4	-2834.48139512065\\
23.6	-2963.13041823722\\
23.8	-3092.62757466695\\
24	-3223.4622768187\\
24.2	-3355.90893580626\\
24.4	-3489.45485559036\\
24.6	-3626.02475676236\\
24.8	-3762.34684666435\\
25	-3900.62707773651\\
};
\end{axis}

\end{tikzpicture}%
        \vspace*{-0.3in}
        \caption{The platform's profit $\Pi$ as a function of $p_t$.}
        \label{fig:tax_profit}
    \end{minipage}
    
    \begin{minipage}[b]{0.32\linewidth}
        \centering
%
%
\definecolor{mycolor1}{rgb}{0.85000,0.32500,0.09800}%
\begin{tikzpicture}

\begin{axis}[%
width=1.694in,
height=1.03in,
at={(0.497in,0.457in)},
scale only axis,
xmin=0,
xmax=18,
xtick={0, 6, 12, 18},
xticklabels={0, 6, 12, 18},
xlabel style={font=\color{white!15!black}},
xlabel={Tax Rate $p_t$},
ymin=16000,
ymax=19500,
scaled y ticks=base 10:-4,
ytick={16000, 17000, 18000, 19000},
yticklabels={1.6, 1.7, 1.8, 1.9},
ylabel style={font=\color{white!15!black}},
ylabel={Social Welfare},
axis background/.style={fill=white}
]
\addplot [color=black, line width=1.0pt, forget plot]
  table[row sep=crcr]{%
0	16346.0964256434\\
0.2	16392.587290121\\
0.4	16439.3629957041\\
0.6	16485.5945304282\\
0.8	16532.0674218085\\
1	16577.9806690161\\
1.2	16623.3403261748\\
1.4	16668.8911618204\\
1.6	16713.1477421888\\
1.8	16758.2928968761\\
2	16802.8548308114\\
2.2	16846.8396099329\\
2.4	16890.9259568858\\
2.6	16935.0805408629\\
2.8	16977.3313378796\\
3	17021.5629178602\\
3.2	17063.9045049441\\
3.4	17105.6532924754\\
3.6	17148.0064500106\\
3.8	17189.1461470245\\
4	17230.8275461244\\
4.2	17271.3079776421\\
4.4	17311.7220745792\\
4.6	17351.505685074\\
4.8	17392.228431362\\
5	17430.7336393395\\
5.2	17470.1176733804\\
5.4	17508.8236894006\\
5.6	17546.8585875022\\
5.8	17584.6895594487\\
6	17621.3890994156\\
6.2	17658.749373087\\
6.4	17694.5411384821\\
6.6	17730.9143772686\\
6.8	17765.7548839213\\
7	17801.0972554073\\
7.2	17835.3235908922\\
7.4	17869.2126246586\\
7.6	17903.0956084524\\
7.8	17935.1769541273\\
8	17967.5790154097\\
8.2	18000.2089346229\\
8.4	18031.4045910339\\
8.6	18061.8442571284\\
8.8	18091.827124813\\
9	18121.0452871114\\
9.2	18150.3203460114\\
9.4	18178.780342715\\
9.6	18205.6793782785\\
9.8	18233.0418020022\\
10	18259.5674529535\\
10.2	18285.0428886204\\
10.4	18310.1354076402\\
10.6	18334.3904627249\\
10.8	18358.3853759917\\
11	18381.1358128424\\
11.2	18403.5597902411\\
11.4	18424.9327611845\\
11.6	18445.1385252052\\
11.8	18465.2333094838\\
12	18484.6827561687\\
12.2	18503.0686456384\\
12.4	18520.5394251086\\
12.6	18536.9969048713\\
12.8	18553.1394846803\\
13	18568.1221170958\\
13.2	18582.0825093983\\
13.4	18595.3924937018\\
13.6	18607.7740603172\\
13.8	18619.0577741408\\
14	18629.5078713688\\
14.2	18638.9868307716\\
14.4	18647.4514314889\\
14.6	18654.9378117241\\
14.8	18661.3943151978\\
15	18666.8051024493\\
15.2	18671.171125745\\
15.4	18674.4664381085\\
15.6	18676.6192856952\\
15.8	18677.7349287347\\
16	18677.612308897\\
16.2	18676.4639404639\\
16.4	18673.9343044199\\
16.6	18670.3085506824\\
16.8	18665.4677493313\\
17	18659.4762133771\\
17.2	18652.1600995347\\
17.4	18643.2854920569\\
17.6	18633.3943926309\\
17.8	18622.089248983\\
18	18609.450347442\\
18.2	18595.5872513122\\
18.4	18579.9701809268\\
18.6	18562.6585247136\\
18.8	18544.4819279007\\
19	18524.5785672062\\
19.2	18503.3761741234\\
19.4	18479.8667235518\\
19.6	18455.5126564926\\
19.8	18429.1028030344\\
20	18401.3180681204\\
20.2	18371.7712605688\\
20.4	18340.8319710966\\
20.6	18307.6698212356\\
20.8	18273.0720827422\\
21	18236.600157459\\
21.2	18198.6755296949\\
21.4	18157.898891565\\
21.6	18116.0833871081\\
21.8	18071.7928049197\\
22	18025.4625383935\\
22.2	17978.0936570888\\
22.4	17927.8766279566\\
22.6	17875.5366103408\\
22.8	17821.042168763\\
23	17765.4883967449\\
23.2	17706.615994499\\
23.4	17646.0827695322\\
23.6	17582.6993544048\\
23.8	17517.6172106358\\
24	17450.2211911755\\
24.2	17380.1676863774\\
24.4	17308.050476923\\
24.6	17231.5841393327\\
24.8	17154.6045678586\\
25	17074.4968024727\\
};
\addplot [color=mycolor1, dashed, line width=1.0pt, forget plot]
  table[row sep=crcr]{%
0	18677.7349287347\\
15.8	18677.7349287347\\
};
\addplot [color=mycolor1, dashed, line width=1.0pt, forget plot]
  table[row sep=crcr]{%
15.8	0\\
15.8	18677.7349287347\\
};
\addplot [color=mycolor1, only marks, mark size=2.0pt, mark=*, mark options={solid, mycolor1}, forget plot]
  table[row sep=crcr]{%
15.8	18677.7349287347\\
};
\node[right, align=left]
at (axis cs:8,19077.735) {$p^*_{t, SW} = 15.8$};
\end{axis}
\end{tikzpicture}%
        \vspace*{-0.3in}
        \caption{The social welfare as a function of $p_t$.}
        \label{fig:tax_social_welfare}
    \end{minipage}
    \begin{minipage}[b]{0.005\linewidth}
        \hfill
    \end{minipage}
    \begin{minipage}[b]{0.32\linewidth}
        \centering
%
%
\begin{tikzpicture}

\begin{axis}[%
width=1.694in,
height=1.03in,
at={(0.573in,0.457in)},
scale only axis,
xmin=0,
xmax=18,
xtick={0, 6, 12, 18},
xticklabels={0, 6, 12, 18},
xlabel style={font=\color{white!15!black}},
xlabel={Tax Rate $p_t$},
ymin=0.75,
ymax=0.83,
scaled y ticks=base 10:0,
ytick={0.75, 0.79, 0.83},
yticklabels={75\%, 79\%, 83\%},
ylabel style={font=\color{white!15!black}},
ylabel={Occupancy},
axis background/.style={fill=white}
]
\addplot [color=black, line width=1.0pt, forget plot]
  table[row sep=crcr]{%
0	0.827046985321677\\
0.2	0.826476637510633\\
0.4	0.825902098361596\\
0.6	0.825323325434937\\
0.8	0.824740271516167\\
1	0.824152895439876\\
1.2	0.823561150455008\\
1.4	0.82296498872126\\
1.6	0.822364357565806\\
1.8	0.821759217297183\\
2	0.821149521673725\\
2.2	0.820535213320725\\
2.4	0.819916245877593\\
2.6	0.819292567213579\\
2.8	0.818664131038362\\
3	0.818030839248384\\
3.2	0.817392770777004\\
3.4	0.816749743805332\\
3.6	0.816101740310801\\
3.8	0.815448697629444\\
4	0.814790602713794\\
4.2	0.814127353759765\\
4.4	0.813458911875732\\
4.6	0.812785217291363\\
4.8	0.812106212707134\\
5	0.81142183836532\\
5.2	0.810732036820638\\
5.4	0.810036736603343\\
5.6	0.809335890484855\\
5.8	0.808629422944019\\
6	0.807917280334966\\
6.2	0.807199390896395\\
6.4	0.806475674838447\\
6.6	0.805746134746941\\
6.8	0.805010621492506\\
7	0.804269103079854\\
7.2	0.803521506002461\\
7.4	0.802767751194659\\
7.6	0.802007784348602\\
7.8	0.801241524285737\\
8	0.800468901117866\\
8.2	0.79968982576766\\
8.4	0.798904256216788\\
8.6	0.798112087790733\\
8.8	0.797313256035487\\
9	0.796507677894023\\
9.2	0.79569526807831\\
9.4	0.794875958178717\\
9.6	0.79404966243806\\
9.8	0.793216291454626\\
10	0.792375770701638\\
10.2	0.79152800942055\\
10.4	0.790672922754701\\
10.6	0.789810417831499\\
10.8	0.788940413088526\\
11	0.788062813638632\\
11.2	0.787177534384083\\
11.4	0.786284470196545\\
11.6	0.785383535952334\\
11.8	0.784474639566358\\
12	0.783557675309207\\
12.2	0.782632545590376\\
12.4	0.781699154247879\\
12.6	0.780757398599263\\
12.8	0.779807173306924\\
13	0.778848380644392\\
13.2	0.777880905160173\\
13.4	0.776904647666012\\
13.6	0.775919495241966\\
13.8	0.774925341467098\\
14	0.773922068415188\\
14.2	0.772909561573171\\
14.4	0.771887717155928\\
14.6	0.770856402729756\\
14.8	0.769815509247625\\
15	0.768764911334254\\
15.2	0.767704494008589\\
15.4	0.766634125442626\\
15.6	0.765553676242789\\
15.8	0.764463026002979\\
16	0.763362046940614\\
16.2	0.762250596495208\\
16.4	0.761128555140959\\
16.6	0.75999577856288\\
16.8	0.758852129027451\\
17	0.757697472087567\\
17.2	0.756531657270845\\
17.4	0.755354553334545\\
17.6	0.754166004632876\\
17.8	0.752965864600965\\
18	0.751753986722323\\
18.2	0.750530217189984\\
18.4	0.749294401303263\\
18.6	0.748046384999473\\
18.8	0.74678600504323\\
19	0.745513102784433\\
19.2	0.744227512052347\\
19.4	0.742929072304458\\
19.6	0.741617610220993\\
19.8	0.740292952674914\\
20	0.73895493845639\\
20.2	0.737603378539937\\
20.4	0.736238097052083\\
20.6	0.734858921963133\\
20.8	0.733465661809538\\
21	0.732058133294039\\
21.2	0.730636143573514\\
21.4	0.729199512422349\\
21.6	0.727748034776347\\
21.8	0.726281520530172\\
22	0.724799765140542\\
22.2	0.72330256991487\\
22.4	0.72178973181546\\
22.6	0.720261041964377\\
22.8	0.718716281758575\\
23	0.717155250684727\\
23.2	0.715577719917841\\
23.4	0.713983482159389\\
23.6	0.7123723059442\\
23.8	0.710743974276064\\
24	0.709098251461867\\
24.2	0.707434907403726\\
24.4	0.705753710258726\\
24.6	0.704054419418031\\
24.8	0.702336803950312\\
25	0.700600614118723\\
};
\end{axis}
\end{tikzpicture}%
        \vspace*{-0.3in}
        \caption{Charger occupancy $\rho_c$ as a function of $p_t$.}
        \label{fig:tax_occupancy}
    \end{minipage}
    \begin{minipage}[b]{0.005\linewidth}
        \hfill
    \end{minipage}
    \begin{minipage}[b]{0.32\linewidth}
        \centering
%
%
\begin{tikzpicture}

\begin{axis}[%
width=1.694in,
height=1.03in,
at={(0.531in,0.457in)},
scale only axis,
xmin=0,
xmax=18,
xtick={0, 6, 12, 18},
xticklabels={0, 6, 12, 18},
xlabel style={font=\color{white!15!black}},
xlabel={Tax Rate $p_t$},
ymin=55,
ymax=120,
scaled y ticks=base 10:0,
ytick={60, 90, 120},
yticklabels={60, 90, 120},
ylabel style={font=\color{white!15!black}},
ylabel={Optimal Density},
axis background/.style={fill=white}
]
\addplot [color=black, line width=1.0pt, forget plot]
  table[row sep=crcr]{%
0	57.31\\
0.2	57.68\\
0.4	58.06\\
0.6	58.44\\
0.8	58.83\\
1	59.22\\
1.2	59.61\\
1.4	60.01\\
1.6	60.4\\
1.8	60.81\\
2	61.22\\
2.2	61.63\\
2.4	62.05\\
2.6	62.48\\
2.8	62.89\\
3	63.34\\
3.2	63.77\\
3.4	64.2\\
3.6	64.65\\
3.8	65.09\\
4	65.55\\
4.2	66\\
4.4	66.46\\
4.6	66.92\\
4.8	67.41\\
5	67.87\\
5.2	68.36\\
5.4	68.85\\
5.6	69.34\\
5.8	69.84\\
6	70.33\\
6.2	70.85\\
6.4	71.35\\
6.6	71.88\\
6.8	72.39\\
7	72.93\\
7.2	73.46\\
7.4	74\\
7.6	74.56\\
7.8	75.09\\
8	75.65\\
8.2	76.24\\
8.4	76.81\\
8.6	77.38\\
8.8	77.96\\
9	78.54\\
9.2	79.15\\
9.4	79.76\\
9.6	80.34\\
9.8	80.97\\
10	81.6\\
10.2	82.22\\
10.4	82.86\\
10.6	83.5\\
10.8	84.17\\
11	84.82\\
11.2	85.5\\
11.4	86.17\\
11.6	86.82\\
11.8	87.52\\
12	88.24\\
12.2	88.95\\
12.4	89.66\\
12.6	90.36\\
12.8	91.12\\
13	91.86\\
13.2	92.59\\
13.4	93.36\\
13.6	94.14\\
13.8	94.9\\
14	95.69\\
14.2	96.49\\
14.4	97.29\\
14.6	98.11\\
14.8	98.94\\
15	99.76\\
15.2	100.6\\
15.4	101.45\\
15.6	102.33\\
15.8	103.19\\
16	104.09\\
16.2	104.97\\
16.4	105.9\\
16.6	106.82\\
16.8	107.75\\
17	108.68\\
17.2	109.63\\
17.4	110.62\\
17.6	111.59\\
17.8	112.58\\
18	113.58\\
18.2	114.58\\
18.4	115.62\\
18.6	116.69\\
18.8	117.73\\
19	118.8\\
19.2	119.87\\
19.4	121\\
19.6	122.1\\
19.8	123.24\\
20	124.38\\
20.2	125.54\\
20.4	126.7\\
20.6	127.9\\
20.8	129.1\\
21	130.32\\
21.2	131.54\\
21.4	132.82\\
21.6	134.08\\
21.8	135.38\\
22	136.7\\
22.2	138\\
22.4	139.35\\
22.6	140.72\\
22.8	142.11\\
23	143.48\\
23.2	144.91\\
23.4	146.34\\
23.6	147.81\\
23.8	149.28\\
24	150.77\\
24.2	152.29\\
24.4	153.82\\
24.6	155.43\\
24.8	157\\
25	158.61\\
};
\end{axis}

\end{tikzpicture}%
        \vspace*{-0.3in}
        \caption{Optimal density $K^*_{\lambda}$ as a function of $p_t$.}
        \label{fig:tax_optimal_K}
    \end{minipage}
    
    \begin{minipage}[b]{0.32\linewidth}
        \centering
%
%
\begin{tikzpicture}

\begin{axis}[%
width=1.694in,
height=1.03in,
at={(0.573in,0.457in)},
scale only axis,
xmin=0,
xmax=18,
xtick={0, 6, 12, 18},
xticklabels={0, 6, 12, 18},
xlabel style={font=\color{white!15!black}},
xlabel={Tax Rate $p_t$},
ymin=80,
ymax=82,
scaled y ticks=base 10:0,
ytick={80, 81, 82},
yticklabels={80, 81, 82},
ylabel style={font=\color{white!15!black}},
ylabel={Price (HK\$)},
axis background/.style={fill=white}
]
\addplot [color=black, line width=1.0pt, forget plot]
  table[row sep=crcr]{%
0	80.2835958575708\\
0.2	80.2943728935851\\
0.4	80.3052705280025\\
0.6	80.3161716318961\\
0.8	80.3271893485845\\
1	80.3382114600979\\
1.2	80.3492408233357\\
1.4	80.3603840795807\\
1.6	80.3714342596567\\
1.8	80.3826985369968\\
2	80.3939729380198\\
2.2	80.4052604244884\\
2.4	80.4166560749235\\
2.6	80.4281565236627\\
2.8	80.4394953165006\\
3	80.4512030492436\\
3.2	80.4627525765572\\
3.4	80.4743271256743\\
3.6	80.4860860428454\\
3.8	80.4977900938959\\
4	80.5096717969887\\
4.2	80.5215050475474\\
4.4	80.5334409222959\\
4.6	80.5454073963966\\
4.8	80.5576086622576\\
5	80.569641351715\\
5.2	80.58190364474\\
5.4	80.5941990137365\\
5.6	80.6065312365063\\
5.8	80.6189609082893\\
6	80.6313751942694\\
6.2	80.6439981032311\\
6.4	80.6565572092417\\
6.6	80.6693168253425\\
6.8	80.6820244693453\\
7	80.6949268022623\\
7.2	80.7078322984896\\
7.4	80.7208363305645\\
7.6	80.7339799454457\\
7.8	80.7470565740215\\
8	80.7603152719097\\
8.2	80.7737455135494\\
8.4	80.7871600830813\\
8.6	80.8006412346452\\
8.8	80.8142238275028\\
9	80.8278777352275\\
9.2	80.8416940944565\\
9.4	80.8555797199988\\
9.6	80.8694617935055\\
9.8	80.8835547088983\\
10	80.8977244409517\\
10.2	80.9119537678309\\
10.4	80.9263131010164\\
10.6	80.9407592120279\\
10.8	80.9553508588784\\
11	80.9699978893599\\
11.2	80.9847883054031\\
11.4	80.9996595294664\\
11.6	81.0146078875911\\
11.8	81.0297242874571\\
12	81.0449646050769\\
12.2	81.0602996196518\\
12.4	81.0757453117828\\
12.6	81.0913002178563\\
12.8	81.107009796691\\
13	81.1228247554141\\
13.2	81.1387597590439\\
13.4	81.1548332736746\\
13.6	81.1710343779475\\
13.8	81.1873652544868\\
14	81.2038318095089\\
14.2	81.2204334231279\\
14.4	81.2371760350294\\
14.6	81.2540567585035\\
14.8	81.2710794198025\\
15	81.2882595599846\\
15.2	81.3055834050833\\
15.4	81.3230591877839\\
15.6	81.3406718618435\\
15.8	81.3584676381548\\
16	81.3763934706062\\
16.2	81.3945143686268\\
16.4	81.4127532673698\\
16.6	81.4311836455423\\
16.8	81.4497852069919\\
17	81.468577074401\\
17.2	81.4875332866336\\
17.4	81.5066221912395\\
17.6	81.5259442921057\\
17.8	81.5454381069314\\
18	81.5651245510566\\
18.2	81.5850271120472\\
18.4	81.6050711412019\\
18.6	81.6252757406286\\
18.8	81.6457705593636\\
19	81.6664351664697\\
19.2	81.6873385509978\\
19.4	81.7083468250461\\
19.6	81.7296721166603\\
19.8	81.7511545025894\\
20	81.7728937913935\\
20.2	81.7948462778469\\
20.4	81.8170678748968\\
20.6	81.8394584856409\\
20.8	81.8621282416339\\
21	81.8850288128962\\
21.2	81.9082209432581\\
21.4	81.9315388414076\\
21.6	81.9552163919108\\
21.8	81.9790849798369\\
22	82.003207384238\\
22.2	82.0277113413919\\
22.4	82.0523897755178\\
22.6	82.0773386169827\\
22.8	82.1025635565972\\
23	82.1281992492189\\
23.2	82.1539952122289\\
23.4	82.1801509954933\\
23.6	82.2065415806664\\
23.8	82.2333058391048\\
24	82.2603844204538\\
24.2	82.2877496093039\\
24.4	82.3154767035707\\
24.6	82.3433294499846\\
24.8	82.371731402539\\
25	82.4004119775315\\
};
\end{axis}
\end{tikzpicture}%
        \vspace*{-0.3in}
        \caption{Valet charging price $p_v$ as a function of $p_t$.}
        \label{fig:tax_price}
    \end{minipage}
    \begin{minipage}[b]{0.005\linewidth}
        \hfill
    \end{minipage}
    \begin{minipage}[b]{0.32\linewidth}
        \centering
%
%
\begin{tikzpicture}

\begin{axis}[%
width=1.694in,
height=1.03in,
at={(0.454in,0.457in)},
scale only axis,
xmin=0,
xmax=18,
xtick={0, 6, 12, 18},
xticklabels={0, 6, 12, 18},
xlabel style={font=\color{white!15!black}},
xlabel={Tax Rate $p_t$},
ymin=58,
ymax=66,
scaled y ticks=base 10:0,
ytick={58, 62, 66},
yticklabels={58, 62, 66},
ylabel style={font=\color{white!15!black}},
ylabel={Marginal Cost},
axis background/.style={fill=white}
]
\addplot [color=black, line width=1.0pt, forget plot]
  table[row sep=crcr]{%
0	58.8077947801122\\
0.2	58.8753565915652\\
0.4	58.9406836474195\\
0.6	59.0073499145905\\
0.8	59.0718676356371\\
1	59.1377505462164\\
1.2	59.2049618679206\\
1.4	59.2701416877588\\
1.6	59.3399924069007\\
1.8	59.4045579119608\\
2	59.4705021760117\\
2.2	59.5378049627154\\
2.4	59.6032623917224\\
2.6	59.6669595813229\\
2.8	59.7383013185323\\
3	59.7985846500743\\
3.2	59.8664682753965\\
3.4	59.9356965500754\\
3.6	60.0002874076732\\
3.8	60.0692575930486\\
4	60.1337238050771\\
4.2	60.2025224258286\\
4.4	60.2698151836597\\
4.6	60.3385036495218\\
4.8	60.4001578227679\\
5	60.4716498567351\\
5.2	60.5362354700455\\
5.4	60.6023015903925\\
5.6	60.6698129386131\\
5.8	60.7360952232685\\
6	60.8064539029111\\
6.2	60.8704027379044\\
6.4	60.9409990399996\\
6.6	61.0053632653173\\
6.8	61.0762717378535\\
7	61.141108995526\\
7.2	61.2099332197838\\
7.4	61.277769154132\\
7.6	61.3422740687072\\
7.8	61.4154457749996\\
8	61.4829454724957\\
8.2	61.5449807250684\\
8.4	61.613235245807\\
8.6	61.6829713786924\\
8.8	61.7519459658586\\
9	61.8224029130513\\
9.2	61.887758564449\\
9.4	61.9546952718984\\
9.6	62.0296086741984\\
9.8	62.0953306832183\\
10	62.1626527835419\\
10.2	62.2336117313791\\
10.4	62.3019773008881\\
10.6	62.3719015045486\\
10.8	62.4373824830475\\
11	62.5084435305376\\
11.2	62.5751820714392\\
11.4	62.6454758213358\\
11.6	62.721115442677\\
11.8	62.7888092070911\\
12	62.8544285626759\\
12.2	62.9235593928904\\
12.4	62.9943061825267\\
12.6	63.068419805463\\
12.8	63.1334464239154\\
13	63.2036987748818\\
13.2	63.2772886189972\\
13.4	63.345619691215\\
13.6	63.4139515482275\\
13.8	63.4872930943154\\
14	63.5572988674797\\
14.2	63.6273479789655\\
14.4	63.6990765836628\\
14.6	63.7692923167796\\
14.8	63.8396371099967\\
15	63.9132136329025\\
15.2	63.9853811014161\\
15.4	64.0577164284447\\
15.6	64.1272965658715\\
15.8	64.2015377222454\\
16	64.2716666051997\\
16.2	64.34638512535\\
16.4	64.4157574847014\\
16.6	64.4883014024079\\
16.8	64.5612055822083\\
17	64.6358395965095\\
17.2	64.7095035658338\\
17.4	64.7796550458686\\
17.6	64.8542137290214\\
17.8	64.9279501550987\\
18	65.0021830186166\\
18.2	65.0781712096585\\
18.4	65.1509750265905\\
18.6	65.2219797679923\\
18.8	65.2984231338757\\
19	65.373087154234\\
19.2	65.4495493389855\\
19.4	65.5209263024018\\
19.6	65.5975829992838\\
19.8	65.6715867330508\\
20	65.7474570744056\\
20.2	65.8229942901066\\
20.4	65.9003718612694\\
20.6	65.9753764239177\\
20.8	66.0522601244658\\
21	66.1289619747319\\
21.2	66.207530713564\\
21.4	66.2819883462024\\
21.6	66.3603515486786\\
21.8	66.4367397857712\\
22	66.5131563883997\\
22.2	66.5933887791337\\
22.4	66.6708762989796\\
22.6	66.7484982833467\\
22.8	66.8262886588421\\
23	66.9077939312213\\
23.2	66.9859822101409\\
23.4	67.0661546862389\\
23.6	67.1449145192403\\
23.8	67.2256603113495\\
24	67.3067343771501\\
24.2	67.3873699306243\\
24.4	67.4692038722279\\
24.6	67.5468063750351\\
24.8	67.6295722029696\\
25	67.7113031350359\\
};
\end{axis}
\end{tikzpicture}%
        \vspace*{-0.3in}
        \caption{The marginal cost $C_m$ as a function of $p_t$.}
        \label{fig:tax_marginal_cost}
    \end{minipage}
    \begin{minipage}[b]{0.005\linewidth}
        \hfill
    \end{minipage}
    \begin{minipage}[b]{0.32\linewidth}
        \centering
%
%
\begin{tikzpicture}

\begin{axis}[%
width=1.694in,
height=1.03in,
at={(0.573in,0.457in)},
scale only axis,
xmin=0,
xmax=18,
xtick={0, 6, 12, 18},
xticklabels={0, 6, 12, 18},
xlabel style={font=\color{white!15!black}},
xlabel={Tax Rate $p_t$},
ymin=0.20,
ymax=0.27,
scaled y ticks=base 10:0,
ytick={0.20, 0.23, 0.26},
yticklabels={0.20, 0.23, 0.26},
ylabel style={font=\color{white!15!black}},
ylabel={Lerner Index},
axis background/.style={fill=white}
]
\addplot [color=black, line width=1.0pt, forget plot]
  table[row sep=crcr]{%
0	0.267499242504762\\
0.2	0.266756131596005\\
0.4	0.266042150659752\\
0.6	0.265311720969071\\
0.8	0.264609304586878\\
1	0.263890128104374\\
1.2	0.263154682468066\\
1.4	0.262445763959219\\
1.6	0.261678070653928\\
1.8	0.260978307606588\\
2	0.260261683772479\\
2.2	0.259528485469809\\
2.4	0.258819437403731\\
2.6	0.258133442810313\\
2.8	0.257351117340015\\
3	0.256709876501512\\
3.2	0.255972902263866\\
3.4	0.255219662085827\\
3.6	0.254525963956887\\
3.8	0.253777556837507\\
4	0.253086958835096\\
4.2	0.252342310414101\\
4.4	0.251617533121277\\
4.6	0.250875927008829\\
4.8	0.250224046793656\\
5	0.249448690074777\\
5.2	0.248761412525937\\
5.4	0.248056282809344\\
5.6	0.247333782908945\\
5.8	0.246627660056785\\
6	0.245871055077419\\
6.2	0.245196119121163\\
6.4	0.244438379859128\\
6.6	0.243760011041121\\
6.8	0.243000257621707\\
7	0.242317808338207\\
7.2	0.241586206981682\\
7.4	0.240868009553445\\
7.6	0.240192616415567\\
7.8	0.239409479666921\\
8	0.238698545622433\\
8.2	0.238057114551602\\
8.4	0.237338765436932\\
8.6	0.236602947251806\\
8.8	0.235877757142511\\
9	0.235135145876693\\
9.2	0.234457426236782\\
9.4	0.233761040531201\\
9.6	0.232966223608775\\
9.8	0.232287318396171\\
10	0.231589600163411\\
10.2	0.230847744574893\\
10.4	0.230139432855175\\
10.6	0.229412942110381\\
10.8	0.228742982142236\\
11	0.228004876374689\\
11.2	0.227321780042681\\
11.4	0.226595813053432\\
11.6	0.225804863121679\\
11.8	0.225113873220838\\
12	0.224449922719338\\
12.2	0.223743809384643\\
12.4	0.223019092328077\\
12.6	0.222254179720559\\
12.8	0.221603082370186\\
13	0.220888831651987\\
13.2	0.22013487996476\\
13.4	0.219447355925215\\
13.6	0.218761322506249\\
13.8	0.218015107457785\\
14	0.2173115793775\\
14.2	0.21660910564844\\
14.4	0.215887605987242\\
14.6	0.215186356709434\\
14.8	0.214485182604311\\
15	0.213746068880471\\
15.2	0.21302598884721\\
15.4	0.212305623174745\\
15.6	0.211620766118195\\
15.8	0.210880691512228\\
16	0.210192736934021\\
16.2	0.209450592285221\\
16.4	0.20877559228157\\
16.6	0.208063809030263\\
16.8	0.207349590693996\\
17	0.20661631861471\\
17.2	0.205896890531498\\
17.4	0.205222185580508\\
17.6	0.204496012991274\\
17.8	0.20378194461402\\
18	0.203064013248361\\
18.2	0.202327026008321\\
18.4	0.201630804121729\\
18.6	0.200958536725306\\
18.8	0.200222832285011\\
19	0.199510949376242\\
19.2	0.198779755835416\\
19.4	0.198112202138967\\
19.6	0.197383504663394\\
19.8	0.196689182769025\\
20	0.195974924867763\\
20.2	0.195267216879238\\
20.4	0.194540044358042\\
20.6	0.193843927553683\\
20.8	0.193127987956798\\
21	0.19241694198052\\
21.2	0.19168637834987\\
21.4	0.191007647561674\\
21.6	0.190285201233041\\
21.8	0.189589151890246\\
22	0.18889567237605\\
22.2	0.188159858538805\\
22.4	0.187459664716891\\
22.6	0.186760932943606\\
22.8	0.18606331198446\\
23	0.185324960940775\\
23.2	0.184629036760834\\
23.4	0.183912978087413\\
23.6	0.183216892133172\\
23.8	0.182500817334508\\
24	0.181784344294719\\
24.2	0.18107652414151\\
24.4	0.18035821969186\\
24.6	0.179693038571351\\
24.8	0.178971097833631\\
25	0.178264992734514\\
};
\end{axis}
\end{tikzpicture}%
        \vspace*{-0.3in}
        \caption{Lerner Index $L$ as a function of $p_t$.}
        \label{fig:tax_lerner_index}
    \end{minipage}
\end{figure*}

\begin{itemize}
    \item The proposed taxation scheme can promote EV penetration and benefit social welfare. Based on \autoref{prop:existence_of_r_hat}, we identify that in this study $r < \hat{r}$ such that setting an appropriate tax rate can facilitate the market expansion. The valet charging demand peaks at $p^*_{t,\lambda} = \HKD13.2$, where $\lambda$ increases by $2.47\%$ and the tax revenue contributes 268 additional public chargers. Furthermore, the social welfare also increases and is maximized at $p^*_{t,SW} = \HKD15.8$. It implies that the taxation scheme will enforce the platform to share part of its profit to provide more charging facilities, which reduces the deadweight loss associated with monopoly pricing. The difference between $p^*_{t,\lambda}$ and $p^*_{t,SW}$ suggests a trade-off between environmental and economic concerns. The city planner may decide the tax rate according to her objective.

    \item The tax regulation can effectively curb the platform's market power. As shown in \autoref{fig:tax_lerner_index}, the markup reduces to 22.01\% at $p_t = p^*_{t,\lambda}$. The underlying reason is that imposing a valet charging tax is equivalent to directly increasing the marginal cost by $p_t$. Yet the elasticity of demand keeps the platform from raising the price proportionally.\footnote{Elasticity of demand equals the reciprocal of the Lerner Index: $E_d = -\dfrac{1}{L} = \dfrac{p_v}{C_m - p_v}$.} This leads to the weaker market power. Therefore, the tax burden primarily falls on the platform instead of customers. At $p_t = p^*_{t,\lambda}$, compared with the case of no regulation, i.e., $p_t = 0$, the price slightly rises by $1.07\%$ (from $\HKD80.28$ to $\HKD81.14$). As opposed to customers, the platform suffers a dramatic profit decline of $64.3\%$ (from $\HKD6469.64/\hr$ to $\HKD2309.96/\hr$), which reveals that the platform is more vulnerable than customers under the proposed taxation scheme.
    
    \item After implementing the proposed regulation, the city planner needs to jointly replan charging infrastructure to achieve higher EV adoption. According to \autoref{fig:tax_optimal_K}, $K^*_{\lambda}$ always increases with $p_t$. Intuitively, the platform will be reluctant to serve more consumers under a higher tax rate due to the unprofitably high marginal cost, which can be directly reflected by the decreasing charger occupancy. The effect of $p_t$ on the occupancy $\rho_c$ is given by
    \begin{equation} \label{eqn:occupancy_derivative}
        \dfrac{\partial \rho_c}{\partial p_t} = \underbrace{\dfrac{t_c}{M} \dfrac{\partial \lambda}{\partial p_t}}_{\text{negative}} + \underbrace{ \bigg(-\dfrac{\lambda}{t_c M^2} \dfrac{\partial M}{\partial p_t} \bigg)}_{\text{positive}} < 0,
    \end{equation}
    
    where the first item captures the effect associated with valet charging demand, and the second item captures the effect associated with charging facility supply. Hence, as illustrated in \autoref{fig:tax_occupancy}, $\rho_c$ goes down in response to the increasing $p_t$. In contrast to the platform, the city planner aims to promote EV adoption. To this end, an effective solution to reducing the marginal cost is to increase the charging station density  (recall \autoref{fig:valet_marginal_cost}), which induces the platform to serve more customers so as to mitigate the profit decline. This explains why $C_m(p_t) - C_m(0) < p_t$ in \autoref{fig:tax_marginal_cost}: denser charging stations partly offset the marginal cost increase associated with the valet charging tax. However, $K$ cannot increase without bounds as the platform will be forced to quite the valet charging market due to prohibitively high fixed cost.
    
    \item As shown in \autoref{fig:tax_supply}, the courier fleet size reduces as the tax rate increases. This is because the regulatory agency tends to build denser charging stations to relieve the platform's tax burden. The higher density will shorten the delivery time and increase the labor efficiency of couriers, thereby inducing the platform to downsize the courier fleet. In this sense, the proposed taxation scheme can facilitate EV penetration at the cost of decreased platform profit and reduced job opportunities.
\end{itemize}

\section{Sensitivity Analysis} \label{section:sensitivity_analysis}

This section presents a sensitivity analysis to verify the robustness of the economic insights derived from the preceding sections.

One of the key observations from our numerical results is that the optimal charging station densities for different stakeholders are not consistent. To investigate how these optimal densities relate to each other under different model parameters, we perturb the model parameters by 50\% in both directions from their nominal value. The tested parameters include $\lambda_0$, $N_0$, $M_0$, $C$, $t_c$, $\alpha$, $\beta$, $\theta$, $\phi$, $\epsilon_1$, $\epsilon_2$, and $\theta$. \autoref{fig:sa_valet_curve_supply}-\ref{fig:sa_valet_curve_demand} present the curves of $N(K)$ and $\lambda(K)$ under varying model parameters, which supports our assumption in \autoref{prop:different_optimal_K} that both $N(K)$ and $\lambda(K)$ first increase and then decrease with $K$. Furthermore, \autoref{fig:sa_valet_comparison} report the comparison among $K^*_N$, $K^*_\lambda$, and $K^*_\Pi$ under varying model parameters. It is clear that $K^*_N < K^*_\lambda < K^*_\Pi$ holds well for extensive parameter perturbation. Specifically, as discussed in \autoref{section:analysis_valet}, the optimal density for the platform is sensitive to the fixed cost $C$. $K^*_\Pi$ becomes smaller than $K^*_\lambda$ only when $C$ is scaled up from the nominal value by 180\%, i.e., raised from \HKD60 to \HKD168. Such a high fixed cost is unrealistic since it indicates that the wage of dedicated coordinators is much higher than the average courier wage. This leads to an unprofitable market, which forces the platform to give up the business. Hence, we conclude that $K^*_N < K^*_\lambda < K^*_\Pi$ is robust to parameter variation in the practical regime.

Another important finding is that the proposed taxation scheme can facilitate the expansion of valet charging market, and the charging infrastructure needs to be replanned to enable such expansion. To understand how the public charger cost $r$ affects this result, we perturb $r$ by $-50\%$ to $+150\%$ from its nominal value. \autoref{fig:sa_tax_perturb_r} shows the optimal tax rate $p^*_{t,\lambda}$, the optimal charging station density $K^*(p^*_{t,\lambda})$, and the maximum valet charging demand $\lambda^*(p^*_{t,\lambda})$ under distinct value of $r$. The results are consistent with \autoref{prop:existence_of_r_hat}. When $r$ is below a certain threshold $\hat{r}$, the taxation scheme can always lead to a higher valet charging demand and a denser charging station network. If $r$ goes beyond the threshold, imposing the tax will result in reduced valet charging demand. In this case, the market degrades to the scenario without the tax regulation since the taxation scheme will not be implemented. Fixing other model parameters in this study, the threshold $\hat{r}$ is found to be \HKD43.75, which is as high as $175\%$ of the nominal value of $r$ and unrealistic in practice. Thereby the effectiveness of the proposed taxation scheme is insensitive to a wide range of $r$.

\section{Conclusion} \label{section:conlusion}

To conclude, this paper explores on-demand valet charging services for EVs, an innovative business model that offers an affordable and convenient EV charging solution for those who do not have a private charger. Valet charging facilitates EV adoption by unleashing the potential of public charging infrastructure. A queuing network model is developed to capture the stochastic matching between customers, couriers, and charging stations. An economic equilibrium model is formulated to predict the market outcome. By analyzing the solution property, we find that the profit-maximizing market equilibrium is always in the normal regime rather than the WGC regime. Based on the mathematical model, we present two case studies to examine how public policies will affect the valet charging market.

At first, we focus on charging infrastructure planning. Numerical results reveal that the optimal charging station densities for different stakeholders are not consistent. Couriers prefer a lower density to shorten their idle time and improve their income. The platform prefers a higher density to reduce the marginal cost and increase its profit. While customers prefer a density in-between that trades off the EV charging convenience and congestion. We also find that the market power associated with monopoly pricing results in an inefficient market outcome, which calls for government interventions.

Second, to regulate the platform and weaken its market power, we propose a taxation scheme that imposes a per-service tax on the platform and then invests the tax revenue in public charging facilities. We show that enforcing a proper tax rate can squeeze the market power and meanwhile prompt the platform to serve more customers as the profit decline can be partly alleviated by the market expansion. Moreover, to sustain the effectiveness of the proposed regulation, the city planner should redeploy a denser charging station network.

This paper also opens avenues that merit further research. Future studies may consider the spatial imbalance between charging demand courier supply, time-variant charging demand, spatiotemporal pricing, and synergy between valet charging and other charging services (such as public charger, battery swapping, vehicle-to-vehicle emergency charging, among others). Another direction lies in the intersection of EV charging and power system. Possible topics include integrating V2G into valet charging, adaptive charging scheduling under stochastic arrival and departure, joint planning of charging and electric infrastructure, etc.

\section*{Acknowledgments} 
This research was supported by Hong Kong Research Grants Council under project HKUST26200420.

\bibliographystyle{unsrt}
\bibliography{valet_ref}

\newpage
\appendix

\section{\bf{Evaluation of the Approximation Performance}} 
\label{appendix:approximation_evalution}
We use the mean waiting time given by \eqref{eqn:original_response_time} as the ground truth to assess the approximation performance of \eqref{eqn:response_time} at different service rates $\mu$, utilization rates $\rho$, and service capacity $N$. Here the service rate $\mu$ refers to the expected number of customers completing the service process per unit time. The mean service time is then $1/\mu$. The results are reported in the figure below.

\begin{figure}[htbp]
    \centering
    \subfigure[Service Rate $\mu=2.4$]{
%
%
\definecolor{mycolor1}{rgb}{0.00000,0.44700,0.74100}%
\definecolor{mycolor2}{rgb}{0.85000,0.32500,0.09800}%
\begin{tikzpicture}

\begin{axis}[%
width=2.275in,
height=1.501in,
at={(0.499in,0.459in)},
scale only axis,
xmin=0.6,
xmax=0.9,
xtick={0.6, 0.7, 0.8, 0.9},
xticklabels={0.6,0.7,0.8,0.9},
xlabel style={font=\color{white!15!black}},
xlabel={Utilization Rate $\rho$},
ymin=0,
ymax=1.7,
ytick={0, 0.5, 1, 1.5},
yticklabels={0.0, 0.5, 1.0, 1.5},
ylabel style={font=\color{white!15!black}},
ylabel={Mean Waiting Time (min)},
axis background/.style={fill=white},
legend style={nodes={scale=0.9, transform shape}, at={(0.0,1.0)}, anchor=north west, legend cell align=left, align=left, draw=white!15!black}
]
\addplot [color=black, line width=1.0pt]
  table[row sep=crcr]{%
0.6	0.000173395921221738\\
0.61	0.000269876928303999\\
0.62	0.000413865138026746\\
0.63	0.000625708637618602\\
0.64	0.000933149769259172\\
0.65	0.00137351893149512\\
0.66	0.00199642762507504\\
0.67	0.00286704010312065\\
0.68	0.00407001285609344\\
0.69	0.00571420464894576\\
0.7	0.00793827900418296\\
0.71	0.0109173487085734\\
0.72	0.0148708519002031\\
0.73	0.0200719066618868\\
0.74	0.0268584727396972\\
0.75	0.0356467646301322\\
0.76	0.0469475234536195\\
0.77	0.0613859855143426\\
0.78	0.0797267125717798\\
0.79	0.102904917124799\\
0.8	0.132066594364788\\
0.81	0.168620769825754\\
0.82	0.214308663904973\\
0.83	0.271296851216368\\
0.84	0.342305043808915\\
0.85	0.43078480188135\\
0.86	0.54117479081222\\
0.87	0.679273960329662\\
0.88	0.852801589903588\\
0.89	1.0722632631119\\
0.9	1.35233705517954\\
};
\addlegendentry{$N = 60$, ground truth}

\addplot [color=black, dashed, line width=1.0pt]
  table[row sep=crcr]{%
0.6	0.0061542904082244\\
0.61	0.00745220587664036\\
0.62	0.00900543575933519\\
0.63	0.0108615151416778\\
0.64	0.0130765784568693\\
0.65	0.0157169343792625\\
0.66	0.0188609483609038\\
0.67	0.0226013000824782\\
0.68	0.0270477000087949\\
0.69	0.0323301710732374\\
0.7	0.0386030299036942\\
0.71	0.0460497392043999\\
0.72	0.0548888520716879\\
0.73	0.0653813345771986\\
0.74	0.0778396412042696\\
0.75	0.0926390377507866\\
0.76	0.110231831354482\\
0.77	0.13116539692645\\
0.78	0.156105212853305\\
0.79	0.185864581083692\\
0.8	0.221443376991184\\
0.81	0.264079162187775\\
0.82	0.315315475142472\\
0.83	0.377094380566312\\
0.84	0.451883899325092\\
0.85	0.542856605929894\\
0.86	0.654144987925084\\
0.87	0.791214912267782\\
0.88	0.961426107659343\\
0.89	1.17489868792273\\
0.9	1.4458999621873\\
};
\addlegendentry{$N = 60$, approximation}

\addplot [color=mycolor1, line width=1.0pt]
  table[row sep=crcr]{%
0.6	3.39764508034067e-06\\
0.61	6.43244804569278e-06\\
0.62	1.19025506676997e-05\\
0.63	2.15444738175884e-05\\
0.64	3.81783413851756e-05\\
0.65	6.62864220382983e-05\\
0.66	0.000112846613105993\\
0.67	0.000188507739229286\\
0.68	0.000309213293017745\\
0.69	0.000498401693065382\\
0.7	0.000789934128946595\\
0.71	0.00123192585576745\\
0.72	0.00189168440197354\\
0.73	0.00286199079343818\\
0.74	0.00426900160941752\\
0.75	0.00628210705352347\\
0.76	0.00912616335658315\\
0.77	0.0130966417639175\\
0.78	0.0185784231893237\\
0.79	0.0260692499424009\\
0.8	0.0362092726701316\\
0.81	0.0498187770279261\\
0.82	0.0679471596824982\\
0.83	0.0919377412377516\\
0.84	0.123515380864764\\
0.85	0.164907658823335\\
0.86	0.219016628074654\\
0.87	0.289668669763859\\
0.88	0.381988396399431\\
0.89	0.502975987929248\\
0.9	0.662430852354264\\
};
\addlegendentry{$N = 90$, ground truth}

\addplot [color=mycolor1, dashed, line width=1.0pt]
  table[row sep=crcr]{%
0.6	0.00117647579867172\\
0.61	0.00148335192200341\\
0.62	0.00186523252642703\\
0.63	0.00233943637772253\\
0.64	0.00292711590141262\\
0.65	0.00365409043187595\\
0.66	0.00455186359714381\\
0.67	0.00565886800850405\\
0.68	0.00702199159575721\\
0.69	0.00869845437231985\\
0.7	0.010758123221501\\
0.71	0.013286376972669\\
0.72	0.0163876666820602\\
0.73	0.0201899596010035\\
0.74	0.0248503140072659\\
0.75	0.0305619119453508\\
0.76	0.0375629867952332\\
0.77	0.0461482355139214\\
0.78	0.0566835209420597\\
0.79	0.0696249775621607\\
0.8	0.0855440807626777\\
0.81	0.105160897992692\\
0.82	0.129388727780067\\
0.83	0.159394843116943\\
0.84	0.196684416034541\\
0.85	0.243218476824679\\
0.86	0.30158296597918\\
0.87	0.375236437192996\\
0.88	0.468882345436873\\
0.89	0.589045264600532\\
0.9	0.744993859382064\\
};
\addlegendentry{$N = 90$, approximation}

\addplot [color=mycolor2, line width=1.0pt]
  table[row sep=crcr]{%
0.6	7.94221803809069e-08\\
0.61	1.82898165133513e-07\\
0.62	4.08357488137909e-07\\
0.63	8.84940501912869e-07\\
0.64	1.86333120377647e-06\\
0.65	3.81602712895214e-06\\
0.66	7.60861629307635e-06\\
0.67	1.47838004683708e-05\\
0.68	2.80190499512461e-05\\
0.69	5.18438387542093e-05\\
0.7	9.37339863172503e-05\\
0.71	0.00016573896736756\\
0.72	0.000286843028037661\\
0.73	0.000486312247838788\\
0.74	0.000808337231126262\\
0.75	0.0013183458413344\\
0.76	0.00211143551275415\\
0.77	0.00332346742983564\\
0.78	0.00514548896985667\\
0.79	0.00784232982770768\\
0.8	0.0117764894858188\\
0.81	0.0174388598093975\\
0.82	0.0254885022721325\\
0.83	0.0368047780190094\\
0.84	0.0525568631527926\\
0.85	0.0742985015432663\\
0.86	0.104100514222372\\
0.87	0.144741497749252\\
0.88	0.199990984294331\\
0.89	0.275044466389376\\
0.9	0.377217377317689\\
};
\addlegendentry{$N = 120$, ground truth}

\addplot [color=mycolor2, dashed, line width=1.0pt]
  table[row sep=crcr]{%
0.6	0.00030717856368802\\
0.61	0.000400756222363879\\
0.62	0.000521142034878049\\
0.63	0.000675597461242571\\
0.64	0.000873261440332418\\
0.65	0.00112562105870859\\
0.66	0.00144709938375745\\
0.67	0.00185579037782065\\
0.68	0.00237437894293432\\
0.69	0.00303129471205953\\
0.7	0.00386216202002052\\
0.71	0.00491162668420472\\
0.72	0.00623566437370177\\
0.73	0.00790450765940396\\
0.74	0.0100063724671931\\
0.75	0.0126522241391818\\
0.76	0.0159819052619956\\
0.77	0.0201720616301653\\
0.78	0.0254464638536247\\
0.79	0.0320895525566929\\
0.8	0.0404643695196562\\
0.81	0.0510365302230993\\
0.82	0.0644066332631004\\
0.83	0.081354634273815\\
0.84	0.102901481453447\\
0.85	0.130396141435049\\
0.86	0.165640796818088\\
0.87	0.211074870939077\\
0.88	0.270052316200327\\
0.89	0.347271659079032\\
0.9	0.449465903934624\\
};
\addlegendentry{$N = 120$, approximation}

\end{axis}

\end{tikzpicture}%
    }
    \subfigure[Service Rate $\mu=0.2$]{
%
%
\definecolor{mycolor1}{rgb}{0.00000,0.44700,0.74100}%
\definecolor{mycolor2}{rgb}{0.85000,0.32500,0.09800}%
\begin{tikzpicture}

\begin{axis}[%
width=2.275in,
height=1.501in,
at={(0.456in,0.458in)},
scale only axis,
xmin=0.6,
xmax=0.9,
xtick={0.6, 0.7, 0.8, 0.9},
xticklabels={0.6,0.7,0.8,0.9},
xlabel style={font=\color{white!15!black}},
xlabel={Utilization Rate $\rho$},
ymin=0,
ymax=20,
ytick={ 0, 10, 20},
yticklabels={ 0, 10, 20},
ylabel style={font=\color{white!15!black}},
ylabel={Mean Waiting Time (min)},
axis background/.style={fill=white},
legend style={nodes={scale=0.9, transform shape}, at={(0.0,1.0)}, anchor=north west, legend cell align=left, align=left, draw=white!15!black}
]
\addplot [color=black, line width=1.0pt]
  table[row sep=crcr]{%
0.6	0.00208075105466085\\
0.61	0.00323852313964799\\
0.62	0.00496638165632095\\
0.63	0.00750850365142322\\
0.64	0.0111977972311101\\
0.65	0.0164822271779415\\
0.66	0.0239571315009005\\
0.67	0.0344044812374478\\
0.68	0.0488401542731213\\
0.69	0.0685704557873491\\
0.7	0.0952593480501957\\
0.71	0.131008184502881\\
0.72	0.178450222802437\\
0.73	0.240862879942642\\
0.74	0.322301672876368\\
0.75	0.427761175561586\\
0.76	0.563370281443436\\
0.77	0.736631826172112\\
0.78	0.956720550861359\\
0.79	1.23485900549759\\
0.8	1.58479913237746\\
0.81	2.02344923790905\\
0.82	2.57170396685969\\
0.83	3.25556221459643\\
0.84	4.107660525707\\
0.85	5.16941762257622\\
0.86	6.49409748974664\\
0.87	8.15128752395598\\
0.88	10.2336190788431\\
0.89	12.8671591573429\\
0.9	16.2280446621545\\
};
\addlegendentry{$N = 60$, ground truth}

\addplot [color=black, dashed, line width=1.0pt]
  table[row sep=crcr]{%
0.6	0.0738514848986929\\
0.61	0.0894264705196843\\
0.62	0.108065229112023\\
0.63	0.130338181700134\\
0.64	0.156918941482432\\
0.65	0.18860321255115\\
0.66	0.226331380330846\\
0.67	0.271215600989739\\
0.68	0.32457240010554\\
0.69	0.387962052878849\\
0.7	0.463236358844331\\
0.71	0.5525968704528\\
0.72	0.658666224860256\\
0.73	0.784576014926385\\
0.74	0.934075694451239\\
0.75	1.11166845300944\\
0.76	1.32278197625378\\
0.77	1.5739847631174\\
0.78	1.87326255423966\\
0.79	2.23037497300431\\
0.8	2.65732052389422\\
0.81	3.16894994625331\\
0.82	3.78378570170968\\
0.83	4.52513256679575\\
0.84	5.42260679190112\\
0.85	6.51427927115875\\
0.86	7.84973985510101\\
0.87	9.49457894721343\\
0.88	11.5371132919122\\
0.89	14.0987842550728\\
0.9	17.3507995462475\\
};
\addlegendentry{$N = 60$, approximation}

\addplot [color=mycolor1, line width=1.0pt]
  table[row sep=crcr]{%
0.6	4.07717409640881e-05\\
0.61	7.71893765483133e-05\\
0.62	0.000142830608012396\\
0.63	0.000258533685811062\\
0.64	0.00045814009662211\\
0.65	0.000795437064459582\\
0.66	0.00135415935727192\\
0.67	0.00226209287075144\\
0.68	0.00371055951621296\\
0.69	0.00598082031678461\\
0.7	0.00947920954735918\\
0.71	0.0147831102692094\\
0.72	0.0227002128236825\\
0.73	0.0343438895212584\\
0.74	0.0512280193130102\\
0.75	0.0753852846422816\\
0.76	0.109513960278998\\
0.77	0.15715970116701\\
0.78	0.222941078271884\\
0.79	0.312830999308811\\
0.8	0.43451127204158\\
0.81	0.597825324335113\\
0.82	0.815365916189982\\
0.83	1.10325289485302\\
0.84	1.48218457037718\\
0.85	1.97889190588002\\
0.86	2.62819953689585\\
0.87	3.47602403716631\\
0.88	4.5838607567932\\
0.89	6.03571185515099\\
0.9	7.94917022825117\\
};
\addlegendentry{$N = 90$, ground truth}

\addplot [color=mycolor1, dashed, line width=1.0pt]
  table[row sep=crcr]{%
0.6	0.0141177095840607\\
0.61	0.0178002230640409\\
0.62	0.0223827903171243\\
0.63	0.0280732365326704\\
0.64	0.0351253908169515\\
0.65	0.0438490851825115\\
0.66	0.0546223631657258\\
0.67	0.0679064161020488\\
0.68	0.0842638991490867\\
0.69	0.104381452467838\\
0.7	0.129097478658012\\
0.71	0.159436523672028\\
0.72	0.196652000184722\\
0.73	0.242279515212042\\
0.74	0.298203768087191\\
0.75	0.366742943344209\\
0.76	0.4507558415428\\
0.77	0.553778826167057\\
0.78	0.680202251304717\\
0.79	0.835499730745929\\
0.8	1.02652896915213\\
0.81	1.2619307759123\\
0.82	1.55266473336081\\
0.83	1.91273811740331\\
0.84	2.36021299241451\\
0.85	2.91862172189615\\
0.86	3.61899559175017\\
0.87	4.50283724631596\\
0.88	5.62658814524251\\
0.89	7.0685431752064\\
0.9	8.93992631258477\\
};
\addlegendentry{$N = 90$, approximation}

\addplot [color=mycolor2, line width=1.0pt]
  table[row sep=crcr]{%
0.6	9.53066164570883e-07\\
0.61	2.19477798160215e-06\\
0.62	4.90028985765491e-06\\
0.63	1.06192860229544e-05\\
0.64	2.23599744453177e-05\\
0.65	4.57923255474257e-05\\
0.66	9.13033955169162e-05\\
0.67	0.000177405605620449\\
0.68	0.000336228599414953\\
0.69	0.000622126065050512\\
0.7	0.001124807835807\\
0.71	0.00198886760841072\\
0.72	0.00344211633645194\\
0.73	0.00583574697406546\\
0.74	0.00970004677351522\\
0.75	0.0158201500960128\\
0.76	0.0253372261530498\\
0.77	0.0398816091580277\\
0.78	0.0617458676382801\\
0.79	0.0941079579324927\\
0.8	0.141317873829826\\
0.81	0.209266317712771\\
0.82	0.305862027265592\\
0.83	0.441657336228114\\
0.84	0.630682357833515\\
0.85	0.8915820185192\\
0.86	1.24920617066846\\
0.87	1.73689797299103\\
0.88	2.39989181153198\\
0.89	3.30053359667253\\
0.9	4.52660852781227\\
};
\addlegendentry{$N = 120$, ground truth}

\addplot [color=mycolor2, dashed, line width=1.0pt]
  table[row sep=crcr]{%
0.6	0.00368614276425625\\
0.61	0.00480907466836655\\
0.62	0.00625370441853661\\
0.63	0.00810716953491087\\
0.64	0.010479137283989\\
0.65	0.0135074527045031\\
0.66	0.0173651926050894\\
0.67	0.0222694845338478\\
0.68	0.028492547315212\\
0.69	0.0363755365447145\\
0.7	0.0463459442402464\\
0.71	0.0589395202104567\\
0.72	0.0748279724844215\\
0.73	0.0948540919128477\\
0.74	0.120076469606317\\
0.75	0.151826689670182\\
0.76	0.191782863143948\\
0.77	0.242064739561985\\
0.78	0.305357566243497\\
0.79	0.385074630680317\\
0.8	0.485572434235877\\
0.81	0.612438362677193\\
0.82	0.772879599157208\\
0.83	0.976255611285783\\
0.84	1.23481777744137\\
0.85	1.56475369722059\\
0.86	1.98768956181705\\
0.87	2.53289845126894\\
0.88	3.24062779440395\\
0.89	4.16725990894841\\
0.9	5.39359084721549\\
};
\addlegendentry{$N = 120$, approximation}

\end{axis}
\end{tikzpicture}%
    }    
    \caption{Approximation performance under varying service rate, server capacity, and utilization rate.}
    \label{fig:approximation_evaluation}
\end{figure}
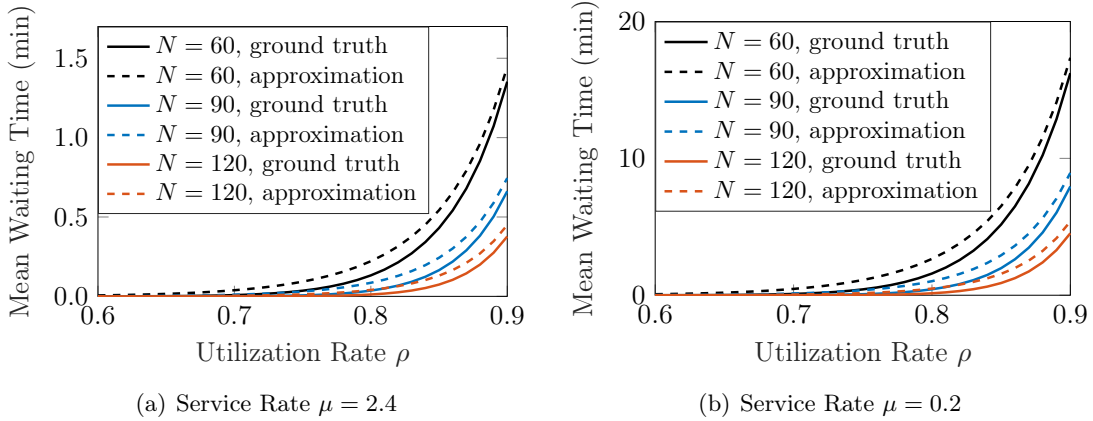

It can be directly observed from \autoref{fig:approximation_evaluation} that \eqref{eqn:response_time} approximates the ground truth remarkably well in different scenarios. The error is negligible over a wide range of $N$ and $\rho$. The percentage of error is slightly higher in light traffic (e.g., $\rho < 0.7$). This is because the ground truth of the expected waiting time is very close to zero. We argue that this would not affect the results derived in the simulation as the absolute value of approximation error is negligible. Thus, \eqref{eqn:response_time} is sufficient to accurately approximate the mean waiting time without causing analytical mistakes.

\section{\bf{Model Parameters Setting}} 
\label{appendix:parameters_calibration}
We use Hong Kong as the prototype of this study. Most of the parameters are selected according to official data and parameters in previous studies. Specifically, we set $A = 1\e 3\text{km}^2$, $\lambda_0 = 6\e 5$, and $M_0 = 3\e 3$ based on the square area of Hong Kong \cite{GovHK}, the current private vehicle stock \cite{TDHK}, and the number of public chargers \cite{EPDHK}, respectively. The number of potential couriers is specified as 125\% of the taxi driver population in Hong Kong, i.e., $N_0 = 5\e 4$ \cite{HKTC}. The per-time percentage of charging demand over all EVs $\tau$ is set to be 1\%. The average charging time is $t_c = 5\hr$ since over $80\%$ of the existing public chargers in Hong Kong are standard or medium level \cite{EPDHK}.\footnote{Standard charger means a charger employing a household type 13A socket. Medium charger means a charger with rated power of or below 20kW.}

We assume that valet charging customers are less sensitive to the pickup time than ride-hailing customers do and consequently tune $\alpha$ to be 50\% of the value of time (VOT) in \cite{zhou_competitive_2020}, i.e., $\alpha = \HKD 60/\hr$. Further, the VOT assigned to $t_w$ and $t_d$ is set as $\beta = \alpha/6 = \HKD 10/\hr$. As couriers pick up EVs by public transit, bicycle, or scooter, the pickup time of valet charging services is typically longer than that of ride-hailing services. In this light, $\phi$ is set to be $25\%$ of that in ride-hailing service, i.e., $\phi = 0.04$ \cite{zhou_competitive_2020}. The parameter $\theta = 0.06$ in \eqref{eqn:delivery_time} is obtained by Monte-Carlo experiment. The procedure is as follows.

\begin{algorithm}[ht]
\SetAlgoLined
\caption{Monte-Carlo simulation to calibrate $\theta$}
Initialize the lower bound $K_{lb}$ and upper bound $K_{ub}$ for $K$;\\
Initialize the average speed $v$;\\
Initialize the sample size $n$;\\
\For{$K \leftarrow K_{lb} \text{ to } K_{ub}$}{
 Suppose the charging station locates at the center of each zone. Calculate its location:
 \[(x_c, y_c) \leftarrow \left(0.5\sqrt{A/K}, 0.5\sqrt{A/K}\right)\]\\
 Randomly generate $n$ customer locations $(x_i, y_i)$ that evenly distributed in the zone\;
 Calculate the average Manhattan distance between customers and charging station:
 \[d_k \leftarrow \dfrac{1}{n}\sum_{i = 1}^n \left | x_i - x_c \right | + \left | y_i - y_c \right |\]\\
 Calculate the average travel time by $t_d^k = \dfrac{d_k}{v}$ and store the data pair $(\sqrt{A/K}, t_d^k)$;
}
Perform linear regression over $(\sqrt{A/K}, t_d^k)$;\\
Return the slope $\theta$.
\end{algorithm}

Parameters in \eqref{eqn:supply_logit} are set as $w_0 = \HKD 110/\hr$ and $\eta = 0.1$ such that the average wage of couriers is close to that of taxi drivers in Hong Kong. Likewise, parameters in \eqref{eqn:demand_logit} are set as $c_s = \HKD 80$, $c_0 = \HKD 75$, $\epsilon_1 = 0.11$, and $\epsilon_2 = 0.1$ such that the occupancy of public chargers maintains around $80\%$. The fixed cost $C = \HKD 60$ is assumed to be close to the hourly wage of couriers.

\section{Proof of \autoref{prop:optimal_condition}}
\label{appendix:proof_prop_optimal_condition}
\begin{proof}
Suppose the optimal demand and supply are $\lambda^*$ and $N^*$. The optimal courier wage, denoted by $w^*$, is then uniquely determined by $w(N^*)$ due to the one-to-one mapping. At $(\lambda^*, N^*)$, equation \eqref{eqn:fixed_point_equation} has two positive roots when strict inequality holds in \eqref{eqn:root_existence_condition}. \autoref{prop:optimal_condition} says that the platform will receive a higher profit by choosing the larger root of \eqref{eqn:fixed_point_equation}. To prove this, we can compare the profits under two different roots.

Given $(\lambda^*, N^*)$, let $(p_v^*, N_i^*, t_r^*, t_p^*, t_w^*)$ and $(p_v^{**}, N_i^{**}, t_r^{**}, t_p^{**}, t_w^{**})$ denote the market outcomes under two distinct roots of $\eqref{eqn:fixed_point_equation}$. Without loss of generality, we assume
\[N_i^* > \dfrac{N^* - 2 \lambda^* t_d}{3} > N_i^{**}.\]
Since $t_p$ decreases with $N_i$, clearly we have $t_p^* < t_p^{**}$. The response time $t_r$ is increasing in the service time of delivery queue, i.e., $t_p + t_d$. And $t_d$ is exogenous as it only depends $K$. Hence, we also have $t_r^* < t_r^{**}$. Besides, the waiting time $t_w$ only depends on the demand level, which indicates $t_w^* = t_w^{**}$. Furthermore, the valet charging costs are the same in both cases, i.e.,
\[p_v^* + \alpha (t_r^* + t_p^*) + \beta (2 t_d + t_w^*) = p_v^{**} + \alpha (t_r^{**} + t_p^{**}) + \beta (2 t_d + t_w^{**}).\]
With $t_r^* < t_r^{**}$, $t_p^* < t_p^{**}$, and $t_w^* = t_w^{**}$, we must have $p_v^* > p_v^{**}$, which implies
\[p_v^* \lambda^* - N^* w^* - KC > p_v^{**} \lambda^* - N^* w^* - KC.\]
This shows choosing the larger positive root of \eqref{eqn:fixed_point_equation} will lead to a higher profit. Therefore, $N_i \geq (N - 2 \lambda t_d)/3$ always holds at the profit-maximizing optimum.

\end{proof}

\section{Definitions of Marginal Cost and Marginal Revenue}
\label{appendix:definition_of_Cm_Rm}
Let $c_v(\lambda)$ and $w(N)$ represent the inverse demand and supply model, respectively. The first-order optimality condition for \eqref{eqn:valet_profit_maximization_unconstrained} reads as follows:
\begin{align}
    \dfrac{\partial \Pi}{\partial \lambda} &= p_v + \lambda \left[c_v^\prime(\lambda) - \alpha \left(\dfrac{\partial t_r}{\partial \lambda} + \dfrac{\partial t_p}{\partial \lambda}\right) - \beta \dfrac{\partial t_w}{\partial \lambda} \right] = 0, \label{eqn:FOC_1}\\
    \dfrac{\partial \Pi}{\partial N} &= -\alpha \lambda \left(\dfrac{\partial t_r}{\partial N} + \dfrac{\partial t_p}{\partial N}\right) - \Big(w(N) + N w^\prime(N)\Big) = 0, \label{eqn:FOC_2}
\end{align}
where $p_v$ is given by
\[p_v = c_v(\lambda) - \alpha (t_r + t_p) - \beta (2 t_d + t_w).\]
Equation \eqref{eqn:FOC_2} implies the relation between $\lambda$ and $N$ at the optimum. This relation can be summarized as $N = N(\lambda)$, which transforms \eqref{eqn:valet_profit_maximization_unconstrained} into a univariate maximization problem. As such, the marginal cost and marginal revenue can be defined as follows:
\begin{equation} \label{eqn:marginal_cost}
    C_m = \dfrac{\partial N w}{\partial \lambda} = \Big(w(N) + N w^\prime(N)\Big) \dfrac{\partial N}{\partial \lambda},
\end{equation}
\begin{align} 
    \begin{split}
        R_m = \dfrac{\partial \lambda p_v}{\partial \lambda} ={}& c_v(\lambda) - \alpha (t_r + t_p) - \beta (2 t_d + t_w) \\
        & + \lambda \left[c_v^\prime(\lambda) - \alpha \left(\dfrac{\partial t_r}{\partial \lambda} + \dfrac{\partial t_r}{\partial N}\dfrac{\partial N}{\partial \lambda} + \dfrac{\partial t_p}{\partial \lambda} + \dfrac{\partial t_p}{\partial N}\dfrac{\partial N}{\partial \lambda}\right) - \beta \dfrac{\mathrm{d } t_w}{\mathrm{d } \lambda} \right],
    \end{split} \label{eqn:marginal_revenue}
\end{align}
where $\dfrac{\partial N}{\partial \lambda}$ is derived by applying implicit function theorem on \eqref{eqn:FOC_2}.

\section{Proof of \autoref{prop:different_optimal_K}}
\label{appendix:proof_prop_different_optimal_K}
\begin{proof}
    Based on the first-order condition \eqref{eqn:FOC_2}, we have
    \begin{equation} \label{eqn:FOC_rewritten}
        w(N) + N w^\prime(N) = -\alpha \lambda \left( \dfrac{\partial t_p}{\partial N} + \dfrac{\partial t_r}{\partial N}\right)
    \end{equation}    
    at the optimum. The left-hand side (LHS) of \eqref{eqn:FOC_rewritten} is independently of $\lambda$ and $K$, but increasing in $N$ since $N w(N)$ is convex. We next proceed to analyze the monotonicity of the right-hand side (RHS) of \eqref{eqn:FOC_rewritten}.
    \begin{enumerate}
        \item [(i)] The expression of $\dfrac{\partial t_p}{\partial N}$ at the RHS is as follows:
        \[\dfrac{\partial t_p}{\partial N} = \dfrac{-\frac{1}{2} t_p}{N_i - \lambda t_p} = \dfrac{-\frac{1}{2} t_p}{N - \lambda (3 t_p + 2 t_d)} < 0.\]
        Based on \autoref{prop:optimal_condition} and \autoref{lemma:WGC}, at the optimum, $t_p$ increases with $\lambda$ and decreases with $N$. Hence, $\dfrac{\partial t_p}{\partial N}$ is increasing in $N$ but decreasing in $\lambda$. In addition, the partial derivative of $N_i$ with respect to $t_d$ is
        \[\dfrac{\partial N_i}{\partial t_d} = \dfrac{-2\lambda}{1 + 2\lambda \dfrac{\mathrm{d} t_p}{\mathrm{d} N_i}} = \dfrac{-2 \lambda N_i}{N_i - \lambda t_p} < 0,\]
        which implies $\dfrac{\partial t_p}{\partial K} < 0$ since $\dfrac{\mathrm{d} t_p}{\mathrm{d} N_i} < 0$ and $\dfrac{\mathrm{d} t_d}{\mathrm{d} K} < 0$. Given this, we have $\dfrac{\partial t_p}{\partial N}$ is increasing in $K$.
        
        \item [(ii)] The expression of $\dfrac{\partial t_r}{\partial N}$ at the RHS is as follows:
        \begin{multline*}
            \dfrac{\partial t_r}{\partial N} = - \underbrace{
            \bigg( \dfrac{t_r}{1 - \rho_d} + (t_p + t_d) \dfrac{\sqrt{2N+2}}{N} \dfrac{\rho_d^{\sqrt{2N+2}-2}}{1 - \rho_d} \bigg)
            }_{\Psi_1 > 0} 
            \underbrace{
            \left( \dfrac{\rho_d}{N} + \dfrac{\lambda t_p}{N (N - 2\lambda t_d - 3\lambda t_p)} \right)
            }_{\Psi_2 > 0} \\
            - \underbrace{
             (t_p + t_d) \dfrac{\rho_d \log \dfrac{1}{\rho_d}}{1 - \rho_d} \dfrac{\rho_d^{\sqrt{2N+2}-2}}{N \sqrt{2N+2}}
            }_{\Psi_3 > 0}.
        \end{multline*}
        \begin{itemize}
            \item If $N$ increases while $\lambda$ and $K$ are fixed, $\rho_d$, $t_p$, and $t_r$ will decrease. $\Psi_1$, $\Psi_2$, and $\Psi_3$ will also decrease. Thus, $\dfrac{\partial t_r}{\partial N}$ will increase.
            \item If $\lambda$ increases while $N$ and $K$ are fixed, $\rho_d$, $t_p$, and $t_r$ will increase. $\Psi_1$, $\Psi_2$, and $\Psi_3$ will also increase. Thus, $\dfrac{\partial t_r}{\partial N}$ will decrease.
            \item If $K$ increases while $N$ and $\lambda$ are fixed, $t_d$, $\rho_d$, $t_p$, and $t_r$ will decrease. $\Psi_1$, $\Psi_2$, and $\Psi_3$ will also decrease. Thus, $\dfrac{\partial t_r}{\partial N}$ will increase.
        \end{itemize}
        Overall, $\dfrac{\partial t_r}{\partial N}$ is increasing in $N$ and $K$ but decreasing in $\lambda$.
    \end{enumerate}
    Based on the above analysis, we pin down that the RHS of \eqref{eqn:FOC_rewritten} is decreasing in $N$ and $K$ but increasing in $\lambda$. Let $\nabla_+ N(K)$ and $\nabla_+ \lambda(K)$ denote the right-hand derivatives of $N$ and $\lambda$ with respect to $K$. As $N$ is maximized at $K = K^*_N$, clearly we have $\nabla_+ N(K^*_N) = 0$. When $K$ slightly increases from $K = K^*_N$, $N$ and also $w(N) + N w^\prime(N)$ remain unchanged. To guarantee the equality of \eqref{eqn:FOC_rewritten}, $\lambda$ must increase, i.e., $\nabla_+ \lambda(K^*_N) > 0$. This suffices to show $K^*_{\lambda} > K^*_N$.
\end{proof}

\section{Proof of \autoref{prop:existence_of_r_hat}}
\label{appendix:proof_prop_effect_of_cost}
\begin{proof}
    Under the taxation scheme, the marginal cost \eqref{eqn:marginal_cost} becomes
    \begin{equation} \label{eqn:marginal_cost_tax}
        C_m = \dfrac{\partial}{\partial \lambda}(\lambda p_t + N w) = p_t + \Big(w(N) + N w^\prime(N)\Big) \dfrac{\partial N}{\partial \lambda},
    \end{equation}
    Obviously, the marginal cost $C_m$ depends on $\lambda$, $p_t$, and $K$. When $\lambda$ and $K$ are fixed, imposing the tax is equivalent to raising the marginal cost by $p_t$, i.e.,
    \[C_m(\lambda, p_t, K) - C_m(\lambda, 0, K) = p_t.\]
    And the marginal revenue \eqref{eqn:marginal_revenue} becomes
    \begin{align} 
        \begin{split}
            R_m ={}& \dfrac{\partial \lambda p_v}{\partial \lambda} \\
            ={}& c_v(\lambda) - \alpha (t_r + t_p) - \beta (2 t_d + t_w) \\
            & + \lambda \left[c_v^\prime(\lambda) - \alpha \left(\dfrac{\partial t_r}{\partial \lambda} + \dfrac{\partial t_r}{\partial N}\dfrac{\partial N}{\partial \lambda} + \dfrac{\partial t_p}{\partial \lambda} + \dfrac{\partial t_p}{\partial N}\dfrac{\partial N}{\partial \lambda}\right) - \beta \left( \dfrac{\partial t_w}{\partial \lambda} + \dfrac{\partial t_w}{\partial M} \dfrac{\mathrm{d } M}{\mathrm{d } \lambda} \right) \right],
        \end{split} \label{eqn:marginal_revenue_tax}
    \end{align}
    which also depends on $\lambda$, $p_t$, and $K$. For simplicity, we let $\widetilde{R}_m = -\beta t_w -\beta \lambda \left(\dfrac{\partial t_w}{\partial \lambda} + \dfrac{\partial t_w}{\partial M} \dfrac{\mathrm{d } M}{\mathrm{d } \lambda} \right)$ denote the terms in \eqref{eqn:marginal_revenue_tax} that depend on $p_t$. When $\lambda$ and $K$ are fixed, imposing the tax is equivalent to changing the marginal revenue by
    \[R_m(\lambda, p_t, K) - R_m(\lambda, 0, K) = \widetilde{R}_m(\lambda, p_t, K) - \widetilde{R}_m(\lambda, 0, K),\]
    where $\widetilde{R}_m(\lambda, 0, K) = -\beta \left(t_w + \lambda \dfrac{\mathrm{d } t_w}{\mathrm{d } \lambda}\right) \Big|_{M = M_0}$ since \eqref{eqn:marginal_revenue_tax} is reduced to \eqref{eqn:marginal_revenue} at $p_t = 0$. In view of the fact that the platform's profit is maximized when the marginal cost equals the marginal revenue, we show the existence of $\hat{r}$ by comparing $C_m$ and $R_m$ in two extreme cases.
    \begin{enumerate}
        \item [(i)] Fix $\lambda = \lambda^*(0)$ and $K = K^*(0)$. When $r$ goes to zero, the tax revenue can contribute an infinite number of public chargers. Since $\lambda$ is strictly upper bounded by the maximal labor supply $N_0$, we have $\lim_{r \rightarrow 0} t_w = 0$ and therefore
        \[\lim_{r \rightarrow 0} \widetilde{R}_m(\lambda^*(0), p_t, K^*(0)) = 0, \quad \forall p_t > 0.\]
        In this case, charging the tax $p_t > 0$ will change $C_m$ and $R_m$ by
        \begin{align}
            \Delta C_m &= C_m(\lambda^*(0), p_t, K^*(0)) - C_m(\lambda^*(0), 0, K^*(0)) = p_t > 0, \\
            \begin{split}
                \Delta R_m &= \widetilde{R}_m(\lambda^*(0), p_t, K^*(0)) - \widetilde{R}_m(\lambda^*(0), 0, K^*(0))\\
                &= \beta \left( t_w + \lambda \dfrac{\mathrm{d } t_w}{\mathrm{d } \lambda} \right) \Bigg |_{\lambda = \lambda^*(0), K = K^*(0), M = M_0} > 0.
            \end{split}
        \end{align}
        Clearly, as $r \rightarrow 0$, $\Delta R_m$ is a constant for any $p_t > 0$. Additionally, the platform's profit becomes $\Pi(0, K^*(0)) + (\beta t_w - p_t)\lambda^*(0)$. Knowing that $\Pi(0, K^*(0)) \geq 0$, the profitability constraint \eqref{eqn:tax_constraint_profitability} yields $p_t \leq \beta t_w$. Therefore, we can always find some $p_t$ satisfying
        \[0 < p_t < \min \left\{\Delta R_m, \beta t_w \right\} \]
        such that $\Delta C_m < \Delta R_m$ and thus
        \[R_m(\lambda^*(0), p_t, K^*(0)) > C_m(\lambda^*(0), p_t, K^*(0)).\]
        This implies that the platform can improve its profit by further increasing $\lambda$, i.e., $\lambda(p_t, K^*(0)) > \lambda^*(0)$. Finally, we let $K = K^*(p_t)$ and obtain
        \[\lambda^*(p_t) = \lambda(p_t, K^*(p_t)) \geq \lambda(p_t, K^*(0)) > \lambda^*(0).\]
        
        \item [(ii)] Fix $\lambda = \hat{\lambda} = \lambda(0, K^*(p_t))$ and $K = K^*(p_t)$, where $p_t > 0$ is any arbitrary tax rate. When $r$ goes to infinity, under any finite $p_t$ and $\lambda$, we have $\lim_{r \rightarrow \infty} M_0 + \lambda p_t/r = M_0$ and 
        \[\lim_{r \rightarrow \infty} \widetilde{R}_m(\hat{\lambda}, p_t, K^*(p_t)) = \widetilde{R}_m(\hat{\lambda}, 0, K^*(p_t)) = -\beta \left( t_w + \lambda \dfrac{\mathrm{d } t_w}{\mathrm{d } \lambda} \right) \Bigg|_{\lambda = \hat{\lambda}, K = K^*(p_t), M = M_0}.\]
        In this case, any $p_t > 0$ will change $C_m$ and $R_m$ by
        \begin{align}
            \Delta C_m &= C_m(\hat{\lambda}, p_t, K^*(p_t)) - C_m(\hat{\lambda}, 0, K^*(p_t)) = p_t > 0, \\
            \Delta R_m &= \widetilde{R}_m(\hat{\lambda}, p_t, K^*(p_t)) - \widetilde{R}_m(\hat{\lambda}, 0, K^*(p_t)) = 0,
        \end{align}
        meaning that
        \[R_m(\hat{\lambda}, p_t, K^*(p_t)) < C_m(\hat{\lambda}, p_t, K^*(p_t)).\]
        This implies that the platform can improve its profit by decreasing $\lambda$, i.e., $\lambda^*(p_t) < \hat{\lambda} = \lambda(0, K^*(p_t))$. Finally, we let $K = K^*(0)$ and obtain
        \[\lambda^*(p_t) < \lambda(0, K^*(p_t)) \leq \lambda(0, K^*(0)) = \lambda^*(0).\]
    \end{enumerate}
    Due to continuity, these two cases suffice to show the existence of $\hat{r}_2 \geq \hat{r}_1 > 0$ such that
    \begin{enumerate}
        \item [(i)] if $r \in (0, \hat{r}_1), \exists p_t > 0, \lambda^*(p_t) > \lambda^*(0)$,
        \item [(ii)] if $r \in (\hat{r}_2, +\infty), \forall p_t > 0, \lambda^*(p_t) < \lambda^*(0)$.
    \end{enumerate}
    The complement of the first statement is that if $r \in (\hat{r}_1, \infty)$, $\forall p_t > 0, \lambda^*(p_t) \leq \lambda^*(0)$. We next show that the equality will not hold and therefore $\hat{r} = \hat{r}_1 = \hat{r}_2$.
    
    Assume that when $r = \tilde{r}_1 > \hat{r}_1$, there exists $\tilde{p}_t > 0$ satisfying $\lambda^*(\tilde{p}_t) = \lambda^*(0)$. It follows that
    \[C_m(\lambda^*(0), \tilde{p}_t, K^*(\tilde{p}_t)) \big|_{r = \tilde{r}_1} = R_m(\lambda^*(0), \tilde{p}_t, K^*(\tilde{p}_t)) \big|_{r = \tilde{r}_1}.\]
    Now fix $p_t = \tilde{p}_t$ and $K = K^*(\tilde{p}_t)$. Let $r = \tilde{r}_2 \in (\hat{r}_1, \tilde{r}_1)$. We first have
    \[C_m(\lambda^*(0), \tilde{p}_t, K^*(\tilde{p}_t)) \big|_{r = \tilde{r}_2} = C_m(\lambda^*(0), \tilde{p}_t, K^*(\tilde{p}_t)) \big|_{r = \tilde{r}_1}\]
    since $C_m$ is independent of $r$. While the marginal revenue increases, i.e.,
    \[R_m(\lambda^*(0), \tilde{p}_t, K^*(\tilde{p}_t)) \big|_{r = \tilde{r}_2} > R_m(\lambda^*(0), \tilde{p}_t, K^*(\tilde{p}_t)) \big|_{r = \tilde{r}_1},\]
    because $R_m$ is decreasing in $r$ when $\lambda$, $p_t$, and $K$ are fixed. Thus, we obtain
    \[R_m(\lambda^*(0), \tilde{p}_t, K^*(\tilde{p}_t)) \big|_{r = \tilde{r}_2} > C_m(\lambda^*(0), \tilde{p}_t, K^*(\tilde{p}_t)) \big|_{r = \tilde{r}_2},\]
    which implies $\lambda^*(\tilde{p}_t) > \lambda^*(0)$ and thereby $\tilde{r}_2 \in (0, \hat{r}_1)$. It contradicts our assumption that $\tilde{r}_2 > \hat{r}_1$. Hence, if $r \in (\hat{r}_1, \infty)$, $\forall p_t > 0, \lambda^*(p_t) < \lambda^*(0)$, and $\hat{r} = \hat{r}_1 = \hat{r}_2$ follows. The proof is completed.
\end{proof}

\section{Results of Sensitivity Analysis}
\label{appendix:sensitivity_analysis_results}

\begin{figure}[ht]
    \centering
    \subfigure[Perturbing $\lambda_0$]{\includestandalone[width=0.325\linewidth]{figure/fig_sa_curve_lambda0_supply}}
    \hfill
    \subfigure[Perturbing $N_0$]{\includestandalone[width=0.325\linewidth]{figure/fig_sa_curve_N0_supply}}
    \hfill
    \subfigure[Perturbing $M_0$]{\includestandalone[width=0.325\linewidth]{figure/fig_sa_curve_M0_supply}}

    \subfigure[Perturbing $t_c$]{\includestandalone[width=0.325\linewidth]{figure/fig_sa_curve_tc_supply}}
    \hfill
    \subfigure[Perturbing $\alpha$]{\includestandalone[width=0.325\linewidth]{figure/fig_sa_curve_alpha_supply}}    
    \hfill
    \subfigure[Perturbing $\beta$]{\includestandalone[width=0.325\linewidth]{figure/fig_sa_curve_beta_supply}}
    
    \subfigure[Perturbing $\theta$]{\includestandalone[width=0.325\linewidth]{figure/fig_sa_curve_theta_supply}}
    \hfill
    \subfigure[Perturbing $\phi$]{\includestandalone[width=0.325\linewidth]{figure/fig_sa_curve_phi_supply}}
    \hfill
    \subfigure[Perturbing $\eta$]{\includestandalone[width=0.325\linewidth]{figure/fig_sa_curve_eta_supply}}
    
    \subfigure[Perturbing $\epsilon_1$]{\includestandalone[width=0.325\linewidth]{figure/fig_sa_curve_epsilon1_supply}}
    \subfigure[Perturbing $\epsilon_2$]{\includestandalone[width=0.325\linewidth]{figure/fig_sa_curve_epsilon2_supply}}
    
    \caption{The curve of $N(K)$ under distinct value of $\lambda_0$, $N_0$, $M_0$, $t_c$, $\alpha$, $\beta$, $\theta$, $\phi$, $\eta$, $\epsilon_1$, and $\epsilon_2$, respectively.}
    \label{fig:sa_valet_curve_supply}
\end{figure}

\begin{figure}[ht]
    \centering
    \subfigure[Perturbing $\lambda_0$]{\includestandalone[width=0.325\linewidth]{figure/fig_sa_curve_lambda0_demand}}
    \hfill
    \subfigure[Perturbing $N_0$]{\includestandalone[width=0.325\linewidth]{figure/fig_sa_curve_N0_demand}}
    \hfill
    \subfigure[Perturbing $M_0$]{\includestandalone[width=0.325\linewidth]{figure/fig_sa_curve_M0_demand}}

    \subfigure[Perturbing $t_c$]{\includestandalone[width=0.325\linewidth]{figure/fig_sa_curve_tc_demand}}
    \hfill
    \subfigure[Perturbing $\alpha$]{\includestandalone[width=0.325\linewidth]{figure/fig_sa_curve_alpha_demand}}    
    \hfill
    \subfigure[Perturbing $\beta$]{\includestandalone[width=0.325\linewidth]{figure/fig_sa_curve_beta_demand}}
    
    \subfigure[Perturbing $\theta$]{\includestandalone[width=0.325\linewidth]{figure/fig_sa_curve_theta_demand}}
    \hfill
    \subfigure[Perturbing $\phi$]{\includestandalone[width=0.325\linewidth]{figure/fig_sa_curve_phi_demand}}
    \hfill
    \subfigure[Perturbing $\eta$]{\includestandalone[width=0.325\linewidth]{figure/fig_sa_curve_eta_demand}}
    
    \subfigure[Perturbing $\epsilon_1$]{\includestandalone[width=0.325\linewidth]{figure/fig_sa_curve_epsilon1_demand}}
    \subfigure[Perturbing $\epsilon_2$]{\includestandalone[width=0.325\linewidth]{figure/fig_sa_curve_epsilon2_demand}}
    
    \caption{The curve of $\lambda(K)$ under distinct value of $\lambda_0$, $N_0$, $M_0$, $t_c$, $\alpha$, $\beta$, $\theta$, $\phi$, $\eta$, $\epsilon_1$, and $\epsilon_2$, respectively.}
    \label{fig:sa_valet_curve_demand}
\end{figure}

\begin{figure}[ht]
    \centering
    \subfigure[Perturbing $\lambda_0$]{
%
%
\definecolor{mycolor1}{rgb}{0.00000,0.44700,0.74100}%
\definecolor{mycolor2}{rgb}{0.85000,0.32500,0.09800}%
\begin{tikzpicture}

\begin{axis}[%
width=1.694in,
height=1.03in,
at={(0.531in,0.457in)},
scale only axis,
xmin=0.5,
xmax=1.5,
xtick={0.5,0.75,1,1.25,1.5},
xticklabels={0.5, 0.75, 1.0, 1.25, 1.5},
xlabel style={font=\color{white!15!black}},
xlabel={Scale of $\lambda_0$},
ymin=20,
ymax=101,
ytick={20, 60, 100},
yticklabels={20, 60, 100},
ylabel style={font=\color{white!15!black}},
ylabel={Optimal Density},
axis background/.style={fill=white},
legend style={nodes={scale=0.8, transform shape}, at={(1, 1)}, anchor=north east, legend cell align=left, align=left, draw=white!15!black}
]
\addplot [color=black, line width=1.0pt]
  table[row sep=crcr]{%
0.5	99.7\\
0.55	99.5\\
0.6	99.3\\
0.65	99.1\\
0.7	98.9\\
0.75	98.7\\
0.8	98.5\\
0.85	98.3\\
0.9	98.2\\
0.95	98\\
1	97.9\\
1.05	97.8\\
1.1	97.6\\
1.15	97.5\\
1.2	97.4\\
1.25	97.3\\
1.3	97.2\\
1.35	97.1\\
1.4	97\\
1.45	96.9\\
1.5	96.8\\
};
\addlegendentry{$K^*_\Pi$}

\addplot [color=mycolor1, line width=1.0pt]
  table[row sep=crcr]{%
0.5	75.2\\
0.55	72\\
0.6	69.4\\
0.65	67.2\\
0.7	65.2\\
0.75	63.5\\
0.8	62\\
0.85	60.7\\
0.9	59.5\\
0.95	58.3\\
1	57.3\\
1.05	56.4\\
1.1	55.5\\
1.15	54.7\\
1.2	53.9\\
1.25	53.2\\
1.3	52.6\\
1.35	52\\
1.4	51.4\\
1.45	50.8\\
1.5	50.3\\
};
\addlegendentry{$K^*_\lambda$}

\addplot [color=mycolor2, line width=1.0pt]
  table[row sep=crcr]{%
0.5	50\\
0.55	47.7\\
0.6	45.8\\
0.65	44.2\\
0.7	42.8\\
0.75	41.6\\
0.8	40.5\\
0.85	39.5\\
0.9	38.6\\
0.95	37.8\\
1	37.1\\
1.05	36.4\\
1.1	35.8\\
1.15	35.2\\
1.2	34.7\\
1.25	34.2\\
1.3	33.7\\
1.35	33.3\\
1.4	32.9\\
1.45	32.5\\
1.5	32.1\\
};
\addlegendentry{$K^*_N$}

\end{axis}

\end{tikzpicture}%
}
    \hfill
    \subfigure[Perturbing $N_0$]{
%
%
\definecolor{mycolor1}{rgb}{0.00000,0.44700,0.74100}%
\definecolor{mycolor2}{rgb}{0.85000,0.32500,0.09800}%
\begin{tikzpicture}

\begin{axis}[%
width=1.694in,
height=1.03in,
at={(0.531in,0.457in)},
scale only axis,
xmin=0.5,
xmax=1.5,
xtick={0.5,0.75,1,1.25,1.5},
xticklabels={0.5, 0.75, 1.0, 1.25, 1.5},
xlabel style={font=\color{white!15!black}},
xlabel={Scale of $N_0$},
ymin=20,
ymax=106,
ytick={20, 60, 100},
yticklabels={20, 60, 100},
ylabel style={font=\color{white!15!black}},
ylabel={Optimal Density},
axis background/.style={fill=white},
legend style={nodes={scale=0.8, transform shape}, at={(1, 1)}, anchor=north east, legend cell align=left, align=left, draw=white!15!black}
]
\addplot [color=black, line width=1.0pt]
  table[row sep=crcr]{%
0.5	104.2\\
0.55	103.4\\
0.6	102.6\\
0.65	101.9\\
0.7	101.2\\
0.75	100.5\\
0.8	100\\
0.85	99.4\\
0.9	98.9\\
0.95	98.4\\
1	97.9\\
1.05	97.4\\
1.1	97\\
1.15	96.6\\
1.2	96.2\\
1.25	95.8\\
1.3	95.5\\
1.35	95.1\\
1.4	94.8\\
1.45	94.4\\
1.5	94.1\\
};
\addlegendentry{$K^*_\Pi$}

\addplot [color=mycolor1, line width=1.0pt]
  table[row sep=crcr]{%
0.5	70.3\\
0.55	68.4\\
0.6	66.7\\
0.65	65.1\\
0.7	63.7\\
0.75	62.4\\
0.8	61.2\\
0.85	60.2\\
0.9	59.1\\
0.95	58.2\\
1	57.3\\
1.05	56.5\\
1.1	55.7\\
1.15	55\\
1.2	54.3\\
1.25	53.6\\
1.3	53\\
1.35	52.4\\
1.4	51.8\\
1.45	51.3\\
1.5	50.7\\
};
\addlegendentry{$K^*_\lambda$}

\addplot [color=mycolor2, line width=1.0pt]
  table[row sep=crcr]{%
0.5	46.5\\
0.55	45.1\\
0.6	43.8\\
0.65	42.7\\
0.7	41.7\\
0.75	40.8\\
0.8	39.9\\
0.85	39.1\\
0.9	38.4\\
0.95	37.7\\
1	37.1\\
1.05	36.5\\
1.1	36\\
1.15	35.4\\
1.2	34.9\\
1.25	34.5\\
1.3	34\\
1.35	33.6\\
1.4	33.2\\
1.45	32.8\\
1.5	32.5\\
};
\addlegendentry{$K^*_N$}

\end{axis}

\end{tikzpicture}%
}
    \hfill
    \subfigure[Perturbing $M_0$]{
%
%
\definecolor{mycolor1}{rgb}{0.00000,0.44700,0.74100}%
\definecolor{mycolor2}{rgb}{0.85000,0.32500,0.09800}%
\begin{tikzpicture}

\begin{axis}[%
width=1.694in,
height=1.03in,
at={(0.531in,0.457in)},
scale only axis,
xmin=0.5,
xmax=1.5,
xtick={0.5,0.75,1,1.25,1.5},
xticklabels={0.5, 0.75, 1.0, 1.25, 1.5},
xlabel style={font=\color{white!15!black}},
xlabel={Scale of $M_0$},
ymin=0,
ymax=150,
ytick={0, 50, 100, 150},
yticklabels={0, 50, 100, 150},
ylabel style={font=\color{white!15!black}},
ylabel={Optimal Density},
axis background/.style={fill=white},
legend style={nodes={scale=0.8, transform shape}, at={(0, 1)}, anchor=north west, legend cell align=left, align=left, draw=white!15!black}
]
\addplot [color=black, line width=1.0pt]
  table[row sep=crcr]{%
0.5	59.7\\
0.55	64\\
0.6	68.1\\
0.65	72.1\\
0.7	76\\
0.75	79.8\\
0.8	83.6\\
0.85	87.2\\
0.9	90.8\\
0.95	94.4\\
1	97.9\\
1.05	101.3\\
1.1	104.7\\
1.15	108.1\\
1.2	111.4\\
1.25	114.7\\
1.3	117.9\\
1.35	121.1\\
1.4	124.3\\
1.45	127.4\\
1.5	130.5\\
};
\addlegendentry{$K^*_\Pi$}

\addplot [color=mycolor1, line width=1.0pt]
  table[row sep=crcr]{%
0.5	45.4\\
0.55	46.7\\
0.6	48.1\\
0.65	49.3\\
0.7	50.6\\
0.75	51.8\\
0.8	52.9\\
0.85	54.1\\
0.9	55.2\\
0.95	56.2\\
1	57.3\\
1.05	58.4\\
1.1	59.4\\
1.15	60.4\\
1.2	61.4\\
1.25	62.4\\
1.3	63.4\\
1.35	64.4\\
1.4	65.3\\
1.45	66.3\\
1.5	67.2\\
};
\addlegendentry{$K^*_\lambda$}

\addplot [color=mycolor2, line width=1.0pt]
  table[row sep=crcr]{%
0.5	30.3\\
0.55	31.1\\
0.6	31.8\\
0.65	32.5\\
0.7	33.2\\
0.75	33.9\\
0.8	34.6\\
0.85	35.2\\
0.9	35.9\\
0.95	36.5\\
1	37.1\\
1.05	37.7\\
1.1	38.3\\
1.15	38.9\\
1.2	39.5\\
1.25	40.1\\
1.3	40.6\\
1.35	41.2\\
1.4	41.8\\
1.45	42.3\\
1.5	42.9\\
};
\addlegendentry{$K^*_N$}

\end{axis}

\end{tikzpicture}%
}

    \subfigure[Perturbing $t_c$]{
%
%
\definecolor{mycolor1}{rgb}{0.00000,0.44700,0.74100}%
\definecolor{mycolor2}{rgb}{0.85000,0.32500,0.09800}%
\begin{tikzpicture}

\begin{axis}[%
width=1.694in,
height=1.03in,
at={(0.531in,0.457in)},
scale only axis,
xmin=0.5,
xmax=1.5,
xtick={0.5,0.75,1,1.25,1.5},
xticklabels={0.5, 0.75, 1.0, 1.25, 1.5},
xlabel style={font=\color{white!15!black}},
xlabel={Scale of $t_c$},
ymin=30,
ymax=157,
ylabel style={font=\color{white!15!black}},
ylabel={Optimal Density},
axis background/.style={fill=white},
legend style={nodes={scale=0.8, transform shape}, at={(1, 1)}, anchor=north east, legend cell align=left, align=left, draw=white!15!black}
]
\addplot [color=black, line width=1.0pt]
  table[row sep=crcr]{%
0.5	154.3\\
0.55	144.9\\
0.6	136.8\\
0.65	129.8\\
0.7	123.7\\
0.75	118.2\\
0.8	113.3\\
0.85	108.9\\
0.9	104.9\\
0.95	101.2\\
1	97.9\\
1.05	94.8\\
1.1	92\\
1.15	89.3\\
1.2	86.9\\
1.25	84.6\\
1.3	82.5\\
1.35	80.4\\
1.4	78.6\\
1.45	76.8\\
1.5	75.1\\
};
\addlegendentry{$K^*_\Pi$}

\addplot [color=mycolor1, line width=1.0pt]
  table[row sep=crcr]{%
0.5	80.4\\
0.55	76.4\\
0.6	73\\
0.65	70.1\\
0.7	67.7\\
0.75	65.4\\
0.8	63.4\\
0.85	61.6\\
0.9	60.1\\
0.95	58.6\\
1	57.3\\
1.05	56.1\\
1.1	55\\
1.15	54\\
1.2	53\\
1.25	52.1\\
1.3	51.3\\
1.35	50.5\\
1.4	49.8\\
1.45	49.1\\
1.5	48.4\\
};
\addlegendentry{$K^*_\lambda$}

\addplot [color=mycolor2, line width=1.0pt]
  table[row sep=crcr]{%
0.5	49\\
0.55	47\\
0.6	45.2\\
0.65	43.7\\
0.7	42.4\\
0.75	41.3\\
0.8	40.3\\
0.85	39.4\\
0.9	38.5\\
0.95	37.8\\
1	37.1\\
1.05	36.5\\
1.1	35.9\\
1.15	35.3\\
1.2	34.8\\
1.25	34.4\\
1.3	33.9\\
1.35	33.5\\
1.4	33.1\\
1.45	32.8\\
1.5	32.4\\
};
\addlegendentry{$K^*_N$}

\end{axis}

\end{tikzpicture}%
}
    \hfill
    \subfigure[Perturbing $\alpha$]{
%
%
\definecolor{mycolor1}{rgb}{0.00000,0.44700,0.74100}%
\definecolor{mycolor2}{rgb}{0.85000,0.32500,0.09800}%
\begin{tikzpicture}

\begin{axis}[%
width=1.694in,
height=1.03in,
at={(0.531in,0.457in)},
scale only axis,
xmin=0.5,
xmax=1.5,
xtick={0.5,0.75,1,1.25,1.5},
xticklabels={0.5, 0.75, 1.0, 1.25, 1.5},
xlabel style={font=\color{white!15!black}},
xlabel={Scale of $\alpha$},
ymin=20,
ymax=101,
ytick={20, 60, 100},
yticklabels={20, 60, 100},
ylabel style={font=\color{white!15!black}},
ylabel={Optimal Density},
axis background/.style={fill=white},
legend style={nodes={scale=0.8, transform shape}, at={(0, 1)}, anchor=north west, legend cell align=left, align=left, draw=white!15!black}
]
\addplot [color=black, line width=1.0pt]
  table[row sep=crcr]{%
0.5	97.2\\
0.55	97.3\\
0.6	97.4\\
0.65	97.4\\
0.7	97.5\\
0.75	97.6\\
0.8	97.6\\
0.85	97.7\\
0.9	97.8\\
0.95	97.8\\
1	97.9\\
1.05	98\\
1.1	98\\
1.15	98.1\\
1.2	98.2\\
1.25	98.2\\
1.3	98.3\\
1.35	98.4\\
1.4	98.4\\
1.45	98.5\\
1.5	98.6\\
};
\addlegendentry{$K^*_\Pi$}

\addplot [color=mycolor1, line width=1.0pt]
  table[row sep=crcr]{%
0.5	53\\
0.55	53.4\\
0.6	53.8\\
0.65	54.2\\
0.7	54.6\\
0.75	55.1\\
0.8	55.5\\
0.85	56\\
0.9	56.4\\
0.95	56.9\\
1	57.3\\
1.05	57.8\\
1.1	58.2\\
1.15	58.7\\
1.2	59.1\\
1.25	59.6\\
1.3	60.1\\
1.35	60.6\\
1.4	61.1\\
1.45	61.5\\
1.5	62\\
};
\addlegendentry{$K^*_\lambda$}

\addplot [color=mycolor2, line width=1.0pt]
  table[row sep=crcr]{%
0.5	33.9\\
0.55	34.2\\
0.6	34.5\\
0.65	34.8\\
0.7	35.1\\
0.75	35.4\\
0.8	35.8\\
0.85	36.1\\
0.9	36.4\\
0.95	36.8\\
1	37.1\\
1.05	37.4\\
1.1	37.8\\
1.15	38.1\\
1.2	38.5\\
1.25	38.8\\
1.3	39.2\\
1.35	39.6\\
1.4	39.9\\
1.45	40.3\\
1.5	40.7\\
};
\addlegendentry{$K^*_N$}

\end{axis}

\end{tikzpicture}%
}    
    \hfill
    \subfigure[Perturbing $\beta$]{
%
%
\definecolor{mycolor1}{rgb}{0.00000,0.44700,0.74100}%
\definecolor{mycolor2}{rgb}{0.85000,0.32500,0.09800}%
\begin{tikzpicture}

\begin{axis}[%
width=1.694in,
height=1.03in,
at={(0.531in,0.457in)},
scale only axis,
xmin=0.5,
xmax=1.5,
xtick={0.5,0.75,1,1.25,1.5},
xticklabels={0.5, 0.75, 1.0, 1.25, 1.5},
xlabel style={font=\color{white!15!black}},
xlabel={Scale of $\beta$},
ymin=20,
ymax=101,
ytick={20, 60, 100},
yticklabels={20, 60, 100},
ylabel style={font=\color{white!15!black}},
ylabel={Optimal Density},
axis background/.style={fill=white},
legend style={nodes={scale=0.8, transform shape}, at={(0, 1)}, anchor=north west, legend cell align=left, align=left, draw=white!15!black}
]
\addplot [color=black, line width=1.0pt]
  table[row sep=crcr]{%
0.5	99.4\\
0.55	99.1\\
0.6	98.8\\
0.65	98.6\\
0.7	98.4\\
0.75	98.2\\
0.8	98.1\\
0.85	98\\
0.9	98\\
0.95	97.9\\
1	97.9\\
1.05	97.9\\
1.1	97.9\\
1.15	97.9\\
1.2	98\\
1.25	98\\
1.3	98.1\\
1.35	98.2\\
1.4	98.3\\
1.45	98.4\\
1.5	98.5\\
};
\addlegendentry{$K^*_\Pi$}

\addplot [color=mycolor1, line width=1.0pt]
  table[row sep=crcr]{%
0.5	55.7\\
0.55	55.7\\
0.6	55.8\\
0.65	55.9\\
0.7	56\\
0.75	56.2\\
0.8	56.4\\
0.85	56.6\\
0.9	56.8\\
0.95	57.1\\
1	57.3\\
1.05	57.6\\
1.1	57.9\\
1.15	58.2\\
1.2	58.5\\
1.25	58.8\\
1.3	59.1\\
1.35	59.5\\
1.4	59.8\\
1.45	60.2\\
1.5	60.5\\
};
\addlegendentry{$K^*_\lambda$}

\addplot [color=mycolor2, line width=1.0pt]
  table[row sep=crcr]{%
0.5	33.5\\
0.55	33.8\\
0.6	34.2\\
0.65	34.5\\
0.7	34.9\\
0.75	35.3\\
0.8	35.6\\
0.85	36\\
0.9	36.4\\
0.95	36.7\\
1	37.1\\
1.05	37.5\\
1.1	37.9\\
1.15	38.2\\
1.2	38.6\\
1.25	39\\
1.3	39.4\\
1.35	39.8\\
1.4	40.2\\
1.45	40.6\\
1.5	41\\
};
\addlegendentry{$K^*_N$}

\end{axis}
\end{tikzpicture}%
}
    
    \subfigure[Perturbing $\theta$]{
%
%
\definecolor{mycolor1}{rgb}{0.00000,0.44700,0.74100}%
\definecolor{mycolor2}{rgb}{0.85000,0.32500,0.09800}%
\begin{tikzpicture}

\begin{axis}[%
width=1.694in,
height=1.03in,
at={(0.531in,0.457in)},
scale only axis,
xmin=0.5,
xmax=1.5,
xtick={0.5,0.75,1,1.25,1.5},
xticklabels={0.5, 0.75, 1.0, 1.25, 1.5},
xlabel style={font=\color{white!15!black}},
xlabel={Scale of $\theta$},
ymin=0,
ymax=150,
ytick={0, 50, 100, 150},
yticklabels={0, 50, 100, 150},
ylabel style={font=\color{white!15!black}},
ylabel={Optimal Density},
axis background/.style={fill=white},
legend style={nodes={scale=0.8, transform shape}, at={(0, 1)}, anchor=north west, legend cell align=left, align=left, draw=white!15!black}
]
\addplot [color=black, line width=1.0pt]
  table[row sep=crcr]{%
0.5	57.3\\
0.55	61.8\\
0.6	66.1\\
0.65	70.4\\
0.7	74.6\\
0.75	78.7\\
0.8	82.7\\
0.85	86.6\\
0.9	90.4\\
0.95	94.2\\
1	97.9\\
1.05	101.5\\
1.1	105.1\\
1.15	108.6\\
1.2	112\\
1.25	115.3\\
1.3	118.6\\
1.35	121.8\\
1.4	124.9\\
1.45	128\\
1.5	130.9\\
};
\addlegendentry{$K^*_\Pi$}

\addplot [color=mycolor1, line width=1.0pt]
  table[row sep=crcr]{%
0.5	14.9\\
0.55	18\\
0.6	21.3\\
0.65	24.9\\
0.7	28.8\\
0.75	32.9\\
0.8	37.3\\
0.85	41.9\\
0.9	46.8\\
0.95	51.9\\
1	57.3\\
1.05	62.9\\
1.1	68.8\\
1.15	74.9\\
1.2	81.2\\
1.25	87.7\\
1.3	94.5\\
1.35	101.5\\
1.4	108.7\\
1.45	116.2\\
1.5	123.8\\
};
\addlegendentry{$K^*_\lambda$}

\addplot [color=mycolor2, line width=1.0pt]
  table[row sep=crcr]{%
0.5	9.2\\
0.55	11.1\\
0.6	13.2\\
0.65	15.6\\
0.7	18.1\\
0.75	20.8\\
0.8	23.7\\
0.85	26.7\\
0.9	30\\
0.95	33.5\\
1	37.1\\
1.05	40.9\\
1.1	45\\
1.15	49.2\\
1.2	53.6\\
1.25	58.2\\
1.3	62.9\\
1.35	67.9\\
1.4	73.1\\
1.45	78.4\\
1.5	83.9\\
};
\addlegendentry{$K^*_N$}

\end{axis}

\end{tikzpicture}%
}
    \hfill
    \subfigure[Perturbing $\phi$]{
%
%
\definecolor{mycolor1}{rgb}{0.00000,0.44700,0.74100}%
\definecolor{mycolor2}{rgb}{0.85000,0.32500,0.09800}%
\begin{tikzpicture}

\begin{axis}[%
width=1.694in,
height=1.03in,
at={(0.531in,0.457in)},
scale only axis,
xmin=0.5,
xmax=1.5,
xtick={0.5,0.75,1,1.25,1.5},
xticklabels={0.5, 0.75, 1.0, 1.25, 1.5},
xlabel style={font=\color{white!15!black}},
xlabel={Scale of $\phi$},
ymin=20,
ymax=105,
ytick={20, 60, 100},
yticklabels={20, 60, 100},
ylabel style={font=\color{white!15!black}},
ylabel={Optimal Density},
axis background/.style={fill=white},
legend style={nodes={scale=0.8, transform shape}, at={(0, 1)}, anchor=north west, legend cell align=left, align=left, draw=white!15!black}
]
\addplot [color=black, line width=1.0pt]
  table[row sep=crcr]{%
0.5	93.6\\
0.55	94.1\\
0.6	94.5\\
0.65	95\\
0.7	95.5\\
0.75	95.9\\
0.8	96.3\\
0.85	96.7\\
0.9	97.1\\
0.95	97.5\\
1	97.9\\
1.05	98.3\\
1.1	98.6\\
1.15	99\\
1.2	99.3\\
1.25	99.7\\
1.3	100\\
1.35	100.3\\
1.4	100.6\\
1.45	100.9\\
1.5	101.1\\
};
\addlegendentry{$K^*_\Pi$}

\addplot [color=mycolor1, line width=1.0pt]
  table[row sep=crcr]{%
0.5	41.7\\
0.55	43.1\\
0.6	44.5\\
0.65	46\\
0.7	47.5\\
0.75	49\\
0.8	50.6\\
0.85	52.2\\
0.9	53.8\\
0.95	55.5\\
1	57.3\\
1.05	59.1\\
1.1	61\\
1.15	63\\
1.2	65\\
1.25	67\\
1.3	69.2\\
1.35	71.5\\
1.4	73.8\\
1.45	76.2\\
1.5	78.7\\
};
\addlegendentry{$K^*_\lambda$}

\addplot [color=mycolor2, line width=1.0pt]
  table[row sep=crcr]{%
0.5	25.6\\
0.55	26.6\\
0.6	27.6\\
0.65	28.7\\
0.7	29.8\\
0.75	30.9\\
0.8	32\\
0.85	33.2\\
0.9	34.5\\
0.95	35.8\\
1	37.1\\
1.05	38.5\\
1.1	39.9\\
1.15	41.5\\
1.2	43\\
1.25	44.7\\
1.3	46.4\\
1.35	48.2\\
1.4	50.1\\
1.45	52.1\\
1.5	54.2\\
};
\addlegendentry{$K^*_N$}

\end{axis}

\end{tikzpicture}%
}
    \hfill
    \subfigure[Perturbing $\eta$]{
%
%
\definecolor{mycolor1}{rgb}{0.00000,0.44700,0.74100}%
\definecolor{mycolor2}{rgb}{0.85000,0.32500,0.09800}%
\begin{tikzpicture}

\begin{axis}[%
width=1.694in,
height=1.03in,
at={(0.531in,0.457in)},
scale only axis,
xmin=0.5,
xmax=1.5,
xtick={0.5,0.75,1,1.25,1.5},
xticklabels={0.5, 0.75, 1.0, 1.25, 1.5},
xlabel style={font=\color{white!15!black}},
xlabel={Scale of $\eta$},
ymin=0,
ymax=110,
ytick={0, 50, 100},
yticklabels={0, 50, 100},
ylabel style={font=\color{white!15!black}},
ylabel={Optimal Density},
axis background/.style={fill=white},
legend style={nodes={scale=0.8, transform shape}, at={(1, 0)}, anchor=south east, legend cell align=left, align=left, draw=white!15!black}
]
\addplot [color=black, line width=1.0pt]
  table[row sep=crcr]{%
0.5	63\\
0.55	69.3\\
0.6	74.6\\
0.65	79.1\\
0.7	82.9\\
0.75	86.3\\
0.8	89.2\\
0.85	91.8\\
0.9	94.1\\
0.95	96.1\\
1	97.9\\
1.05	99.5\\
1.1	101\\
1.15	102.4\\
1.2	103.6\\
1.25	104.7\\
1.3	105.8\\
1.35	106.7\\
1.4	107.6\\
1.45	108.4\\
1.5	109.2\\
};
\addlegendentry{$K^*_\Pi$}

\addplot [color=mycolor1, line width=1.0pt]
  table[row sep=crcr]{%
0.5	21.5\\
0.55	25.6\\
0.6	29.6\\
0.65	33.5\\
0.7	37.3\\
0.75	41\\
0.8	44.6\\
0.85	48\\
0.9	51.2\\
0.95	54.3\\
1	57.3\\
1.05	60.1\\
1.1	62.9\\
1.15	65.5\\
1.2	67.9\\
1.25	70.3\\
1.3	72.6\\
1.35	74.8\\
1.4	76.9\\
1.45	78.9\\
1.5	80.8\\
};
\addlegendentry{$K^*_\lambda$}

\addplot [color=mycolor2, line width=1.0pt]
  table[row sep=crcr]{%
0.5	12.6\\
0.55	15.2\\
0.6	17.8\\
0.65	20.5\\
0.7	23\\
0.75	25.6\\
0.8	28\\
0.85	30.4\\
0.9	32.7\\
0.95	34.9\\
1	37.1\\
1.05	39.2\\
1.1	41.2\\
1.15	43.1\\
1.2	45\\
1.25	46.8\\
1.3	48.5\\
1.35	50.2\\
1.4	51.8\\
1.45	53.3\\
1.5	54.8\\
};
\addlegendentry{$K^*_N$}

\end{axis}

\end{tikzpicture}%
}
    
    \subfigure[Perturbing $\epsilon_1$]{
%
%
\definecolor{mycolor1}{rgb}{0.00000,0.44700,0.74100}%
\definecolor{mycolor2}{rgb}{0.85000,0.32500,0.09800}%
\begin{tikzpicture}

\begin{axis}[%
width=1.694in,
height=1.03in,
at={(0.531in,0.457in)},
scale only axis,
xmin=0.5,
xmax=1.5,
xtick={0.5,0.75,1,1.25,1.5},
xticklabels={0.5, 0.75, 1.0, 1.25, 1.5},
xlabel style={font=\color{white!15!black}},
xlabel={Scale of $\epsilon_1$},
ymin=20,
ymax=102,
ytick={20, 60, 100},
yticklabels={20, 60, 100},
ylabel style={font=\color{white!15!black}},
ylabel={Optimal Density},
axis background/.style={fill=white},
legend style={nodes={scale=0.8, transform shape}, at={(0, 1)}, anchor=north west, legend cell align=left, align=left, draw=white!15!black}
]
\addplot [color=black, line width=1.0pt]
  table[row sep=crcr]{%
0.5	97.8\\
0.55	97.8\\
0.6	97.8\\
0.65	97.8\\
0.7	97.9\\
0.75	97.9\\
0.8	97.9\\
0.85	97.9\\
0.9	97.9\\
0.95	97.9\\
1	97.9\\
1.05	97.9\\
1.1	97.9\\
1.15	97.9\\
1.2	97.9\\
1.25	97.9\\
1.3	97.9\\
1.35	97.9\\
1.4	97.9\\
1.45	97.9\\
1.5	97.9\\
};
\addlegendentry{$K^*_\Pi$}

\addplot [color=mycolor1, line width=1.0pt]
  table[row sep=crcr]{%
0.5	56.8\\
0.55	56.8\\
0.6	56.9\\
0.65	57\\
0.7	57\\
0.75	57.1\\
0.8	57.1\\
0.85	57.2\\
0.9	57.2\\
0.95	57.3\\
1	57.3\\
1.05	57.3\\
1.1	57.4\\
1.15	57.4\\
1.2	57.4\\
1.25	57.5\\
1.3	57.5\\
1.35	57.5\\
1.4	57.5\\
1.45	57.5\\
1.5	57.6\\
};
\addlegendentry{$K^*_\lambda$}

\addplot [color=mycolor2, line width=1.0pt]
  table[row sep=crcr]{%
0.5	36.6\\
0.55	36.6\\
0.6	36.7\\
0.65	36.7\\
0.7	36.8\\
0.75	36.9\\
0.8	36.9\\
0.85	37\\
0.9	37\\
0.95	37.1\\
1	37.1\\
1.05	37.1\\
1.1	37.2\\
1.15	37.2\\
1.2	37.3\\
1.25	37.3\\
1.3	37.3\\
1.35	37.4\\
1.4	37.4\\
1.45	37.4\\
1.5	37.4\\
};
\addlegendentry{$K^*_N$}

\end{axis}

\end{tikzpicture}%
}
    \hfill
    \subfigure[Perturbing $\epsilon_2$]{
%
%
\definecolor{mycolor1}{rgb}{0.00000,0.44700,0.74100}%
\definecolor{mycolor2}{rgb}{0.85000,0.32500,0.09800}%
\begin{tikzpicture}

\begin{axis}[%
width=1.694in,
height=1.03in,
at={(0.531in,0.457in)},
scale only axis,
xmin=0.5,
xmax=1.5,
xtick={0.5,0.75,1,1.25,1.5},
xticklabels={0.5, 0.75, 1.0, 1.25, 1.5},
xlabel style={font=\color{white!15!black}},
xlabel={Scale of $\epsilon_2$},
ymin=20,
ymax=102,
ytick={20, 60, 100},
yticklabels={20, 60, 100},
ylabel style={font=\color{white!15!black}},
ylabel={Optimal Density},
axis background/.style={fill=white},
legend style={nodes={scale=0.8, transform shape}, at={(0, 1)}, anchor=north west, legend cell align=left, align=left, draw=white!15!black}
]
\addplot [color=black, line width=1.0pt]
  table[row sep=crcr]{%
0.5	96\\
0.55	96.4\\
0.6	96.7\\
0.65	97\\
0.7	97.2\\
0.75	97.4\\
0.8	97.5\\
0.85	97.6\\
0.9	97.7\\
0.95	97.8\\
1	97.9\\
1.05	98\\
1.1	98\\
1.15	98.1\\
1.2	98.1\\
1.25	98.2\\
1.3	98.2\\
1.35	98.3\\
1.4	98.3\\
1.45	98.3\\
1.5	98.3\\
};
\addlegendentry{$K^*_\Pi$}

\addplot [color=mycolor1, line width=1.0pt]
  table[row sep=crcr]{%
0.5	46.2\\
0.55	48.2\\
0.6	49.9\\
0.65	51.3\\
0.7	52.6\\
0.75	53.6\\
0.8	54.6\\
0.85	55.4\\
0.9	56.1\\
0.95	56.7\\
1	57.3\\
1.05	57.8\\
1.1	58.3\\
1.15	58.7\\
1.2	59\\
1.25	59.3\\
1.3	59.6\\
1.35	59.9\\
1.4	60.2\\
1.45	60.4\\
1.5	60.6\\
};
\addlegendentry{$K^*_\lambda$}

\addplot [color=mycolor2, line width=1.0pt]
  table[row sep=crcr]{%
0.5	28.6\\
0.55	30\\
0.6	31.3\\
0.65	32.4\\
0.7	33.3\\
0.75	34.2\\
0.8	34.9\\
0.85	35.6\\
0.9	36.1\\
0.95	36.6\\
1	37.1\\
1.05	37.5\\
1.1	37.9\\
1.15	38.2\\
1.2	38.5\\
1.25	38.8\\
1.3	39\\
1.35	39.3\\
1.4	39.5\\
1.45	39.7\\
1.5	39.9\\
};
\addlegendentry{$K^*_N$}

\end{axis}

\end{tikzpicture}%
}
    \hfill
    \subfigure[Perturbing $C$]{
%
%
\definecolor{mycolor1}{rgb}{0.00000,0.44700,0.74100}%
\definecolor{mycolor2}{rgb}{0.85000,0.32500,0.09800}%
\begin{tikzpicture}

\begin{axis}[%
width=1.694in,
height=1.03in,
at={(0.531in,0.457in)},
scale only axis,
xmin=0.5,
xmax=3,
xtick={0.5,  1, 1.5,   2, 2.5,   3},
xticklabels={0.5, 1.0 , 1.5, 2.0, 2.5, 3.0},
xlabel style={font=\color{white!15!black}},
xlabel={Scale of $C$},
ymin=33,
ymax=140,
ytick={40,  90, 140},
yticklabels={40,  90, 140},
ylabel style={font=\color{white!15!black}},
ylabel={Optimal Density},
axis background/.style={fill=white},
legend style={nodes={scale=0.8, transform shape}, at={(1, 1)}, anchor=north east, legend cell align=left, align=left, draw=white!15!black}
]
\addplot [color=black, line width=1.0pt]
  table[row sep=crcr]{%
0.5	133.4\\
0.55	128.3\\
0.6	123.6\\
0.65	119.4\\
0.7	115.5\\
0.75	112\\
0.8	108.8\\
0.85	105.8\\
0.9	103\\
0.95	100.3\\
1	97.9\\
1.05	95.6\\
1.1	93.5\\
1.15	91.4\\
1.2	89.5\\
1.25	87.7\\
1.3	86\\
1.35	84.3\\
1.4	82.8\\
1.45	81.3\\
1.5	79.9\\
1.55	78.5\\
1.6	77.2\\
1.65	76\\
1.7	74.8\\
1.75	73.7\\
1.8	72.6\\
1.85	71.5\\
1.9	70.5\\
1.95	69.5\\
2	68.5\\
2.05	67.6\\
2.1	66.7\\
2.15	65.9\\
2.2	65\\
2.25	64.2\\
2.3	63.5\\
2.35	62.7\\
2.4	62\\
2.45	61.3\\
2.5	60.6\\
2.55	59.9\\
2.6	59.2\\
2.65	58.6\\
2.7	58\\
2.75	57.4\\
2.8	56.8\\
2.85	56.2\\
2.9	55.7\\
2.95	55.1\\
3	54.6\\
};
\addlegendentry{$K^*_\Pi$}

\addplot [color=mycolor1, line width=1.0pt]
  table[row sep=crcr]{%
0.5	57.3\\
0.55	57.3\\
0.6	57.3\\
0.65	57.3\\
0.7	57.3\\
0.75	57.3\\
0.8	57.3\\
0.85	57.3\\
0.9	57.3\\
0.95	57.3\\
1	57.3\\
1.05	57.3\\
1.1	57.3\\
1.15	57.3\\
1.2	57.3\\
1.25	57.3\\
1.3	57.3\\
1.35	57.3\\
1.4	57.3\\
1.45	57.3\\
1.5	57.3\\
1.55	57.3\\
1.6	57.3\\
1.65	57.3\\
1.7	57.3\\
1.75	57.3\\
1.8	57.3\\
1.85	57.3\\
1.9	57.3\\
1.95	57.3\\
2	57.3\\
2.05	57.3\\
2.1	57.3\\
2.15	57.3\\
2.2	57.3\\
2.25	57.3\\
2.3	57.3\\
2.35	57.3\\
2.4	57.3\\
2.45	57.3\\
2.5	57.3\\
2.55	57.3\\
2.6	57.3\\
2.65	57.3\\
2.7	57.3\\
2.75	57.3\\
2.8	57.3\\
2.85	57.3\\
2.9	57.3\\
2.95	57.3\\
3	57.3\\
};
\addlegendentry{$K^*_\lambda$}

\addplot [color=mycolor2, line width=1.0pt]
  table[row sep=crcr]{%
0.5	37.1\\
0.55	37.1\\
0.6	37.1\\
0.65	37.1\\
0.7	37.1\\
0.75	37.1\\
0.8	37.1\\
0.85	37.1\\
0.9	37.1\\
0.95	37.1\\
1	37.1\\
1.05	37.1\\
1.1	37.1\\
1.15	37.1\\
1.2	37.1\\
1.25	37.1\\
1.3	37.1\\
1.35	37.1\\
1.4	37.1\\
1.45	37.1\\
1.5	37.1\\
1.55	37.1\\
1.6	37.1\\
1.65	37.1\\
1.7	37.1\\
1.75	37.1\\
1.8	37.1\\
1.85	37.1\\
1.9	37.1\\
1.95	37.1\\
2	37.1\\
2.05	37.1\\
2.1	37.1\\
2.15	37.1\\
2.2	37.1\\
2.25	37.1\\
2.3	37.1\\
2.35	37.1\\
2.4	37.1\\
2.45	37.1\\
2.5	37.1\\
2.55	37.1\\
2.6	37.1\\
2.65	37.1\\
2.7	37.1\\
2.75	37.1\\
2.8	37.1\\
2.85	37.1\\
2.9	37.1\\
2.95	37.1\\
3	37.1\\
};
\addlegendentry{$K^*_N$}

\end{axis}

\end{tikzpicture}%
}    
    
    \caption{Comparison among $K^*_N$, $K^*_\lambda$, and $K^*_\Pi$ under distinct value of $\lambda_0$, $N_0$, $M_0$, $t_c$, $\alpha$, $\beta$, $\theta$, $\phi$, $\eta$, $\epsilon_1$, $\epsilon_2$, $C$, respectively.}
    \label{fig:sa_valet_comparison}
\end{figure}
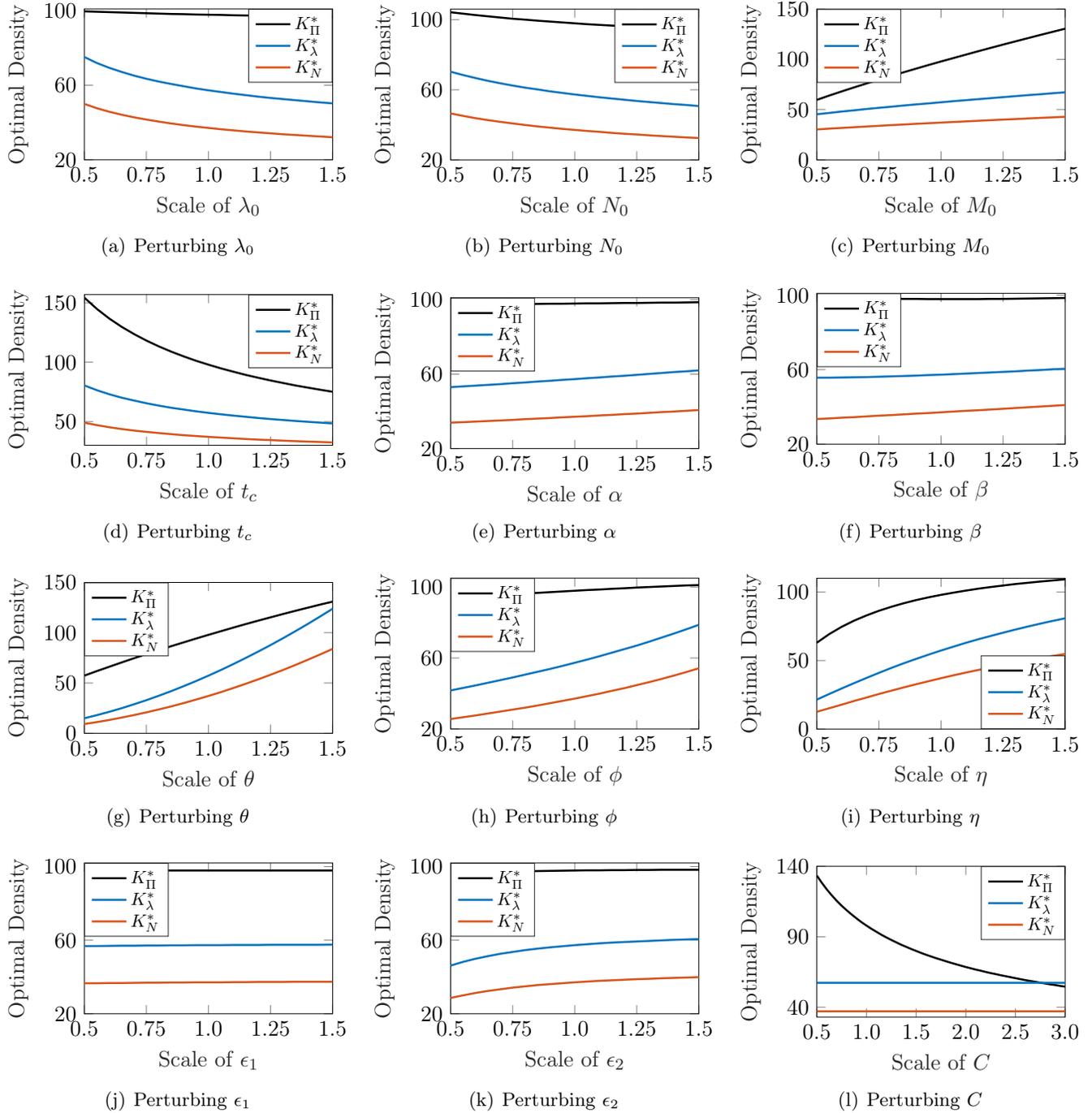

\begin{figure}[t]
    \centering
    \subfigure[]{
%
%
\begin{tikzpicture}

\begin{axis}[%
width=1.694in,
height=1.03in,
at={(0.455in,0.457in)},
scale only axis,
xmin=0.5,
xmax=2.5,
xtick={0.5,1,1.5,2,2.5},
xticklabels={0.5,1.0,1.5,2.0,2.5},
xlabel style={font=\color{white!15!black}},
xlabel={Scale of $r$},
ymin=-0.5,
ymax=30,
ylabel style={font=\color{white!15!black}},
ylabel={Optimal Tax Rate},
axis background/.style={fill=white}
]
\addplot [color=black, line width=1.0pt, forget plot]
  table[row sep=crcr]{%
0.5	26.2\\
0.55	24.6\\
0.6	23.1\\
0.65	21.6\\
0.7	20.3\\
0.75	18.9\\
0.8	17.7\\
0.85	16.5\\
0.9	15.3\\
0.95	14.2\\
1	13.1\\
1.05	12.1\\
1.1	11.1\\
1.15	10.1\\
1.2	9.1\\
1.25	8.2\\
1.3	7.3\\
1.35	6.4\\
1.4	5.6\\
1.45	4.7\\
1.5	3.9\\
1.55	3.1\\
1.6	2.3\\
1.65	1.5\\
1.7	0.7\\
1.75	0\\
1.8	0\\
1.85	0\\
1.9	0\\
1.95	0\\
2	0\\
2.05	0\\
2.1	0\\
2.15	0\\
2.2	0\\
2.25	0\\
2.3	0\\
2.35	0\\
2.4	0\\
2.45	0\\
2.5	0\\
};
\end{axis}

\end{tikzpicture}%
}
    \hfill
    \subfigure[]{
%
%
\begin{tikzpicture}

\begin{axis}[%
width=1.694in,
height=1.03in,
at={(0.531in,0.457in)},
scale only axis,
xmin=0.5,
xmax=2.5,
xtick={0.5,1,1.5,2,2.5},
xticklabels={0.5,1.0,1.5,2.0,2.5},
xlabel style={font=\color{white!15!black}},
xlabel={Scale of $r$},
ymin=50,
ymax=200,
ylabel style={font=\color{white!15!black}},
ylabel={Optimal Density},
axis background/.style={fill=white}
]
\addplot [color=black, line width=1.0pt, forget plot]
  table[row sep=crcr]{%
0.5	193.800000000001\\
0.55	171.800000000001\\
0.6	155.200000000001\\
0.65	141.400000000001\\
0.7	131\\
0.75	121.400000000001\\
0.8	114.200000000001\\
0.85	107.600000000001\\
0.9	101.800000000001\\
0.95	96.8000000000006\\
1	92.2000000000005\\
1.05	88.4000000000004\\
1.1	84.8000000000004\\
1.15	81.6000000000003\\
1.2	78.4000000000003\\
1.25	75.8000000000003\\
1.3	73.2000000000002\\
1.35	71.0000000000002\\
1.4	69.0000000000002\\
1.45	66.8000000000001\\
1.5	65.0000000000001\\
1.55	63.4000000000001\\
1.6	61.6000000000001\\
1.65	60\\
1.7	58.6\\
1.75	57.4\\
1.8	57.4\\
1.85	57.4\\
1.9	57.4\\
1.95	57.4\\
2	57.4\\
2.05	57.4\\
2.1	57.4\\
2.15	57.4\\
2.2	57.4\\
2.25	57.4\\
2.3	57.4\\
2.35	57.4\\
2.4	57.4\\
2.45	57.4\\
2.5	57.4\\
};
\end{axis}
\end{tikzpicture}%
}
    \hfill
    \subfigure[]{
%
%
\begin{tikzpicture}

\begin{axis}[%
width=1.694in,
height=1.03in,
at={(0.531in,0.457in)},
scale only axis,
xmin=0.5,
xmax=2.5,
xtick={0.5,1,1.5,2,2.5},
xticklabels={0.5,1.0,1.5,2.0,2.5},
xlabel style={font=\color{white!15!black}},
xlabel={Scale of $r$},
ymin=490,
ymax=620,
ylabel style={font=\color{white!15!black}},
ylabel={Valet Demand},
axis background/.style={fill=white}
]
\addplot [color=black, line width=1.0pt, forget plot]
  table[row sep=crcr]{%
0.5	613.039959074256\\
0.55	585.439050014434\\
0.6	565.754712785996\\
0.65	551.198013313292\\
0.7	540.134210110773\\
0.75	531.546770857193\\
0.8	524.770570792259\\
0.85	519.353272244269\\
0.9	514.977931428809\\
0.95	511.416266668539\\
1	508.499093770646\\
1.05	506.099885216064\\
1.1	504.121097297751\\
1.15	502.487392405644\\
1.2	501.139235956353\\
1.25	500.029464795448\\
1.3	499.119512809499\\
1.35	498.378814472175\\
1.4	497.781590585762\\
1.45	497.307138576689\\
1.5	496.937775260466\\
1.55	496.658802574424\\
1.6	496.458136490842\\
1.65	496.325121369369\\
1.7	496.251020410893\\
1.75	496.228072094151\\
1.8	496.228072094151\\
1.85	496.228072094151\\
1.9	496.228072094151\\
1.95	496.228072094151\\
2	496.228072094151\\
2.05	496.228072094151\\
2.1	496.228072094151\\
2.15	496.228072094151\\
2.2	496.228072094151\\
2.25	496.228072094151\\
2.3	496.228072094151\\
2.35	496.228072094151\\
2.4	496.228072094151\\
2.45	496.228072094151\\
2.5	496.228072094151\\
};
\end{axis}

\end{tikzpicture}%
}
    \caption{(a) Optimal tax rate $p^*_{t,\lambda}$ under varying $r$; (b) Optimal density $K^*(p^*_{t,\lambda})$ under varying $r$; (c) Maximum demand $\lambda^*(p^*_{t,\lambda})$ under varying $r$.}
    \label{fig:sa_tax_perturb_r}
\end{figure}
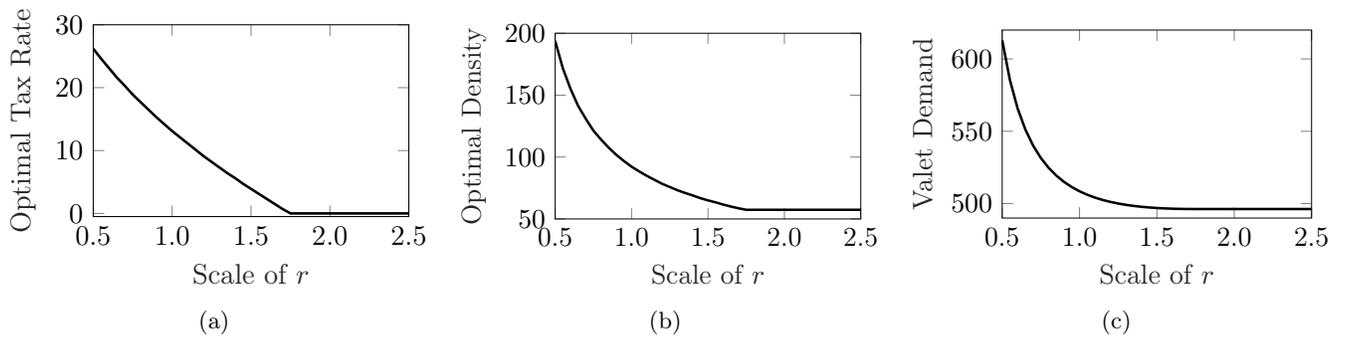

\end{document}